\pdfoutput=1
\documentclass[final,hidelinks,11pt,reqno,nopagebackref]{amsart}
\usepackage[US,bibtex,notcite,notref]{mathieu}
\allowdisplaybreaks[1]
\usepackage{verbatim}

\DeclareMathOperator{\poly}{poly}
\DeclareMathOperator{\id}{id}
\DeclareMathOperator{\pr}{pr}
\DeclareMathOperator{\Der}{Der}
\DeclareMathOperator{\Bott}{{Bott}}
\DeclareMathOperator{\tot}{tot}

\newcommand{\xto}[1]{\xrightarrow{#1}}
\newcommand{\argument}{-}
\newcommand{\XX}{\mathscr{X}}
\newcommand{\XXa}{\XX}

\newcommand{\CC}{\mathbb{C}}
\newcommand{\ZZ}{\mathbb{Z}}
\newcommand{\kk}{\mathbb{K}}
\newcommand{\abs}[1]{\left\vert #1\right\vert}

\newcommand{\lie}[2]{\left[#1,#2\right]}
\newcommand{\pairing}[2]{\left\langle #1,#2\right\rangle}
\newcommand{\cinf}[1]{C^{\infty}(#1)}
\newcommand{\shuffle}[2]{\mathrm{Sh}({#1},{#2})}
\newcommand{\sections}[1]{\Gamma(#1)}
\newcommand{\mfg}{\mathfrak{g}}
\newcommand{\mfk}{\mathfrak{K}}
\newcommand{\frkg}{\mathfrak{g}}
\newcommand{\frkh}{\mathfrak{h}}
\newcommand{\anchor}{\rho}
\newcommand{\cL}{{\mathcal{L}}}
\newcommand{\cM}{\mathcal{M}}
\newcommand{\cF}{\mathcal{F}}

\newcommand{\OmegaAwedgeBs}[2]{\Omega_A^{#1}(\Lambda^{#2} B^\vee)}
\newcommand{\degree}[1]{\abs{#1}}
\newcommand{\totalabs}[1]{\parallel#1\parallel}
\newcommand{\LieDer}{\mathcal{L}}
\newcommand{\set}[1]{\left\{#1\right\}}
\newcommand{\minuspower}[1]{(-1)^{#1}}
\newcommand{\Sbullet}{S}
\newcommand{\quotientmapLB}{\pr_{\moduleB}}
\newcommand{\LADL}{{L}}
\newcommand{\LADA}{{A}}
\newcommand{\baL}[2]{[#1,#2]_{\LADL}}
\newcommand{\baA}[2]{[#1,#2]_{\LADA}}
\newcommand{\baBtoA}[2]{ \beta(#1,#2)}
\newcommand{\baB}[2]{[#1,#2]_{\scriptscriptstyle \moduleB}}
\newcommand{\baABtoB}[2]{\nabla^{\Bott}_{#1}{#2}}
\newcommand{\baBAtoA}[2]{\Delta_{#1}{#2}}
\newcommand{\anchorL}{\rho_{\scriptscriptstyle{L}}}
\newcommand{\anchorA}{\rho_{\scriptscriptstyle{A}}}
\newcommand{\anchorB}{\rho_{\scriptscriptstyle{B}}}
\newcommand{\moduleB}{{B}}
\newcommand{\embeddingi}{\mathrm{i}}
\newcommand{\LADAs}{{A}^\vee}
\newcommand{\moduleBs}{{B}^\vee}
\newcommand{\LADLs}{{L}^\vee}
\newcommand{\baseM}{M}
\newcommand{\splitting}{\mathrm{j}}
\newcommand{\rond}{\circ}
\newcommand{\dA}{d_{\LADA}}
\newcommand{\OmegaAdouble}[2]{{\Omega_{\LADA}^{#1}}{(\Lambda^{#2}\moduleB)}}
\newcommand{\binarybracket}[2]{\{#1,#2\}}
\newcommand{\trinarybracket}[3]{\{#1,#2,#3\}}

\newcommand{\Koszul}[1]{\epsilon(#1)}
\newcommand{\OmegaAsingle}[1]{\Omega_{\LADA}^{#1}}
\newcommand{\algdpbracket}[1]{\lambda_l\bigl(#1\bigr)}
\newcommand{\alphapminusone}{\rho_{l-1}}
\newcommand{\OmegaAB}{\Omega^\bullet_{\LADA}(\moduleB)}
\newcommand{\OmegaAwedgeB}{\Omega^\bullet_{\LADA}(\Lambda^\bullet \moduleB)}
\newcommand{\dL}{d_{\LADL}}
\newcommand{\QL}{Q_L}
\newcommand{\dAB}{d_{\LADA}^{\Bott}}
\newcommand{\structuresheaf}[1]{C^\infty_{#1}}

\newcommand{\qB}{\pr_{\moduleB}}
\newcommand{\pA}{\pr_{\LADA}}
\newcommand{\dbeta}{d_{\beta}}
\newcommand{\dBA}{d_{\moduleB}^{\Delta}}
\newcommand{\anadelta}{\Delta}
\newcommand{\hnatural}{h_\natural}

\newcommand{\sigmanatural}{\sigma_\natural}
\newcommand{\taunatural}{\tau_\natural}

\newcommand{\smiletaunatural}{\breve{\tau}_\natural}
\newcommand{\smilehnatural}{\breve{h}_\natural}

\newcommand{\contraction}[1]{{\iota}_{#1}}
\newcommand{\productinA}{ \cdot_A}
\newcommand{\productinB}{\cdot_B}
\newcommand{\LieDerivative}{{L}}
\newcommand{\barlambda}{\overline{\lambda}}
\newcommand{\cE}{\mathcal{E}}
\newcommand{\diffopoo}[1]{\mathscr{D}^{\leqslant 1}(#1)}
\newcommand{\dual}{^\vee}
\newcommand{\into}{\hookrightarrow}
\newcommand{\onto}{\twoheadrightarrow}

\newcommand{\etendu}[1]{#1_{\natural}}
\newcommand{\dgpoly}{\big(\tot \Omega^{\bullet}_A(\Lambda^\bullet(L/A)),\dAB\big)}
\newcommand{\frakg}{\mathfrak{g}}
\newcommand{\frakh}{\mathfrak{h}}
\newcommand{\KK}{\mathbb{K}}
\newcommand{\RR}{\mathbb{R}}
\newcommand{\XXX}{\mathscr{X}}
\newcommand{\Fedosov}{\cM}

\newcommand{\Cddd}[3]{C^{(#1,#2,#3)}}
\newcommand{\hypercohomology}{\mathbb{H}}

\begin{document}

\title{Shifted derived Poisson manifolds associated with Lie pairs}

\author{Ruggero Bandiera}
\address{Dipartimento di Matematica, Sapienza Università di Roma}
\email{bandiera@mat.uniroma1.it}

\author{Zhuo Chen}
\address{Department of Mathematics, Tsinghua University}
\email{chenzhuo@tsinghua.edu.cn}
\thanks{Research partially supported by NSFC grant 11471179.}

\author{Mathieu Stiénon}
\address{Department of Mathematics, Pennsylvania State University}
\email{stienon@psu.edu}

\author{Ping Xu}
\address{Department of Mathematics, Pennsylvania State University}
\email{ping@math.psu.edu}
\thanks{Research partially supported by NSF grants DMS-1406668 and DMS-1707545}

\begin{abstract}
We study the shifted analogue of the ``Lie--Poisson'' construction for $L_\infty$ algebroids and we prove that any $L_\infty$ algebroid naturally gives rise to shifted derived Poisson manifolds. We also investigate derived Poisson structures from a purely algebraic perspective and, in particular, we establish a homotopy transfer theorem for derived Poisson algebras. 

As an application, we prove that, given a Lie pair $(L,A)$, the space $\operatorname{tot}\Omega^{\bullet}_A(\Lambda^\bullet(L/A))$ admits a degree $(+1)$ derived Poisson algebra structure with the wedge product as associative multiplication and the Chevalley--Eilenberg differential $d_A^{\operatorname{Bott}}:\Omega^{\bullet}_A(\Lambda^\bullet(L/A))\to \Omega^{\bullet +1}_A(\Lambda^\bullet(L/A))$ as unary $L_\infty$ bracket. This degree $(+1)$ derived Poisson algebra structure on $\operatorname{tot}\Omega^{\bullet}_A(\Lambda^\bullet(L/A))$ is unique up to an isomorphism having the identity map as first Taylor coefficient. Consequently, the Chevalley--Eilenberg hypercohomology $\mathbb{H}(\Omega^{\bullet}_A(\Lambda^\bullet(L/A)),d_A^{\operatorname{Bott}})$ admits a canonical Gerstenhaber algebra structure.
\end{abstract}

\maketitle

\tableofcontents

\section{Introduction}

The notion of \emph{Lie pair} is a natural framework encompassing a range of diverse geometric contexts
including complex manifolds, foliations, and $\frakg$-manifolds (that is, manifolds endowed with an action of a Lie algebra $\frakg$).
By a \emph{Lie pair} $(L,A)$, we mean an inclusion $A\hookrightarrow L$ of Lie $\KK$-algebroids over a smooth manifold $M$.
(Throughout the paper, we use the symbol $\KK$ to denote either of
the fields $\RR$ or $\CC$.)
Recall that a \emph{Lie $\KK$-algebroid} is a $\KK$-vector bundle $L\to M$,
whose space of sections is endowed with a Lie bracket $[\argument,\argument]$,
together with a bundle map $\rho: L\to T_M\otimes_{\RR}\KK$ called 
\emph{anchor}
such that $\rho: \sections{L}\to\XX(M)\otimes_{\RR}\KK$ is
a morphism of Lie algebras
and $[X,fY]=f[X,Y]+\big(\rho(X)f\big)Y$ for all $X,Y\in\Gamma(L)$ and $f\in C^\infty(M,\KK)$.
In other words, a $\KK$-vector bundle $L\to M$ is a Lie $\KK$-algebroid if and only if $\sections{L}$ is
a \emph{Lie--Rinehart $\KK$-algebra} \cite{MR0154906} over the commutative ring $C^\infty(M, \KK)$.
A Lie pair over the one-point space $M=\{*\}$ is simply a pair of Lie algebras $(\frakg, \frakh)$ with an inclusion of $\frakh$ into $\frakg$.

Given a Lie pair $(L,A)$, the quotient $L/A$ is naturally an $A$-module: $\nabla^{\Bott}_a b=q([a,l])$,
where $a\in\sections{A}$, $b\in\sections{L/A}$, $q$ denotes the projection $L\onto L/A$, and $l$ is any element of $\sections{L}$ such that $q(l)=b$.
The flat $A$-connection $\nabla^{\Bott}$ on $L/A$ is known as the Bott connection \cite{MR0362335}.

Let $\XXX_A$ and $\XXX_L$ denote the differentiable stacks determined
by the local Lie groupoids integrating the Lie algebroids $A$ and $L$,
respectively. 
The dg algebra $\dgpoly$, where
\[ \tot \Omega^{\bullet}_A(\Lambda^\bullet(L/A))=\bigoplus_{k, l}\Omega^k_A (\Lambda^l (L/A)[-l])[k] ,\]
can be regarded as the space of formal polyvector
fields tangent to the fibers of the differentiable stack 
fibration $\XXX_A\to\XXX_L$. 
For instance, the dg algebra $\big(\tot 
\Omega^{\bullet}_F (\Lambda^\bullet (T_M/T_F)) , d_F^{\Bott} \big)$ associated with the Lie pair $(T_M,T_F)$
encoding a foliation $F$ of a smooth manifold $M$ may be thought of as the space of formal polyvector fields
on the differentiable stack determined by the holonomy groupoid of the foliation $F$.
Therefore, it is natural to ask
whether the dg algebra $\dgpoly$ admits a ``$(+1)$-shifted Lie bracket'' compatible with the wedge product
--- an analogue of the \emph{Schouten bracket} on the polyvector fields of a smooth manifold.
The answer is negative; such a bracket does not exist for arbitrary Lie pairs.
However, it turns out that, for every Lie pair $(L,A)$, there does exist a $(+1)$-shifted $L_\infty$ algebra structure 
on $\tot \Omega^{\bullet}_A(\Lambda^\bullet(L/A))$,
which is compatible with the wedge product in the sense that all its higher brackets satisfy the graded Leibniz rule.
This is what we call a \emph{degree $(+1)$ derived Poisson algebra}, i.e.\ a ``Gerstenhaber algebra up to homotopy.''

The main theorem of the paper can be summarized as follows:

\begin{theorem}\label{thm: main1}
Given any Lie pair $(L,A)$, the space $\tot \Omega^{\bullet}_A (\Lambda^\bullet (L/A))$ admits a structure of degree $(+1)$ derived Poisson
algebra, with the wedge product as associative multiplication and the Chevalley--Eilenberg differential
\[ \dAB: \Omega^{\bullet}_A (\Lambda^\bullet (L/A)) \to\Omega^{\bullet +1}_A (\Lambda^\bullet (L/A)) \] as unary $L_\infty$ bracket 
--- the $A$-module structure on $\Lambda^{\bullet}(L/A)$ is the natural extension of the Bott $A$-connection on $L/A$.
This degree $(+1)$ derived Poisson algebra structure is unique up to an isomorphism of degree $(+1)$ derived Poisson algebras having the identity map as first Taylor coefficient.
\end{theorem}

As an immediate consequence, we obtain the following

\begin{theorem}\label{thm: main2}
Given any Lie pair $(L, A)$, the Chevalley--Eilenberg hypercohomology \[ \mathbb{H} ( \Omega^{\bullet}_A (\Lambda^\bullet (L/A)),\dAB) \]
admits a canonical Gerstenhaber algebra structure.
\end{theorem}

In \cite{BSX:17}, an $L_\infty [1]$ algebra structure on 
$\tot\Omega^{\bullet}_A (\Lambda^\bullet (L/A))$ 
was constructed explicitly via Fedosov dg Lie algebroids --- see Section~\ref{pineapple}.
Its construction relies on the choice of additional geometric data: 
a splitting of the short exact sequence
$0\to\LADA\to\LADL\to L/A\to 0$ 
and a torsion-free $L$-connection $\nabla$ on $L/A$ extending the Bott $A$-connection.
A similar construction for polydifferential operators rather than polyvector fields was 
described in \cite{BSX:17} as well.
Understanding the extent to which the resulting $L_\infty[1]$ algebra structures on 
$\tot \Omega^{\bullet}_A (\Lambda^\bullet (L/A))$ and its polydifferential operator counterpart
depend on the geometric data chosen is an issue that was addressed in~\cite{BSX:17}.

In this paper, we propose a more direct approach to the two-fold problem of existence and uniqueness 
of a structure of derived Poisson algebra of degree $(+1)$ on $\tot \Omega^{\bullet}_A (\Lambda^\bullet (L/A))$:
we describe two new ways of constructing such structures --- one involves $L_\infty$ algebroids while the other 
involves deformations of Dirac structures --- and we prove that all three approaches yield exactly the same 
degree $(+1)$ derived Poisson algebra structure on $\tot \Omega^{\bullet}_A (\Lambda^\bullet (L/A))$ --- see Section~\ref{melon}.
The uniqueness of the degree $(+1)$ derived Poisson algebra structure in Theorem~\ref{thm: main1}
then follows from a standard result of \v{S}evera on deformations of Dirac structures \cite{arXiv:1707.00265,Severa_private_communication}.

Let us briefly recall the construction via deformations of Dirac structures.
The deformation of Dirac structures were investigated about 15 years ago by \v{S}evera, Roytenberg, and many others \cite{arXiv:1707.00265,MR2699145}.
Given a Courant algebroid $E$ of signature $(n,n)$, the deformations of a Dirac structure $D$ in $E$ are governed
by an $L_\infty$ algebra structure on $\Gamma(\Lambda^\bullet D^\vee)$,
which is in fact a degree $(-1)$ derived Poisson algebra unique up to isomorphism \cite{arXiv:1707.00265}. 
Here $\Gamma(\Lambda^\bullet D^\vee) =\bigoplus_l \Gamma(\Lambda^l D^\vee)[-l]$. 
Now, given a Lie pair $(L, A)$, it is well known that $E=L\oplus L^\vee$ is a Courant algebroid of signature $(n,n)$
and $D=A\oplus A^\perp$ is a Dirac structure in $E$ \cite{MR1472888}.
It is easy to see that the isomorphism 
$\Gamma(\Lambda^m D^\vee)\cong\bigoplus_{k+l=m} \Omega^{k}_A (\Lambda^l (L/A))$
identifies the first $L_\infty$ bracket on $\Gamma(\Lambda^\bullet D^\vee)$ with the Chevalley--Eilenberg differential
$\dAB: \Omega^{\bullet}_A (\Lambda^\bullet (L/A))\to\Omega^{\bullet +1}_A (\Lambda^\bullet (L/A))$.
Thus one obtains a degree $(-1)$ derived Poisson algebra structure on the dg algebra
$\big(\bigoplus_{k+l= \bullet} \Omega^{k}_A (\Lambda^l (L/A))[-k-l],
\dAB\big)$.
Shifting the graduation, we obtain a structure of derived
 Poisson algebra of degree $(+1)$ on
$\tot \Omega^{\bullet}_A(\Lambda^\bullet(L/A))=\bigoplus_{k, l}\Omega^k_A (\Lambda^l (L/A))[-k+l]$. 

Next, we proceed to outline the construction of the degree $(+1)$ derived Poisson algebra structure 
on $\tot \Omega^{\bullet}_A (\Lambda^\bullet (L/A))$ via $L_\infty$ algebroids.

An $L_\infty$ algebroid is a vector bundle $\cL\to\cM$
of $\ZZ$-graded manifolds endowed with (1) a sequence $(\lambda_l)_{l\geq 1}$
of maps $\lambda_l: \Lambda^l\sections{\cL}\to\sections{\cL}[2-l]$, called multi-brackets, that determine a structure of $L_\infty$ algebra on $\sections{\cL}$
and (2) a sequence $(\rho_l)_{l\geq 0}$ of bundle maps $\rho_l: \Lambda^l\cL\to T_{\cM}\otimes_{\RR}\KK[1-l]$, called anchor maps,
that determine a morphism of $L_\infty$ algebras from $\sections{\cL}$ to
$\XX(\cM)\otimes_{\RR}\KK$.
The $L_\infty$-brackets $\lambda_l$ and the anchor maps $\rho_l$ must satisfy the usual compatibility condition \cite{MR2103009,MR2840338,MR3277952}.

There exists an equivalent and more compact definition of $L_\infty$ algebroids \`a la Va\u{\i}ntrob via dg manifolds \cite{MR1480150}, which we will 
recall briefly.
A dg manifold is a $\ZZ$-graded manifold $\cM$ together with a homological vector field, i.e.\ vector field $Q\in\XX(\cM)$ of degree $(+1)$ satisfying $Q^2=0$.
An $L_\infty$ algebroid is a vector bundle $\cL\to \cM$ of $\ZZ$-graded manifolds together with a homological vector field $Q$ on $\cL[1]$ 
tangent to the zero section $\cM\subset\cL[1]$ --- see Proposition~\ref{Prop: LinfinityalgebroidandQ}.

It turns out that $L_\infty$ algebroids are closely related to \emph{shifted derived $C^\infty$-Poisson manifolds} in the sense of Pridham \cite{MR3653066}.
Let $\hat{\XX}^\bullet_{\poly}(\cM,n)$ denote the completion of the space of $n$-shifted polyvector fields on $\cM$ --- see Appendix
\ref{Appendix: shiftedpoly} for details.
A $(-k)$-shifted derived Poisson manifold (see Definition \ref{Def: -kshiftedDPM}) can be thought of as a dg manifold $(\cM,Q)$
equipped with a formal series $\pi=\sum_{l= 2}^\infty \pi_l$ of $(k-2)$-shifted polyvector fields,
with $\pi_l\in \hat{\XX}^l_{\poly}(\cM,k-2)$ of degree $(+1)$ in $\hat{\XX}^\bullet_{\poly}(\cM,k-2) [k-1]$,
satisfying the Maurer--Cartan equation $[Q,\pi]+\frac{1}{2}[\pi,\pi]=0$.
The well known ``Lie--Poisson'' construction admits the following analogue in the ``shifted derived'' context.

\begin{theorem}\label{mower}
Let $\cL\to \cM$ be a vector bundle of $\ZZ$-graded manifolds,
and let $k\in\mathbb{Z}$ be a fixed integer.
The following statements are equivalent.
\begin{enumerate}
\item The vector bundle $\cL\to \cM$ is an $L_\infty$ algebroid.
\item The space $\sections{\hat{ \Sbullet }(\cL [k])}$ is a degree $k$ derived Poisson algebra with $l$-th bracket of weight $(1-l)$.
\item The graded manifold $ \cL^\vee [-k]$ is a $(-k)$-shifted derived Poisson
manifold and the weight of the $l$th-bracket on $\cinf{ \cL^\vee[-k] }$ is $(1-l)$.
\end{enumerate}
\end{theorem}

Now, going back to a Lie pair $(L,A)$, it is easily seen that, once a splitting of the short exact $0\to\LADA\to\LADL\to L/A\to 0$ has been chosen,
the graded vector bundle $A[1]\times L/A\to A[1]$ acquires a natural $L_\infty$ algebroid structure.
This $L_\infty$ algebroid was studied by Vitagliano in the special case of a Lie pair corresponding to a foliation \cite{MR3277952}.
Applying Theorem \ref{mower} to this $L_\infty$ algebroid and the integer $k=+1$ yields a degree $(+1)$ derived Poisson algebra structure
on $\tot \Omega^{\bullet}_A (\Lambda^\bullet (L/A))$.

The discussion of the constructions outlined above occupies Sections \ref{onion} and \ref{dandelion}.
Section \ref{cranberry} is devoted to the
study of derived Poisson structures from a purely algebraic perspective
and contains, in particular, the proof of a homotopy transfer result for derived Poisson algebras.

We would like to point out that, in the context of $\ZZ_2$-grading,
``shifted derived Poisson algebras''
were studied by Voronov \cite{MR2163405, MR2223157},
Khudaverdian--Voronov \cite{MR2757715}, and Bruce \cite{MR2840338},
who called them homotopy Poisson algebras and homotopy Schouten algebras, respectively.
Derived Poisson algebras of degree $0$ were also studied
by Oh--Park \cite{MR2180451} and Cattaneo--Felder \cite{MR2304327},
who called them $P_\infty$ algebras.
See also \cite{MR3653066, MR3654355}. 
For recent developments, see \cite{MR3575558}, \cite{arXiv:1808.10049} and \cite{arXiv:1411.6720}. 
References \cite{MR3633027, MR3631929, MR3090103} are also related to the present paper.

\subsection*{Notations}
In this paper, unless specified otherwise, graded means $\ZZ$-graded.
Given a graded vector space $V=\bigoplus_{n\in\ZZ}V^n$, we say that an element $v\in V^n$ has degree $n$ and we write $\degree{v}=n$.
Given a graded vector space $V=\bigoplus_{n\in\ZZ}V^{n}$, the symbol $V[k]$ denotes the graded vector space obtained from $V$
by shifting the graduation according to the rule $(V[k])^{n}=V^{n+k}$.
We write $\degree{v}^{[k]}$ to denote the degree of $v$ when regarded as an element of $V[-k]$.
Therefore, \[ \degree{v}^{[k]}=\degree{v}+k .\]
The dual $V^\vee$ of a graded vector space $V$ is graded according to the rule $(V^\vee)^n=(V^{-n})^\vee$.

Likewise, if $E=\bigoplus_{n\in\ZZ}E^{n}$ is a graded vector bundle over a manifold $M$,
$E[k]$ denotes the graded vector bundle obtained by shifting the graduation of the fibers of $E$ according to the above rule.

Given a vector space $V$, the symbol $\hat{S}(V)$ denotes the $\mathfrak{m}$-adic completion of the symmetric algebra $S(V)$,
where $\mathfrak{m}$ is the ideal of $S(V)$ generated by $V$. Thus, $\hat{\Sbullet}(V)=\prod_{p=0}^{\infty} S^p(V)$.
The symbol $\overline{S}(V)$ denotes the reduced symmetric algebra of $V$, i.e.\ $\overline{S}(V)=\bigoplus_{p=1}^{\infty} S^p(V)$.

The Koszul sign $\Koszul{\sigma;v_1,\cdots,v_p}$ of a permutation $\sigma\in \mathfrak{S}_p$
of $p$ homogeneous vectors $v_1,v_2,\dots,v_{p}$ of a graded vector space $V$ --- which will be abbreviated as $\Koszul{\sigma}$ --- is determined by the relation
\[ v_{\sigma(1)}\odot v_{\sigma(2)}\odot\cdots\odot v_{\sigma(p)} = \Koszul{\sigma;v_1,\cdots,v_p}\ v_1\odot v_2\odot\cdots\odot v_p ,\]
where $\odot$ denotes the multiplication in the symmetric algebra $\Sbullet(V)$.

An $(r,s)$-shuffle is a permutation $\sigma$ of the set $\{1,2,\cdots,r+s\}$ such that $\sigma(1)\le\sigma(2)\le\cdots\le\sigma(r)$
and $\sigma(r+1)\le\sigma(r+2)\le\cdots\le\sigma(r+s)$. We write $\shuffle{r}{s}$ to denote the set of $(r,s)$-shuffles.

\section{Derived Poisson algebras}
\label{cranberry}

\subsection{Derived Poisson algebras}

Let $k\in \ZZ$ be a fixed integer, and $\kk$ a field
of characteristic zero.

Given a $\ZZ$-graded $\kk$-vector space $V=\bigoplus_{k\in\ZZ}V^i$, we denote by $(\overline{S}(V)=\oplus_{n\geq1}V^{\odot n},\Delta)$ the reduced symmetric tensor coalgebra on $V$. This is the cofree cocommutative locally conilpotent coalgebra generated by $V$. In particular, every coderivation $Q:\overline{S}(V)\to\overline{S}(V)$ is completely determined by its corestriction $p\circ Q=:q=:(q_1,\ldots,q_n,\ldots)$, where $p:\overline{S}(V)\to V$ denotes
the canonical projection. The sequence of maps $q_n:V^{\odot n}\to V$,
$n\geq 1$, are called the \emph{Taylor coefficients} of $Q$. Similarly, given another graded space $W$, every morphism of coalgebras $F:\overline{S}(W)\to\overline{S}(V)$ is completely determined by its corestriction $p\circ F=f=(f_1,\ldots,f_n,\ldots)$. 
Again we call the sequence of maps $f_n:W^{\odot n}\to V$, $n\geq 1$, 
the Taylor coefficients of $F$.

\begin{definition}\label{def:derivedpoisson}
A degree $k$ derived Poisson algebra is a $\ZZ$-graded commutative
$\kk$-algebra $A=\bigoplus_{k\in\ZZ}A^i$
together with a degree $(+1)$ coderivation
\[ Q: \overline{S} (A[1-k])\to \overline{S} (A[1-k])[1] \]
of the reduced symmetric tensor coalgebra
$\big(\overline{S}(A[1-k]),\Delta\big)$ satisfying
\begin{enumerate}
\item the cohomological condition $Q\circ Q=0$; and
\item the Leibniz rule
\begin{equation}\label{leibniz}q_n(a_1,\ldots, a_{n-1}, a_na^\prime_n) 
= q_n(a_1,\ldots, a_{n-1}, a_n) a^\prime_n 
+(-1)^\epsilon a_n\, q_n(a_1,\ldots, a_{n-1} , a^\prime_n)
\end{equation}
with $\epsilon={\{(2-n)+(\alpha_1+\cdots+\alpha_{n-1})+k(n-1)\}\alpha_n}$,
for all $n\geq1$ and $a_1\in A^{\alpha_1}$, $a_2\in A^{\alpha_2}$, \dots, $a_n\in A^{\alpha_n}$, $a^\prime_n\in A$. An equivalent expression is 
\begin{equation*}q_n(a_1,\ldots, a_{n-1}, a_na^\prime_n) 
= q_n(a_1,\ldots, a_{n-1}, a_n) a^\prime_n 
+(-1)^{\alpha_n\alpha^\prime_n}\, q_n(a_1,\ldots, a_{n-1} , a^\prime_n)a_n
\end{equation*}
for all $a_1,\ldots,a_{n-1}\in A$, $a_n\in A^{\alpha_n}, a^\prime_n\in A^{\alpha^\prime_n}$.
\end{enumerate}
\end{definition}

To understand the meaning of Equation \eqref{leibniz}, and in particular the sign $(-1)^\epsilon$ of its last term, observe that,
if $a_1\in A^{\alpha_1}, a_2\in A^{\alpha_2},\cdots, a_n\in A^{\alpha_n}$,
or equivalently $a_1\in (A[1-k])^{\alpha_1-1+k}, a_2\in
(A[1-k])^{\alpha_2-1+k}, \cdots, a_n\in (A[1-k])^{\alpha_n-1+k}$,
then we have
\[ q_n(a_1,a_2,\ldots, a_n) \in
(A[1-k])^{1+(\alpha_1-1+k)+(\alpha_2-1+k)+\cdots+(\alpha_n-1+k)} ,\]
or equivalently
\[ q_n(a_1,a_2,\ldots, a_n) \in
A^{(1-k)+1+(\alpha_1-1+k)+(\alpha_2-1+k)+\cdots+(\alpha_n-1+k)} .\]
Therefore,
\[ A\ni x \mapsto q_n(a_1,\ldots, a_{n-1},x) \in A \]
is an operator on $A$ of degree
\[ 1+(\alpha_1-1+k)+\cdots+(\alpha_{n-1}-1+k) \\
= (2-n)+(\alpha_1+\cdots+\alpha_{n-1})+k(n-1). \]
Equation \eqref{leibniz} means that this operator
is a graded derivation on $A$.

An equivalent description of the above definition is the following:

\begin{definition}\label{Defn: kshiftedPinfinityalgebra}
A degree $k$ derived Poisson algebra is a $\ZZ$-graded commutative algebra
$A=\bigoplus_{i\in\ZZ}A^i$ (over a field $\kk$
of characteristic zero) endowed with a family of $\kk$-multilinear maps
$\lambda_n\mathrel{: } A^{\otimes n}\to A$ of degree $k(n-1)+2-n$, ($n=1,2,\cdots$),
defining
an $L_\infty$ algebra structure on $A [-k]$ such that
for all $a_1,\cdots,a_{n-1}\in A$, the map
\[ A\to A, \ \ \ \ a\mapsto\lambda_n({a_1,\cdots,a_{n-1},a}) \]
is a graded derivation.\end{definition} 
We note that the maps $\lambda_n$ and $q_n$ from the previous definitions
are related by \emph{d\'ecalage}, that is,
\[q_n(a_1,\ldots,a_n) = (-1)^{\epsilon}\lambda_n(a_1,\ldots, a_n)\]
for all $n\geq1$ and $a_1\in A^{\alpha_1},\ldots,a_n\in A^{\alpha_n}$,
where $\epsilon = \sum_{i=1}^n(n-i)(\alpha_i+k)$. In particular,
while the $q_n$ can be considered as degree one graded symmetric 
maps $q_n:A[1-k]^{\odot n}\to A[1-k]$, the $\lambda_n$
can be considered as degree $2-n$ graded antisymmetric maps $\lambda_n:A[-k]^{\wedge n}\to A[-k]$. In the sequel, the structure maps $\lambda_n$ are also denoted by $\{\cdots\}_n$, just as the usual Poisson brackets.

A degree $k$ Poisson algebra (see Appendix B) is a degree $k$ derived Poisson algebra, where the only nontrivial bracket
is the binary bracket $\lambda_2=\{\argument,\argument\}$.

\begin{remark}
In the context of $\ZZ_2$-grading,
Definition \ref{Defn: kshiftedPinfinityalgebra}
reduces to homotopy Poisson algebras and homotopy Schouten algebras
studied by Voronov \cite{MR2163405, MR2223157},
Khudaverdian--Voronov \cite{MR2757715},
and Bruce \cite{MR2840338}.

Derived Poisson algebras of degree $0$ were also studied
by Oh--Park \cite{MR2180451} and Cattaneo--Felder \cite{MR2304327},
who called them $P_\infty$ algebras.
\end{remark}

Given a degree $k$ derived Poisson algebra $(A, (\lambda_l)_{l\geq 1})$,	
we have a cochain complex $\lambda_1 : A^\bullet\to A^{\bullet+1}$.
The binary bracket $\lambda_2$ is of degree $k$, and
satisfies the Jacobi identity up to homotopy.
The following is thus immediate.

\begin{proposition}
\label{Prop: kshiftedPinfinitygiveskshiftedPoisson}
If $A $ is a degree $k$ derived Poisson algebra, then the binary bracket $\lambda_2$ induces,
on the cohomology groups $\mathbb{H}(A,\lambda_1)$, a degree $k$ Poisson algebra structure.
\end{proposition}

\subsection{Morphisms of derived Poisson algebras}

\begin{definition}\label{def: morphisms}
Let $(A,Q_A)$ and $(B,Q_B)$ be degree $k$ derived Poisson algebras: here $Q_A$ denotes the corresponding coderivation on $\overline{S}(A[1-k])$, similarly for $Q_B$, see Definition \ref{def:derivedpoisson}. A morphism $f_\infty=(f_1,\ldots,f_n,\ldots): B\to A$ of derived Poisson algebras is a collection of degree $(k-1)(n-1)$ maps $f_n: B^{\otimes n}\to A$, $n\geq1$, such that
\begin{enumerate}
\item If we regard the $f_n$ as degree $0$ maps
$B[1-k]^{\otimes n}\to A[1-k]$ they are graded symmetric, and the unique morphism $F:\overline{S}(B[1-k])\to\overline{S}(A[1-k])$ of coalgebras with corestriction $p\circ F = f_\infty$ satisfies $F\circ Q_B=Q_A\circ F$. 
In other words, this says that the maps $f_n$ induce (after d\'ecalage) an $L_\infty$ morphism from $B[-k]$ to $A[-k]$.
\item The following relation is satisfied for all $n\geq0$ and $x_1,\ldots,x_n,y,z\in B$:
\begin{equation}\label{eqn: derived morphism}
f_{n+1}(x_1,\ldots,x_n,yz) = \sum_{i=0}^{n}\sum_{\sigma\in\operatorname{Sh}(i,n-i)}
\minuspower{\diamond}f_{i+1}(x_{\sigma(1)},\ldots,x_{\sigma(i)},y)f_{n-i+1}(x_{\sigma(i+1)},\ldots,x_{\sigma(n)},z)
,\end{equation}
where
$\diamond=\epsilon(\sigma; x_1, x_2, \cdots , x_n)+\abs{y}((n-i)(k-1)+\abs{x_{\sigma(i+1)}}+\cdots+\abs{x_{\sigma(n)}})$, and the Koszul sign $\epsilon(\sigma; x_1, x_2, \cdots , x_n)$ associated to the permutation $\sigma$ is computed by regarding $x_1,\ldots,x_n$ as elements of $B[1-k]$, that is, of degrees $|x_1|+k-1,\ldots,|x_n|+k-1$. 

\noindent In particular, for $n=0$,
this says that $f_1: B\to A$ is a morphism of associative
algebras, and for $n=1$,
it says that $f_2: B^{\otimes 2}\to A$ is an $f_1$-biderivation.
\end{enumerate}
\end{definition}

\begin{remark}
Since $S(B[1-k])$ is a graded cocommutative coalgebra via the unshuffle coproduct $\Delta$ and $A$ is a graded commutative algebra with product $m_A:A^{\otimes 2}\to A$,
the space $\operatorname{Hom}(S(B[1-k]), A)$ is a graded commutative algebra via the convolution product $f\star g = m_A\circ (f\otimes g)\circ\Delta$. Given $y\in B$, we denote by $f(\ldots,y)$ the map
\begin{eqnarray*}
&& f(\ldots,y)\colon S(B[1-k])\to A,\\
&& x_1\odot\cdots\odot x_n\mapsto (-1)^{(n(k-1)+\abs{x_1}+\cdots+\abs{x_n})\abs{y}}f_{n+1}(x_1,\ldots,x_n,y),
\end{eqnarray*}
where $x_i\in B^{\abs{x_i}}=B[1-k]^{\abs{x_i}+k-1}$, $i=1,\ldots,n$. Then Equation \eqref{eqn: derived morphism} is equivalent to
\[ f(\ldots,yz) = f(\ldots,y)\star f(\ldots,z). \]
In other words, $f$ is a morphism of degree $k$
derived Poisson algebras if and only if
\begin{eqnarray*}
&& (B,m_B)\to(\operatorname{Hom}(S(B[1-k]),A),\star),\\
&&y\mapsto f(\ldots,y)
\end{eqnarray*}
is a morphism of graded algebras.
\end{remark}

The following proposition can be proved by a tedious direct computation, which we omit.

\begin{proposition}
With the above definition of morphisms, degree $k$ derived Poisson algebras
form a category.
\end{proposition}

The following Proposition \ref{prop: morphisms} justifies the previous definition in the framework of deformation theory.
Given a graded algebra $(A,m_A)$, we denote by $L=\operatorname{Coder}(\overline{S}(A[1-k]))$
the graded Lie algebra of coderivations of the reduced
symmetric coalgebra $\overline{S}(A[1-k])$.
From the point of view of deformation theory (cf.\ \cite{MR2130146}), this graded Lie algebra controls the deformations of the trivial $L_\infty$ algebra structure on $A[-k]$.
In other words, the set $\operatorname{MC}(L)$ of Maurer--Cartan elements of $L$, that is, the set of solutions $Q\in L^1$ of the Maurer-Cartan equation $[Q,Q]=0$,
is in bijective correspondence with the set of $L_\infty$ algebra structures on $A[-k]$.
Furthermore, the (formal) exponential group $\exp(L^0)=\{e^R\}_{R\in L^0}$ acts on the set $\operatorname{MC}(L)$ via conjugation $Q\mapsto e^{-R}Qe^R$,
and the corresponding set of orbits parameterizes $L_\infty$ algebra structures on $A[-k]$ up to $L_\infty$ isomorphism.
An easy computation shows that the subspace $M\subset L$, spanned by those
coderivations $Q$ whose Taylor coefficients $p\circ Q=(q_1,\ldots,q_n,\ldots)$ are multiderivations, i.e., they satisfy the Leibniz rule \eqref{leibniz} from Definition \ref{def:derivedpoisson}, is a graded Lie subalgebra.
Moreover, it is clear that the Maurer--Cartan elements in $M$ are precisely the degree $k$ derived Poisson algebra structures on $(A,m_A)$.
In the following Proposition \ref{prop: morphisms}, we show that the (formal) exponential group $\exp(M^0)$ is precisely
the group of coalgebra automorphisms $F: \overline{S}(A[1-k])\to\overline{S}(A[1-k])$ whose Taylor coefficients $p\circ F = (f_1,\ldots,f_n,\ldots)$ satisfy
Equation \eqref{eqn: derived morphism}.

In order to avoid convergence issue, we proceed formally.
We denote by $\mathbb{K}[[t]]$ the algebra of formal power series, by $A[[t]]$
the algebra $A\otimes_{\mathbb{K}}\mathbb{K}[[t]]$ of formal power series with coefficients in $A$,
and by $\overline{S}_{\mathbb{K}[[t]]}(A[[t]][1-k])$ the reduced symmetric
$\mathbb{K}[[t]]$-coalgebra over the $\mathbb{K}[[t]]$-module $A[[t]][1-k]$.
Given a degree zero coderivation $R\in L^0$, we consider the associated formal flow
$e^{tR}: =F^t: \overline{S}_{\mathbb{K}[[t]]}(A[[t]][1-k])\to\overline{S}_{\mathbb{K}[[t]]}(A[[t]][1-k])$.
It is a well defined $\mathbb{K}[[t]]$-linear coalgebra automorphism.

\begin{proposition}\label{prop: morphisms}
Given $R\in L^0,p\circ R=(r_1,\ldots,r_n,\ldots)$ as above, then the Taylor coefficients $f_n^t$ of the formal flow \[ e^{tR}: =F^t,\qquad p\circ F^t=(f^t_1,\ldots,f^t_n,\ldots) \] satisfy Equation \eqref{eqn: derived morphism}
if and only if all Taylor coefficients $r_n: A^{\otimes n}\to A$, $n\geq1$, are multi-derivations, that is, if and only if $R\in M^0$.
\end{proposition}

\begin{proof}
For notational simplicity, in the following computations we abbreviate equations such as Equation \eqref{eqn: derived morphism}
by omitting the $x_1,\ldots,x_n$ arguments, and by writing $\pm_K$
instead of the appropriate Koszul sign
(signs were made precise in Definition \ref{def: morphisms}).
For instance, Equation \eqref{eqn: derived morphism} becomes
\begin{equation*}
f^t_{n+1}(\ldots,yz) = \sum_{i=0}^{n}\sum_{\sigma\in\operatorname{Sh}(i,n-i)}\pm_K f^t_{i+1}(\ldots,y)f^t_{n-i+1}(\ldots,z)
.\end{equation*}
Having set these notations, we proceed with the proof.
One implication is easy: assuming Equation \eqref{eqn: derived morphism} is satisfied, then for all $x_1,\ldots,x_n,y,z\in A$,
\begin{multline*}
r_{n+1}(\ldots,yz) = \frac{d}{dt} f^t_{n+1}(\ldots,yz)_{\big|t=0}
=\sum_{i=0}^n\sum_{\sigma\in\operatorname{Sh}(i,n-i)}\frac{d}{dt}\left( \pm_K f^t_{i+1}(\ldots,y) f^t_{n-i+1}(\ldots,z) \right)_{\big|t=0} = \\
=\sum_{i=0}^n\sum_{\sigma\in\operatorname{Sh}(i,n-i)} \pm_K \frac{d}{dt}
f^t_{i+1}(\ldots,y)_{\big|t=0} f^0_{n-i+1}(\ldots,z) 
\pm_K f^0_{i+1}(\ldots,y)\frac{d}{dt}f^t_{n-i+1}(\ldots,z)_{\big| t=0} = \\ = r_{n+1}(\ldots,y)z \pm_K y\,r_{n+1}(\ldots,z),\qquad\qquad\qquad\qquad\qquad\qquad\qquad
\end{multline*}
since $p\circ F^0=(\id,0,\ldots,0,\ldots)$, i.e. $f_1^0=\id$ and $f^0_k=0$ for $k\geq 2$. This shows that $r_{n+1}$ is a multi-derivation for all $n\geq0$.

We turn to the other implication.
Since $r_1$ is an algebra derivation, $f_1^t=e^{tr_1}: A[[t]]\to A[[t]]$ is an algebra morphism.
Let $n\geq1$,
and fix $x_1,\ldots,x_n,y,z\in A$. We consider the formal
power series
\begin{equation*}
A[[t]]\ni \xi(t): = f^t_{n+1}(\ldots, y z) - \sum_{i=0}^{n}
\sum_{\sigma\in\operatorname{Sh}(i,n-i)}\pm_K
f^t_{i+1}(\ldots,y) f^t_{n-i+1}(\ldots,z)
.\end{equation*}
We need to prove that $\xi(t)=0$. As before,
since $F^0=(\id,0,\cdots,0,\cdots)$, we see that $\xi(0)=0$.
Since $F^t=e^{tR}$, we have $\frac{d}{dt}F^t = RF^t$. Using induction on $n$, we have
\begin{flalign*} &\frac{d}{dt}f^t_{n+1}(\ldots,y z) = &\\ =\: & r_1f^t_{n+1}(\ldots,y z) + \sum_{\stackrel{p\geq2,i_1,\ldots,i_{p-1}\geq1,i_p\geq0}{i_1+\cdots+i_p=n}}
\sum_{\sigma\in\operatorname{Sh}(i_1,\ldots,i_p)}\pm_K \frac{1}{(p-1)!}r_p(f^t_{i_1}(\ldots), \ldots,f^t_{i_p+1}(\ldots,y z)) = &\\
=\: &r_1(\xi(t)) + r_1\left( \sum_{i=0}^{n}\sum_{\sigma\in\operatorname{Sh}(i,n-i)}\pm_K f^t_{i+1}(\ldots,y) f^t_{n-i+1}(\ldots,z) \right)+&\\
&\sum_{\stackrel{p\geq2,i_1,\ldots,i_{p-1}\geq1,i_p,i_{p+1}\geq0}{i_1+\cdots+i_{p+1}=n}}\sum_{\sigma\in\operatorname{Sh}(i_1,\ldots,i_{p+1})}
\pm_K \frac{1}{(p-1)!}r_p(f^t_{i_1}(\ldots), \ldots,f^t_{i_p+1}(\ldots,y)f^t_{i_{p+1}+1}(\ldots, z)) =&\\
=\: & r_1(\xi(t))+\sum_{\stackrel{p,i_1,\ldots,i_{p-1}\geq1,i_p,i_{p+1}\geq0}{i_1+\cdots+i_{p+1}=n}}\sum_{\sigma\in\operatorname{Sh}(i_1,\ldots,i_{p+1})}
\pm_K \frac{1}{(p-1)!}r_p(f^t_{i_1}(\ldots), \ldots,f^t_{i_p+1}(\ldots,y)) f^t_{i_{p+1}+1}(\ldots, z)\, +& \\
&\pm_K \frac{1}{(p-1)!}f^t_{i_p+1}(\ldots,y) r_p(f^t_{i_1}(\ldots), \ldots,f^t_{i_{p+1}+1}(\ldots, z)) = & \\ =\: & r_1(\xi(t)) + 
\sum_{i=0}^n\sum_{\sigma\operatorname{Sh}(i,n-i)}\pm_K \frac{d}{dt}f^t_{i+1}(\ldots,y) f^t_{n-i+1}(\ldots,z) \pm_K f^t_{i+1}(\ldots,y)
\frac{d}{dt}f^t_{n-i+1}(\ldots,z) = &\\ =\:
& r_1(\xi(t)) + \frac{d}{dt}\left( \sum_{i=0}^{n}\sum_{\sigma\in\operatorname{Sh}(i,n-i)}\pm_K f^t_{i+1}(\ldots,y) f^t_{n-i+1}(\ldots,z) \right).&
\end{flalign*}
To sum up, we found that \[\xi'(t) = r_1(\xi(t)), \] thus \[\xi^{(n)}(t) = r_1^n(\xi(t)),\qquad\forall n\geq0 .\]
Expanding in formal Taylor series,
\[ \xi(t) = \sum_{n\geq0}\frac{t^n}{n!}\xi^{(n)}(0)= \sum_{n\geq0}\frac{t^n}{n!}r_1^{n}(\xi(0)) =\sum_{n\geq0}\frac{t^n}{n!}r_1^{n}(0) = 0. \]
\end{proof}

\subsection{Homotopy transfer for derived Poisson algebras}

In this section, we prove a homotopy transfer theorem for derived Poisson algebras. 

We start by recalling the standard Perturbation Lemma. To the best of our knowledge, this result first appeared implicitly in the paper \cite{MR0144348}, and explicitly in \cite{MR0220273, MR0301736}: for the treatment given here, see also \cite{MR1103672,MR1109665, MR1782593}. 

\begin{definition}[\cite{MR0056295}]
Let $(A,d_A)$ and $(B,d_B)$ be cochain complexes.
A contraction $(\sigma,\tau,h)$ of $(A,d_A)$ onto $(B,d_B)$ is the datum of 
cochain maps $\sigma: (A,d_A)\to(B,d_B)$, $\tau: (B,d_B)\to(A,d_A)$
and a contracting homotopy $h: A\to A$ such that
\[ \sigma\tau = \id_B,\qquad hd_A+d_Ah = \tau\sigma-\id_A ,\]
and furthermore
\[ \sigma h=0,\qquad h\tau=0,\qquad h^2=0 .\]
\end{definition}
We denote such a contraction data by
\[ \begin{tikzcd}
(A,d_A) \arrow[loop left, "h"] \arrow[r, shift left, "\sigma"] & (B,d_B) \arrow[l, shift left, "\tau"].
\end{tikzcd} \]

\begin{lemma}[\cite{MR0144348,MR0220273,MR0301736}]
\label{lem: PL}
Let $(\sigma,\tau,h)$ be a contraction of $(A,d_A)$ onto $(B,d_B)$,
and let $\delta_A: A\to A$ be a perturbation of the differential $d_A$:
that is, $\delta_A$ is a degree $(+1)$ map such that $\breve{d}_A: = d_A +\delta_A$ squares to zero.
Then \[ \delta_B: =\sum_{i=0}^{\infty}\sigma\delta_A(h\delta_A)^i \tau \] is a perturbation of the differential $d_B$,
and $(\breve{\sigma},\breve{\tau},\breve{h})$ defined by
\begin{eqnarray*}
\breve{\sigma}&: =& \sum_{i=0}^{+\infty} \sigma(\delta_Ah)^i ,\\
\breve{\tau}&: =& \sum_{i=0}^{+\infty} (h\delta_A)^i \tau,\\
\breve{h}&: =& \sum_{i=0}^{+\infty} h(\delta_Ah)^i \\
\end{eqnarray*}
is a contraction of $(A,\breve{d}_A)$ onto $(B,\breve{d}_B: =d_B+\delta_B)$.
\end{lemma}

To be more precise, we should add some technical assumption ensuring convergence of the above infinite sums, but we will be loose in that respect and proceed formally.

\begin{definition}\label{def: semifullcontraction}
Given dg commutative algebras $(A,d_A,m_A )$, $(B,d_B,m_B )$ and a contraction $(\sigma,\tau,h)$ of $(A,d_A)$ onto $(B,d_B)$,
we say that $(\sigma,\tau,h)$ is a \emph{semifull algebra contraction} if the following identities are satisfied for all $a,b\in A$ and $x,y\in B$:
\begin{eqnarray}
\label{eqA1} h(\,(-1)^{|a|+1} h(a)b + ah(b)\,) &=& h(a)h(b),\\
\label{eqA2} h( a\tau(x)) &=& h(a)\tau(x),\\
\label{eqA3} \sigma(\,(-1)^{|a|+1} h(a)b + ah(b)\,) &=& 0,\\
\label{eqA4} \sigma( a\tau(x)) &=& \sigma(a)x,\\
\label{eqA5} \tau(xy) &=& \tau(x)\tau(y).	
\end{eqnarray}
\end{definition}

\begin{remark}\label{rem: semifullcontraction}
This class of contractions was introduced by Real \cite{MR1782593}.
More precisely, in \cite[Definition 4.5]{MR1782593}, Equations \eqref{eqA1}-\eqref{eqA4} are replaced by the seemingly weaker
\begin{gather*} h(h(a)h(b)) = h(h(a)\tau(x)) = 0 ,\\
\sigma(h(a)h(b)) = \sigma(h(a)\tau(x)) = 0 ,\end{gather*}
whereas Equation \eqref{eqA5} is maintained.
It is straightforward that, if Equations \eqref{eqA1}-\eqref{eqA4} are satisfied, then the above equations are satisfied as well.
In fact, our definition and the one from \cite{MR1782593} are equivalent (in practice, the above equations are usually easier to check, on the other hand, Equations \eqref{eqA1}-\eqref{eqA4} will be the key to the following computations).
To illustrate this fact, we show how to deduce Equation \eqref{eqA1} from the above relations: these imply \[ h(h(a) b) = h( h(a) (\tau\sigma-hd_A-d_Ah)(b))=-h(h(a)d_Ah(b)), \]
and similarly $h(a h(b)) = -h(d_Ah(a) h(b))$.
Thus \[h(\,(-1)^{|a|+1} h(a)b + ah(b)\,) = -hd_A (h(a)h(b)) = (d_Ah-\tau\sigma+\id) (h(a)h(b)) = h(a) h(b) .\]
Equations \eqref{eqA2}-\eqref{eqA4} can be deduced similarly.
\end{remark}

\begin{definition}
Given a dg algebra $(A,d_A,m_A )$, an \emph{algebra perturbation} of $d_A$ is a perturbation in the usual sense,
which is furthermore an algebra derivation.
\end{definition}

\begin{proposition}\label{Prop:semifullpeturbalsosemiful}
Given dg algebras $(A,d_A,m_A)$, $(B,d_B,m_B)$, a semifull algebra contraction $(\sigma,\tau,h)$ of $(A,d_A,m_A)$
onto $(B,d_B,m_B)$ and an algebra perturbation $\delta_A: A\to A$ of $d_A$, we apply the Perturbation Lemma \ref{lem: PL}.
Then $\delta_B$ is an algebra perturbation of $d_B$, and $(\breve{\sigma},\breve{\tau},\breve{h})$ is a semifull algebra contraction
of $(A,\breve{d}_A,m_A)$ onto $(B,\breve{d}_B,m_B)$.
\end{proposition}
\begin{proof}
This was proved in \cite{MR1782593}, see also \cite{MR1103672,MR1109665} for some related results.
For completeness, we sketch a proof of this fact.
To show that
\[ \breve{\tau} (x y) = \breve{\tau}(x)\breve{\tau}(y), \]
we prove inductively that
\[ (h\delta_A)^i \tau(x y) = \sum_{j=0}^i (h\delta_A)^j\tau(x) (h\delta_A)^{i -j}\tau(y),\qquad\forall i\geq0 .\]
The basis of the induction is Equation \eqref{eqA5}.
Assume the above identity holds for a given $i$.
Then
\begin{multline*}
(h\delta_A)^{i +1}\tau(x y) = h\delta_A\left( \sum_{j=0}^i (h\delta_A)^j\tau(x) (h\delta_A)^{i -j}\tau(y)\right)=\\
= h\left( \delta_A(h\delta_A)^i \tau(x)\tau(y)\right) +\\
+ \sum_{j=1}^i h\left( (-1)^{|x|}h\delta_A(h\delta_A)^{j-1}\tau(x)\delta_A(h\delta_A)^{i -j}\tau(y)
+\delta_A(h\delta_A)^{j-1}\tau(x) h\delta_A(h\delta_A)^{i -j}\tau(y)\right)+\\
+ (-1)^{|x|}h\left(\tau(x) \delta_A(h\delta_A)^i \tau(y) \right) =\\
= (h\delta_A)^{i +1}\tau(x) \tau(y) + \sum_{j=1}^i (h\delta_A)^j\tau(x) (h\delta_A)^{i +1-j}\tau(y)
+ \tau(x)(h\delta_A)^{i +1}\tau(y)
,\end{multline*}
using Equations \eqref{eqA1}-\eqref{eqA2}, which proves the inductive step.

Replacing the leftmost $h$ by $\sigma$ in the above computation, and using Equations \eqref{eqA3}-\eqref{eqA4} in the last passage, we see that
\[ \sigma\delta_A(h\delta_A)^i \tau(x y) = \sigma\delta_A(h\delta_A)^i \tau(x) y + (-1)^{|x|} x\,\sigma\delta_A(h\delta_A)^i \tau(y), \]
which implies that $\delta_B$ is indeed an algebra perturbation.

Finally, to show that $(\breve{\tau},\breve{\sigma},\breve{h})$ satisfies Equations \eqref{eqA1}-\eqref{eqA4} in Definition \ref{def: semifullcontraction},
it suffices to show that it satisfies the equivalent conditions in Remark \ref{rem: semifullcontraction}, which follow easily from the definitions.
\end{proof}

The main result of this section is the following theorem, which says that we can transfer derived Poisson algebra structures along semifull algebra contractions.

\begin{theorem}\label{Bandiera: homotopytransfer}
Let $(A,d_A,m_A)$ and $(B,d_B,m_B)$ be dg commutative algebras,
and let $(\sigma,\tau,h)$ be a semifull algebra contraction of $(A,d_A,m_A)$ onto $(B,d_B,m_B)$.
Let $\lambda_n: A^{\otimes n}\to A$, $n\geq2$, be a family of maps making $(A,d_A=:\lambda_1 ,\lambda_2,\ldots,\lambda_n,\ldots)$
into a degree $k$ derived Poisson algebra.
Via homotopy transfer along the contraction $(\sigma,\tau,h)$, there is an induced $L_\infty$ algebra structure on $B[-k]$,
whose structure maps we denote by $\ell_n: B^{\otimes n}\to B$, $n\geq2$.
These maps make $(B,d_B=:\ell_1 ,\ell_2,\ldots,\ell_n,\ldots)$ into a degree $k$ derived Poisson algebra.
Moreover, the $L_\infty$ quasi-isomorphism $\tau_\infty=(\tau_1, \tau_2,\ldots,\tau_n,\ldots)$ from $B[-k]$ to $A[-k]$
induced via homotopy transfer is a morphism of degree $k$ derived Poisson algebras.
\end{theorem}

\begin{remark}
It is a well known fact that $L_\infty$ algebra structures can be transferred 
along contractions. For a proof of this fact, we refer to \cite{MR1932522}.
See also \cite{arXiv:1705.02880,MR2361936,arXiv:1807.03086}. 
The homotopy transfer theorem for $L_\infty$ algebra structures
is a direct consequence of the Goldman--Millson theorem \cite{arXiv:1406.1751,arXiv:1407.6735}.
\end{remark}
\begin{proof}
In the following computations, to improve readability, we shall omit to make signs explicit, and instead denote by $\pm_K$ the appropriate Koszul signs (which can be worked out explicitly as explained in the previous subsections).
	
By homotopy transfer formulas,
$\tau_1=\tau$, and $\tau_{n+1}, \ell_{n+1}$, $n\geq1$, are defined recursively by
\begin{multline*} \tau_{n+1}(x_1,\ldots,x_n,y) = \\ \sum_{\substack{p\geq2,i_1,\ldots,i_{p-1}\geq1,\,i_p \geq0 \\ i_1+\cdots+i_p=n}}\sum_{\sigma\in\operatorname{Sh}(i_1,\ldots,i_p)}\pm_K\frac{1}{(p-1)!} h\lambda_p(\tau_{i_1}(x_{\sigma(1)},\ldots,x_{\sigma(i_1)}),\ldots,\tau_{i_p+1}(x_{\sigma(n-i_p+1)},\ldots,x_{\sigma(n)},y))
\end{multline*}
and\begin{multline*} \ell_{n+1}(x_1,\ldots,x_n,y) = \\ \sum_{\substack{p\geq2,i_1,\ldots,i_{p-1}\geq1,\,i_p \geq0 \\ i_1+\cdots+i_p=n}}\sum_{\sigma\in\operatorname{Sh}(i_1,\ldots,i_p)}\pm_K\frac{1}{(p-1)!} \sigma\lambda_p(\tau_{i_1}(x_{\sigma(1)},\ldots,x_{\sigma(i_1)}),\ldots,\tau_{i_p+1}(x_{\sigma(n-i_p+1)},\ldots,x_{\sigma(n)},y)).
\end{multline*}
For notational simplicity in the following computations,
we will omit the $x_1,\ldots,x_n$ variables, and abbreviate the above
equations as
\begin{multline*} \tau_{n+1}(x_1,\ldots,x_n,y) = \sum_{\substack{p\geq2,i_1,\ldots,i_{p-1}\geq1,\,i_p \geq0 \\ i_1+\cdots+i_p=n}}\sum_{\sigma\in\operatorname{Sh}(i_1,\ldots,i_p)}\pm_K\frac{1}{(p-1)!} h\lambda_p(\tau_{i_1}(\ldots),\ldots,\tau_{i_p+1}(\ldots,y))
\end{multline*}
and
\begin{multline*} \ell_{n+1}(x_1,\ldots,x_n,y) = \sum_{\substack{p\geq2,i_1,\ldots,i_{p-1}\geq1,\,i_p \geq0 \\ i_1+\cdots+i_p=n}}\sum_{\sigma\in\operatorname{Sh}(i_1,\ldots,i_p)}\pm_K\frac{1}{(p-1)!} \sigma\lambda_p(\tau_{i_1}(\ldots),\ldots,\tau_{i_p+1}(\ldots,y)).
\end{multline*}

We shall prove first that $\tau_\infty$ satisfies the required compatibilities with the products.
By assumption $\tau_1=\tau$ is a morphism of graded algebras.
Proceeding inductively,
\begin{flalign*}
&\tau_{n+1}(x_1,\ldots,x_n,y z) = &\\ &\sum_{\substack{p\geq2,i_1,\ldots,i_{p-1}\geq1,i_p\geq0 \\
i_1+\cdots+i_p=n}}\sum_{\sigma\in\operatorname{Sh}(i_1,\ldots,i_p)}\pm_K \frac{1}{(p-1)!}h\lambda_p(\tau_{i_1}(\ldots), \ldots,\tau_{i_p+1}(\ldots,y z)) =&\\
&\sum_{\substack{p\geq2,i_1,\ldots,i_{p-1}\geq1,i_p,i_{p+1}\geq0 \\ i_1+\cdots+i_{p+1}=n}}\sum_{\sigma\in\operatorname{Sh}(i_1,\ldots,i_{p+1})} \pm_K\frac{1}{(p-1)!}h\lambda_p(\tau_{i_1}(\ldots), \ldots,\tau_{i_p+1}(\ldots,y) \tau_{i_{p+1}+1}(\ldots, z)) =&\\
&\sum_{\substack{p\geq2,i_1,\ldots,i_{p-1}\geq1,i_p,i_{p+1}\geq0 \\ i_1+\cdots+i_{p+1}=n}}\sum_{\sigma\in\operatorname{Sh}(i_1,\ldots,i_{p+1})} \pm_K\frac{1}{(p-1)!}h\left(\lambda_p(\tau_{i_1}(\ldots), \ldots,\tau_{i_p+1}(\ldots,y)) \tau_{i_{p+1}+1}(\ldots, z)\right) +& \\
&\pm_K\frac{1}{(p-1)!}h\left(\tau_{i_p+1}(\ldots,y) \lambda_p(\tau_{i_1}(\ldots), \ldots,\tau_{i_{p+1}+1}(\ldots, z))\right) = & \\
& \tau_{n+1}(x_1,\ldots,x_n,y) \tau_1(z) \pm_K \tau_1(y) \tau_{n+1}(x_1,\ldots,x_n,z)
+ \sum_{\substack{i,j\geq1 \\ i+j=n}}\sum_{\substack{p\geq2,i_1,\ldots,i_{p-1}\geq1,i_p\geq0 \\ i_1+\cdots+i_p=i}}
\sum_{\substack{q\geq2,j_1,\ldots,j_{q-1}\geq1,j_q\geq0 \\ j_1+\cdots+j_q=j}}&\\
&\sum_{\sigma\in\operatorname{Sh}(i_1,\ldots,j_q)} \pm_K\frac{1}{(p-1)!}h\left( \lambda_p(\tau_{i_1}(\ldots), \ldots,\tau_{i_p+1}(\ldots,y))\frac{1}{(q-1)!}h\lambda_q(\tau_{j_1}(\ldots),\ldots,\tau_{j_q+1}(\ldots,z)) \right) &\\
& \pm_K\frac{1}{(q-1)!}h\left( \frac{1}{(p-1)!}h\lambda_p(\tau_{i_1}(\ldots), \ldots,\tau_{i_p+1}(\ldots,y)) \lambda_q(\tau_{j_1}(\ldots),\ldots,\tau_{j_q+1}(\ldots,z)) \right) =&\\
& \tau_{n+1}(x_1,\ldots,x_n,y) \tau_1(z) \pm_K \tau_1(y) \tau_{n+1}(x_1,\ldots,x_n,z) + \sum_{\substack{i,j\geq1 \\ i+j=n}}\sum_{\substack{p\geq2,i_1,\ldots,i_{p-1}\geq1,i_p\geq0 \\ i_1+\cdots+i_p=i}}\sum_{\substack{q\geq2,j_1,\ldots,j_{q-1}\geq1,j_q\geq 0 \\ j_1+\cdots+j_q=j}}&\\&\sum_{\sigma\in\operatorname{Sh}(i_1,\ldots,j_q)} \pm_K\frac{1}{(p-1)!}h\lambda_p(\tau_{i_1}(\ldots), \ldots,\tau_{i_p+1}(\ldots,y))\frac{1}{(q-1)!}h\lambda_q(\tau_{j_1}(\ldots),\ldots,\tau_{j_q+1}(\ldots,z))=&\\& \sum_{i=0}^n\sum_{\sigma\in\operatorname{Sh}(i,n-i)}\pm_K\tau_{i+1}(\ldots,y)\tau_{n-i+1}(\ldots,z), &
\end{flalign*}
where we used the identities \eqref{eqA1}-\eqref{eqA2}.

Comparing the formula for $\tau_{n+1}$ and the one for $\ell_{n+1}$ at the beginning of the proof, to show that the $\ell_{n+1}$ are multi-derivations
we can follow the above computations, replacing the leftmost $h$ by $\sigma$ and using the identities \eqref{eqA3}-\eqref{eqA4}
(instead of \eqref{eqA1}-\eqref{eqA2}) when appropriate.
\end{proof}

\section{Shifted derived Poisson manifolds}
\label{onion}

\subsection{Shifted derived Poisson manifolds}

\begin{definition}\label{Def: -kshiftedDPM}
A $(-k)$-shifted derived Poisson manifold is a $\ZZ$-graded manifold $\cM$
whose sheaf of functions $\structuresheaf{\mathcal{M}} $ is a sheaf of degree $k$ derived Poisson algebras.
\end{definition}

Equivalently, a $(-k)$-shifted derived Poisson manifold is a $\ZZ$-graded manifold $\cM$ such that $\cinf{\cM}$,
the space of global functions on $\cM$, is endowed with a degree $k$ derived Poisson algebra structure
\[ \lambda_l : (\cinf{\cM})^{\otimes l}\to\cinf{\cM} ,\quad l\geq 1 .\]

\begin{example}
Let $\mfg$ be a finite dimensional $L_\infty$ algebra.
By extending the $L_\infty$ structure maps $\lambda_l$ (${l\geq 1}$) on $\mfg$ to the completed symmetric algebra $\hat{\Sbullet}(\mfg[k])$ via the Leibniz rule,
one obtains a degree $k$ derived Poisson algebra on $\hat{ \Sbullet }(\mfg[k]) =C^\infty(\mfg^\vee[-k])$.
Thus $\mfg^\vee[-k]$ is a $(-k)$-shifted derived Poisson manifold, called \emph{$(-k)$-shifted derived Lie Poisson manifold}.
In particular, $\mfg^\vee$ is a \emph{derived Lie Poisson manifold}.
See \cite{MR3633027}.

If $\mfg$ is an ordinary Lie algebra, $(\mfg[0])^\vee=\mfg^\vee$ admits a Poisson manifold structure, called Lie--Poisson structure.
On the other hand, we have the standard Schouten algebra
$\Lambda^\bullet\mfg$,
which corresponds to the $(-1)$-shifted Poisson manifold structure on $ \mfg^\vee[-1]$.
\end{example}

The following proposition relates the notion of shifted derived Poisson manifolds introduced above with the one defined by Pridham \cite{MR3653066,arXiv:1804.07622}.

Let $\cM$ be a $\ZZ$-graded manifold. By $\hat{\XX}^\bullet_{\poly}(\cM,n)$,
we denote the formally completed Schouten--Nijenhuis algebra of $n$-shifted
polyvector fields on $\cM$ (see Appendix C).

\begin{proposition}
\label{Prop: kHomotopyPoissonGivesPi}
A $(-k)$-shifted derived Poisson manifold is equivalent to a dg manifold $(\cM, Q)$ equipped with a formal series of $(k-2)$-shifted polyvector fields
$\pi=\sum_{l= 2}^\infty \pi_l$ satisfying the Maurer--Cartan equation:
\begin{equation}\label{MC equation of pi}
[Q,\pi]+\frac{1}{2}[\pi,\pi]=0,
\end{equation}
where $\pi_l\in\hat{\XX}^l_{\poly}(\cM,k-2)$ is of total
degree $(+1)$ in $\hat{\XX}^\bullet_{\poly}(\cM,k-2)[k-1]$.
\end{proposition}
\begin{proof}
Assume that $\cM$ is a $(-k)$-shifted derived Poisson manifold.
Let $\lambda_l : (\cinf{\cM})^{\otimes l}\to\cinf{\cM}$, $l\geq 1$,
be the multi-brackets of the corresponding shifted $L_\infty$ algebra.

The unary bracket $\lambda_1 : \cinf{\cM}\to\cinf{\cM}$ is
a derivation of degree $(+1)$ and squares to zero.
Thus it determines a homological vector field $Q$ on $\mathcal{M}$,
which is also denoted by $\pi_1$.
Since $(\lambda_l)_{l=1}^\infty$ defines an $L_\infty$ algebra structure
on ${\cinf{\cM}}[-k]$, for any $l\geq 2$, $\lambda_l$ is a
skew-symmetric multi-bracket on $\cinf{\cM}[-k]$, i.e.,
\[ \lambda_l(\cdots,f,g\cdots)=-\minuspower{ \abs{f}^{[k]} \abs{g}^{[k]}}\lambda_l(\cdots,g,f\cdots) .\]
Let
$ \pi_l(f_1,f_2,\cdots,f_l) = \minuspower{\star} \lambda_l(f_1,f_2,\cdots,f_l)$,
where $\star=(l-1)\degree{f_1}^{[k]} +(l-2)\degree{f_2}^{[k]} +\cdots+\degree{f_{l-1}}^{[k]}$,
then $\pi_l$ is a symmetric multilinear map on $\cinf{\cM}[-k+1]$, i.e.\
\[ \pi_l(\cdots,f,g\cdots)= \minuspower{ \abs{f}^{[k-1]} \abs{g}^{[k-1]}}\pi_l(\cdots,g,f\cdots) .\]
It is clear that $\pi_l$ is a multi-derivation in each argument.
According to Appendix \ref{Appendix: shiftedpoly}, $\pi_l$ can be considered
as a $(k-2)$-shifted $l$-polyvector field on $\cM$,
i.e.\ $\pi_l\in\XX^l_{\poly}(\cM,k-2) $. Moreover, its total degree is
\[ \totalabs{\pi_l}_{k-2}=\abs{\pi_l}-l(k-1)=\abs{\lambda_l}-l(k-1)=2-k .\]
Therefore, when being considered as an element in
${\XX}^l_{\poly}(\cM,k-2) [k-1]$, $\pi_l$ is of degree $(+1)$.
Let $\Lambda=Q+\pi=\sum_{l\geq 1}\pi_l$. Then
$\Lambda$ is of total degree $(+1)$. Let
\[ \Pi=\frac{1}{2}[\Lambda,\Lambda] .\]
According to Lemma \ref{eq: Recollects},
\[ \Pi=
\Lambda\circ\Lambda
=\sum_{m,n\geq 1}\pi_m\circ\pi_n \in\bigoplus_{l= 2}^\infty {\XX}^l_{\poly}(\cM,k-2) {[k-1]} .\]
The weight $p$ component of $\Pi$ is
\begin{equation}\label{Eqt: Pip}
\Pi_p=\sum_{m+n-1= p} \pi_m\circ \pi_n\,.
\end{equation}

For any $f_1,f_2,\cdots,f_{m+n-1}\in \cinf{\cM}$,
\begin{eqnarray*}
&& (\pi_m\circ \pi_n) (f_1,\cdots,f_{m+n-1})\\
&=& \sum_{\sigma\in\mathrm{Sh}(m,n-1)}{\epsilon^{[k-1]}}(\sigma)\pi_m(\pi_n(f_{\sigma(1)},\cdots,f_{\sigma(n)}),f_{\sigma(n+1)},
\cdots,f_{\sigma(m+n-1)})\\
&=&\sum_{\sigma\in\mathrm{Sh}(m,n-1)}\minuspower{\star}
\\&&\qquad\qquad{\epsilon^{[k-1]}}(\sigma)\lambda_m(\lambda_n(f_{\sigma(1)},\cdots,f_{\sigma(n)}),f_{\sigma(n+1)},
\cdots,f_{\sigma(m+n-1)})\\
&& {\scriptscriptstyle \mbox{\small where } \star=(n-1)\abs{f_{\sigma(1)}}^{[k]}+\cdots+\abs{f_{\sigma(n-1)}}^{[k]}
+(m-1)(\abs{f_{\sigma(1)}}^{[k]}+\cdots+\abs{f_{\sigma(n)}}^{[k]} +2-n)+(m-2)\abs{f_{\sigma(n+1)}}^{[k]}+\cdots+\abs{f_{\sigma(m+n-2)}}^{[k]}}\\
&=&\minuspower{(m-1)n+K}\sum_{\sigma\in\mathrm{Sh}(m,n-1)}\abs{\sigma}
{\epsilon^{[k]}}(\sigma)\lambda_m(\lambda_n(f_{\sigma(1)},\cdots,f_{\sigma(n)}),f_{\sigma(n+1)},\cdots,f_{\sigma(m+n-1)}).
\end{eqnarray*}
Here $K$ denotes the constant integer $(p-1)\abs{f_1}^{[k]}+(p-2)\abs{f_2}^{[k]}+\cdots+\abs{f_{p-1}}^{[k]}$. \\
Equation \eqref{MC equation of pi} is equivalent to the condition that
$\Pi_p=0$ for all $p\geq 2$.
According to Equation \eqref{Eqt: Pip}, the latter is equivalent to
that $\forall$ $f_1,\cdots,f_p\in\cinf{\cM}$,
\[ \sum_{m+n-1=p}\minuspower{(m-1)n}\sum_{\sigma\in\mathrm{Sh}(m,n-1)}\abs{\sigma}
{\epsilon^{[k]}}(\sigma)\lambda_m(\lambda_n(f_{\sigma(1)},\cdots,f_{\sigma(n)}),f_{\sigma(n+1)},\cdots,f_{\sigma(p)})=0 .\]
It is clear that this is exactly the generalized Jacobi identity of
the $L_\infty$ algebra structure on $\cinf{\cM}[-k]$. 

The converse is proved by going backwards. This concludes the proof.
\end{proof}

\begin{example}
A $(-1)$-shifted derived Poisson manifold is equivalent to a dg manifold
$(\cM, Q)$ equipped with a formal series
$\pi=\sum_{l= 2}^\infty \pi_l$ satisfying the Maurer--Cartan equation:
\begin{equation}\label{eq: paris}
[Q,\pi]+\frac{1}{2}[\pi,\pi]=0,
\end{equation}
where, for each $l\geq 2$,
$\pi_l\in \Gamma(S^l T_\cM)$ is of degree $(+1)$, and the bracket
in Equation (\refeq{eq: paris}) is the canonical Poisson bracket on
$\Gamma(\hat{S} T_\cM)$ being identified with the space of
formal polynomials on the
symplectic manifold $T^\vee_\cM$.
\end{example}

\begin{definition}
Let $\cM$ and $\cM'$ be $(-k)$-shifted derived Poisson manifolds
with structure maps
$\lambda_l: $ $(\cinf{\cM})^{\otimes l}$ $\to$ $\cinf{\cM}$ and $\lambda'_l: $ $(\cinf{\cM'})^{\otimes l}$ $\to$ $\cinf{\cM'}$, $ l\geq 1$, respectively.
A morphism of $(-k)$-shifted derived Poisson manifolds from
$\cM$ to $\cM'$ is a map of $\mathbb{Z}$-graded manifolds
$\phi: \cM\to \cM'$ together with a collection of maps
\[ \varphi_n: (\cinf{\cM'})^{\otimes n}\rightarrow \cinf{\cM} ,\quad n=2,3,\cdots \]
such that $\varphi_\infty$ $=$ $(\varphi_1=\phi^*, \varphi_2,\varphi_3,\cdots)$ is a morphism of degree $k$ derived Poisson algebras
from $(\cinf{\cM'},$ $\lambda'_1$, $\lambda'_2$, $\cdots$ , $\lambda'_n$, $\cdots)$
to $(\cinf{\cM},\lambda_1, \lambda_2, \cdots , \lambda_n, \cdots)$.
\end{definition}
In particular, $\phi: \cM\to\cM'$ is a map of dg manifolds.

\subsection{$L_\infty$ algebroids and shifted derived Poisson manifolds}

Below we follow Bruce \cite{MR2840338} for the notations,
who considered the $\ZZ_2$-case.

\begin{definition}
\label{Defn: Linfinityalgebroid}
An $L_\infty$ algebroid consists of a vector bundle $\cL\to\cM$
of $\ZZ$-graded manifolds together with
\begin{itemize}
\item a sequence of multilinear maps $\lambda_l : \Lambda^l \sections{\cL} \to \sections{\cL}$
of degree $(2-l)$, $l\geq 1$, called multi-brackets,
that determine an $L_\infty$ algebra structure on $\sections{\cL}$,
the space of sections of $\cL\to\cM$; and
\item a sequence of bundle maps $\rho_l: \Lambda^l\cL \to T_{\cM}$
of degree $(1-l)$, $l\geq 0$, called multi-anchor maps,
that induce a morphism of $L_\infty$ algebras from $\sections{\cL}$ to
$\XXa(\cM)$
\end{itemize}
such that the following compatibility condition is satisfied:
\begin{multline}\label{pbracketwithf}
\algdpbracket{a_1,a_2,\cdots,a_{l-1},fa_l} = \alphapminusone(a_1,a_2,\cdots,a_{l-1})(f)a_l \\
+ \minuspower{(l+\degree{a_1}+\cdots+\degree{a_{l-1}})\degree{f}} f\algdpbracket{a_1,a_2,\cdots,a_{l-1},a_l}
\end{multline}
$\forall l\geq 1$, $a_1,\cdots,a_l\in\sections{\cL}$, and $f\in C^\infty(\cM)$.
\end{definition}

\begin{remark}
Note that the image of $\rho_1$ may \emph{not} be integrable.
When $\cM$ is an ordinary smooth
manifold $M$ (being considered of degree zero)
and the vector bundle $L=\bigoplus_{i\geq 0}L^i\to M$ is a non-negative graded vector bundle,
due to degree reasons, all higher anchor maps $\rho_l$ vanish except for $\rho_1$,
which must be a bundle map $L^0\to T_M$.
Our notion of $L_\infty$ algebroids reduces to the one studied by Laurent--Gengoux et.\ al.\ \cite{camillesingularfoliation}.
In this case, $\rho_1 (L^0)$ defines a singular foliation on $M$.
When $L$ is concentrated in degree $0$, it becomes a usual Lie algebroid over $M$.

We also note that various forms of $L_\infty$ algebroids have appeared
in the literature. We refer the reader to
\cite{MR2757715, MR2695305, MR1854642, MR2966944, MR3300319, MR3090103, MR3631929, MR3584886}
and the references there on the related topic.
\end{remark}

The following proposition extends
Theorem 1 and Corollary 2 in \cite{MR2840338} to the $\ZZ$-graded context.

\begin{proposition}
\label{Thm: LinfinitytotshiftedPoisson}
Let $\cL\to\cM$ be a vector bundle of $\ZZ$-graded manifolds,
and let $k\in\mathbb{Z}$ be a fixed integer.
The following statements are equivalent.
\begin{enumerate}
\item The vector bundle $\cL\to\cM$ is an $L_\infty$ algebroid.
\item The space $\sections{\hat{\Sbullet}(\cL[k])}$ is a degree $k$
derived Poisson algebra, whose $l$th-bracket
$\lambda_l : \Lambda^l \sections{\cL} \to \sections{\cL}$ is
of weight $(1-l)$.
\item The graded manifold $ \cL^\vee[-k] $ is a $(-k)$-shifted derived
Poisson manifold, where the weight of the $l$th-bracket
on $\cinf{\cL^\vee[-k]}$ is $(1-l)$.
\end{enumerate}
\end{proposition}

Recall that elements in $\sections{S^m(\cL[k])}$ are of weight $m$. 
The weight of the $l$th-bracket
\[ \lambda_l : \sections{ { S }^{\bullet}(\cL [k])}\times \cdots\times\sections{S^{\bullet}(\cL [k])}\to\sections{S^{\bullet}(\cL [k])} \]
is the difference of weights on both sides.

\begin{proof}
Note that (2) and (3) are clearly equivalent by definition. Below,
we prove the equivalence between (1) and (2).

Assume that $\cL\to \cM$ is an $L_\infty$ algebroid.
Then the $L_\infty$ algebra structure on $\Gamma (\cL)$ together with
the multi-anchor maps defines a degree $k$ derived Poisson algebra,
via the Leibniz rule, on
$\sections{\hat{ \Sbullet }(\cL [k])}$
by the following generating relations:
\begin{eqnarray}
\lambda_l ({a_1,\cdots,a_{l-1},f})&=&\rho_{l-1}(a_1,\cdots,a_{l-1})f, \qquad (l\geq 1)
\label{eq: Opera}
\\
\lambda_l(f,g,\cdots )&=&0, \qquad (l\geq 2), \label{eq: Louver}
\end{eqnarray}
$\forall a_1,\cdots,a_l\in \sections{\cL[k]}$, and $f,g\in \cinf{\cM}$.
It is easy to check that its	$l$th bracket is of weight $(1-l)$.

The converse can be proved by going backwards using Equations
\eqref{eq: Opera}-\eqref{eq: Louver}.
\end{proof}

\begin{remark}
We note that Vitagliano has already observed a similar result in
the algebraic context
of $LR_\infty$ algebras. For details, see \cite{MR3300319, MR3313214}.
\end{remark}

Let $\cL\to \cM$ be an $L_\infty$ algebroid with structure maps $(\lambda_l)_{l\geq 1}$ and $(\rho_l)_{l\geq 0}$ as in Definition
\ref{Defn: Linfinityalgebroid}.
The unary bracket $\lambda_1: \sections{\cL}\to\sections{\cL}$ and
$\rho_0\in \XXa (\cM)$ are compatible:
\[ \lambda_1(fa)=\rho_0(f)a+\minuspower{\abs{f}}f\lambda_1(a),\quad\forall a\in\sections{L},f\in\cinf{\cM} .\]
Introduce a dual map
$\lambda_1^\vee: \sections{\cL^\vee}\to\sections{\cL^\vee}$ by
\[ \pairing{\lambda_1^\vee(\xi)}{a}=\rho_0\pairing{\xi}{a}-\minuspower{\abs{\xi}}\pairing{\xi}{\lambda_1(a)},\quad\forall\xi\in\sections{\cL^\vee}, \forall a\in\sections{L} .\]
Since $\lambda_1^2=0$, it follows that $(\lambda_1^\vee)^2=0$. 
Hence $(\sections{\cL^\vee},\lambda_1^\vee)$ is a cochain complex.

\begin{definition}\label{Def: morphismLinfalgebroids}
Let $\cL_1\to \cM_1$ and $\cL_2\to \cM_2$ be
$L_\infty$ algebroids.
\begin{enumerate}
\item A morphism of $L_\infty$ algebroids
from $\cL_1\to \cM_1$ to $\cL_2\to \cM_2$ is a
sequence of bundle maps
\[ \begin{tikzcd}
\Lambda^p \cL_1 \arrow[d] \arrow[r, "\phi_p"] & \cL_2 \arrow[d] \\
\cM_1\arrow[r, "\phi_0"] & \cM_2
\end{tikzcd} \]
for $p=1,2,\cdots$, such that
the induced map
$\phi : \cL_1 [1]\to \cL_2 [1]$ is a map of dg manifolds.\footnote{This means that the induced morphism of commutative algebras
over $\phi_0^\vee: \cinf{\cM_2}\to \cinf{\cM_1}$:
\[ \phi^\vee : \cinf{\cL_2[1]}=\sections{\hat{S} (\cL_2^\vee[-1])}\to\cinf{\cL_1[1]}=\sections{\hat{S} (\cL_1^\vee[-1])} \]
commutes with the homological vector fields $Q_1$ and $Q_2$: $Q_2\circ \phi^\vee=\phi^\vee \circ Q_1$.
Here $Q_1$ and $Q_2$ are, respectively, homological vector fields on $\cL_1[1]$ and $\cL_2[1]$
corresponding to the $L_\infty$ algebroid structures as in Proposition \ref{Prop: LinfinityalgebroidandQ}.}
\item A quasi-isomorphism from $\cL_1\to \cM_1$ to $\cL_2\to \cM_2$
is a morphism of $L_\infty$ algebroids $\phi: \cL_1\to \cL_2$ such that
$\phi_1^\vee: \sections{\cL_2^\vee}\to \sections{\cL^\vee_1}$
is a quasi-isomorphism of cochain complexes.
\end{enumerate}
\end{definition}

The following fact can be easily verified.
\begin{proposition}
Let $\cL_1\to \cM $ and $\cL_2\to \cM $ be
$L_\infty$ algebroids. Let \[ \begin{tikzcd}
\Lambda^p \cL_1 \arrow[d] \arrow[r, "\phi_p"] & \cL_2 \arrow[d] \\
\cM \arrow[r, "\mathrm{Id}"] & \cM
\end{tikzcd} \]
(for $p=1,2,\cdots$) be a morphism of $L_\infty$ algebroids.
Then there induces a morphism of $(-k)$-shifted derived Poisson manifolds from $ \cL_2^\vee[-k]$ to $ \cL_1^\vee[-k]$.
\end{proposition}
\begin{proof}
Consider the morphism of $L_\infty$ algebroids
from $\cL_1\to \cM_1$ to $\cL_2\to \cM_2$ as in Definition \ref{Def: morphismLinfalgebroids}. In general, $\phi_0: \cM_1\to \cM_2$ is a map of graded manifold. When $\cM_1=\cM_2=\cM$ and $\phi_0=\mathrm{Id}$, there induces a family of $\cinf{\cM}$-multilinear maps
\[ \tilde{\phi}_p: \Lambda^p \sections{\cL_1}\to \sections{\cL_2} .\]
As $({\phi_p})_{p\geq 1}$ consists a morphism of $L_\infty$ algebroids
from $\cL_1\to \cM $ to $\cL_2\to \cM $, the above $({\tilde{\phi}_p})_{p\geq 1}$ defines a morphism of $L_\infty$ algebras from $\sections{\cL_1}$ to $\sections{\cL_2}$. Moreover, each $\tilde{\phi}_p$ is skew-symmetric.

Then one extends $({\tilde{\phi}_p})_{p\geq 1}$ to a family of
skew-symmetric maps: 
\[ \tilde{\phi}_p: \Lambda^p \big( \sections{\hat{\Sbullet}(\cL_1[k]) } [-k]
\big) \to \sections{\hat{\Sbullet}(\cL_2[k]) }[-k] ,\]
satisfying  Equation \eqref{eqn: derived morphism}.
Hence we obtain a morphism of degree $k$ derived Poisson algebras from $\sections{\hat{\Sbullet}(\cL_1[k])}$ to $\sections{\hat{\Sbullet}(\cL_2[k])}$, or equivalently, a morphism of $(-k)$-shifted derived Poisson manifolds from $\cL^\vee_2[-k]$ to $\cL^\vee_1[-k]$.
\end{proof}

\section{The $(+1)$-shifted derived Poisson algebra arising from a Lie pair}
\label{dandelion}

\subsection{First construction: $L_\infty$ algebroid arising from a Lie pair}
\label{Sec: resultofoutaliepair}
By a \emph{Lie pair} $(\LADL,\LADA)$, we mean an ordinary (non-graded) Lie algebroid $(\LADL,\baL{\cdot}{\cdot},\anchorL)$
over an ordinary (non-graded) smooth manifold $M$, together with a Lie subalgebroid
$(\LADA,\baA{\cdot}{\cdot},\anchorA)$ of $\LADL$ over the same base $M$.

For convenience, let us denote the quotient vector bundle $L/A$ by $B$.
Denote $\quotientmapLB : L\to B$ the projection map.
Note that $B$ is naturally an $A$-module:
\[ \nabla^{\Bott}_a b=\quotientmapLB[a,l]_L ,\] where $a\in\sections{A}$, $b\in\sections{B}$ and $l\in\sections{L}$
satisfying $\quotientmapLB(l)=b$.
The flat $A$-connection $\nabla^{\Bott}$ on $B$ is also known as the Bott
connection \cite{MR0362335,MR3439229}.

Let
\begin{eqnarray*}
&& \OmegaAsingle{\bullet}=\oplus_{k=0}\sections{\Lambda^k A^\vee}[-k],\\
&& \Omega^\bullet_{\LADA}(\Lambda^{\bullet}\moduleB)=
\oplus_{k, l} \sections{\Lambda^k \LADAs \otimes \Lambda^{l}\moduleB}[-k+l].
\end{eqnarray*}
Note that elements in $\Omega^k_{\LADA}(\Lambda^l \moduleB)$ are of degree $k-l$.
Denote by
\begin{equation}\label{eq: Paris}
\dAB: \qquad\Omega^k_{\LADA}(\Lambda^{l}\moduleB)\to\Omega^{k+1}_{\LADA}(\Lambda^{l}\moduleB)
\end{equation}
the standard Chevalley--Eilenberg differential corresponding to the Bott $A$-connection on $\Lambda^\bullet B$.

\begin{proposition}\label{Prop: main1}
Let $(L, A)$ be a Lie pair. Any splitting $\splitting: B\to L$ of the exact sequence
\begin{equation}\label{Eqt: exactseq}
0\to\LADA \xto{\embeddingi}\LADL \xto{\quotientmapLB} B\to 0
\end{equation}
induces an $L_\infty$ algebroid structure on the $\ZZ$-graded vector bundle $A[1]\oplus B \to A[1]$.
\end{proposition}

\begin{proof}
A splitting of \eqref{Eqt: exactseq} is a pair of maps
$\splitting : \moduleB\to\LADL$ and $\pA : \LADL\to\LADA$ such that
$\qB\rond \splitting =\id_{\moduleB}$, $\pA\rond \embeddingi=\id_{\LADA}$ and $\embeddingi\rond \pA+\splitting \rond \qB=\id_{\LADL}$:
\[ \begin{tikzcd}
0 \arrow[r, shift left] &
A \arrow[l, shift left] \arrow[r, shift left, "\embeddingi"] &
L \arrow[l, shift left, "\pA"] \arrow[r, shift left, "\qB"] &
B \arrow[l, shift left, "\splitting"] \arrow[r, shift left] &
0 \arrow[l, shift left]
\end{tikzcd} .\]

We therefore obtain an isomorphism of vector bundles over $M$:
\begin{equation}\label{I}
\kappa: \LADA\oplus\moduleB\stackrel{\simeq}{\longrightarrow}\LADL,\quad\qquad(a,b)\mapsto\embeddingi(a)+\splitting(b) ,
\end{equation}
for all $a\in\LADA$ and $b\in\moduleB$.
Consider the $\ZZ$-graded vector bundle $\cL\to A[1]$,
where $\cL: =A[1]\oplus B$. Note that the base of $\cL$ is $A[1]$.
Then $\cL[1]= A[1]\oplus B[1]\cong L[1]$ according to \eqref{I}.

Since $\LADL$ is a Lie algebroid, the standard Chevalley--Eilenberg differential
$\dL: \sections{\Lambda^\bullet L^\vee}\to\sections{\Lambda^{\bullet+1} L^\vee}$
defines a homological vector field $\QL$ on $\LADL[1]$.
Via the identification \eqref{I}, one obtains a homological vector field $Q$ on $\cL[1]$.
Since $\LADA$ is a Lie subalgebroid of $\LADL$, it is simple to see that the homological vector field $Q$ on $\cL[1]$
is indeed tangent to the zero section $\LADA[1]\subset \cL[1]$, whose restriction can be
identified with the Chevalley--Eilenberg differential $\dA$ on $\LADA[1]$.
Thus by Proposition \ref{Prop: LinfinityalgebroidandQ},
$\cL=\LADA[1]\oplus\moduleB\to\LADA[1] $ is indeed an $L_\infty$ algebroid.
\end{proof}

\begin{remark}
According to Proposition \ref{Prop: cannonicalLinfyDiracDual},
different choices of splittings
give rise to isomorphic $L_\infty$ algebras $\sections{\cL}\cong\OmegaAB$,
where the isomorphism is given by a collection of multilinear maps
\[ \varphi_n: \wedge^n \OmegaAB\rightarrow \OmegaAB,\quad n=1,2,\cdots \]
with $\varphi_1=\id$.
However since $\varphi_n$ ($n\geq 2$) is
not $\cinf{A[1]}=\Omega^\bullet_A$-multilinear,
$\varphi_n$ does not correspond to a bundle map $\wedge^n \cL \to \cL$.
Therefore, the corresponding $L_\infty$ algebroid structures
on the $\ZZ$-graded vector bundle $\cL=A[1]\oplus B \to A[1]$ are ``not''
isomorphic.
\end{remark}

As an immediate consequence of Proposition \ref{Prop: main1}, we have the following

\begin{proposition}\label{Prop: main1continue}
Let $(L, A)$ be a Lie pair. Then any splitting of \eqref{Eqt: exactseq}
induces a degree $(+1)$ derived Poisson algebra structure on
$ \tot \OmegaAwedgeB$,
where the multiplication is the wedge product,
and the {shifted} $L_\infty$ brackets are given as follows:
\begin{itemize}
\item[(1)]
The unary bracket $l_1$ is the Chevalley--Eilenberg differential $\dAB$ as in Equation \eqref{eq: Paris}.
\item[(2)]
The binary bracket
\[ \binarybracket{\argument}{\argument}_2 : \OmegaAdouble{i}{j} \times\OmegaAdouble{p}{l}\to\OmegaAdouble{i+p}{j+l-1} \]
is generated by the following relations:
\begin{itemize}
\item[a)] $ \binarybracket{u}{v}_2=\pr_B\baL{u}{v}$,\ \ \ $\forall u,v\in\sections{\moduleB}$;
\item[b)] $\binarybracket{u}{\omega}_2 =\pr_{\LADAs}(\LieDerivative_{u}{\omega})$, \ \ \ $\forall u\in \sections{\moduleB}$, $\omega\in \sections{\LADAs}$;
\item[c)] $\binarybracket{u}{f}_2=\anchorL(u)f$, \ \ \ $\forall u\in \sections{\moduleB}$, $ f\in \cinf{M}$;
\item[d)] $\binarybracket{\omega_1}{\omega_2}_2=0$,\ \ \ $\forall \omega_1,\omega_2\in\OmegaAsingle{\bullet}$.
\end{itemize}
\item[(3)] The ternary bracket
\[ \trinarybracket{\argument}{\argument}{\argument}_3 :
\OmegaAdouble{i}{j}\times\OmegaAdouble{p}{l}\times\OmegaAdouble{r}{s}\to\OmegaAdouble{i+p+r-1}{j+l+s-2} \]
is $\cinf{M}$-linear in each entry and generated by the following relations:
\begin{itemize}
\item[a)] $\trinarybracket{\argument}{\argument}{\argument}_3$
vanishes when being restricted to $\sections{\moduleB} \times \sections{\moduleB} \times \sections{\moduleB}$,
$\sections{\moduleB} \times \sections{\LADAs} \times \sections{\LADAs}$,
and $\sections{\LADAs} \times \sections{\LADAs} \times \sections{\LADAs}$;
\item[b)]
$\trinarybracket{u}{v}{\omega}_3=\pairing{\pr_{A}[{u},{v}]_L}{\omega}$, for all $u,v\in \sections{\moduleB}$ and $\omega\in \sections{\LADAs}$.
\end{itemize}
\item[(4)] All the rest of higher brackets vanish.
\end{itemize}
Here we have used the identification $L\cong A\oplus B$ and $L^\vee\cong A^\vee \oplus B^\vee$.
\end{proposition}

To prove Proposition \ref{Prop: main1continue}, we need to give an explicit expression for the homological vector field $Q$ on $\cL[1]$.
Choosing a splitting of \eqref{Eqt: exactseq}, we have an identification $\LADL \cong \LADA\oplus \moduleB$.
Besides the Bott $\LADA$-connection $\nabla$ on $\moduleB$, we also have a $\moduleB$-``connection'' on $\LADA$:
\begin{equation}
\label{eq: ba}
\sections{\moduleB}\otimes\sections{\LADA}\to\sections{\LADA} : (u,a)\mapsto\Delta_u a=\pA\baL{u}{a} .
\end{equation}

Introduce the following maps:
\begin{eqnarray*}
\anchorB: &&\moduleB\to T_\baseM, \quad \anchorB= \rho_L|_{B},\\
\baB{\argument}{\argument}: &&
\sections{\moduleB}\otimes\sections{
\moduleB}\to\sections{\moduleB }, \quad \baB{u_1}{u_2}=\quotientmapLB\baL{ {u_1}}{ {u_2}},
\\
\baBtoA{\argument}{\argument}: && \sections{\moduleB}\otimes\sections{
\moduleB}\to\sections{\LADA},\quad
\baBtoA{u_1}{u_2}=\pA\baL{u_1}{u_2},
\end{eqnarray*}
$\forall u_1,u_2\in\sections{\moduleB}$.
The Lie algebroid structure on $\LADL$ can be described as follows:
\begin{equation}\label{Eqt: allstructuremaps}
\begin{cases} \baL{ a_1}{ a_2}
= \baA{a_1}{a_2}; \\
\baL{ {u_1}}{ {u_2}}
= \baBtoA{ u_1}{ u_2} + \baB{ u_1}{ u_2}; \\
\baL{ a }{ {u }}
=- \baBAtoA{ u }{a }+ {\baABtoB{a }{ u }} ; \\
\anchorL( a+ {u})=\anchorA(a)+\anchorB( u),
\end{cases}
\end{equation}
$\forall a,a_1,a_2\in\sections{\LADA}$, $ u, u_1, u_2\in\sections{\moduleB}$.

Let $\OmegaAwedgeBs{p}{q}=\sections{\Lambda^p\LADAs\otimes\Lambda^q\moduleBs}$.
Then
\begin{equation}
\label{eq: Saint-Jean}
\cinf{\cL[1]}=\sections{\Lambda^\bullet \LADLs[-1]}
\cong \oplus_{p\geq 0,q\geq 0}\OmegaAwedgeBs{p}{q} [-p-q].
\end{equation}
Since $\wedge^\bullet B$ is an $\LADA$-module, we have the standard Chevalley--Eilenberg differential
\[ \dAB : \OmegaAwedgeBs{p}{q}\to\OmegaAwedgeBs{p+1}{q} .\]
In a similar fashion, we can define a map
\[ \dBA : \OmegaAwedgeBs{p}{q}\to\OmegaAwedgeBs{p}{q+1} .\]
Note that, however, $\dBA$ may not square to zero.
There is also a degree $(+1)$ derivation:
\[ \dbeta : \OmegaAwedgeBs{p}{q}\to\OmegaAwedgeBs{p-1}{q+2} ,\]
generated by the following relations:
\begin{align*}
&\quad\dbeta f = 0, \quad\quad \forall f\in \cinf{M}; \\
&\pairing{\dbeta(\omega)}{u_1\wedge u_2} = -\pairing{\omega}{\baBtoA{u_1}{u_2}},\quad \forall \omega\in\sections{\LADAs},u_1,u_2\in\sections{\moduleB}.
\end{align*}

The following lemma is immediate.

\begin{lemma}
Under the isomorphism \eqref{eq: Saint-Jean},
the homological vector field $Q $ on $\cL[1]$ is given by
\begin{equation}\label{Eqt: dL3parts}
\dL=\dAB+\dBA+\dbeta.
\end{equation}
\end{lemma}

Consider the vector bundle of
$\ZZ$-graded manifolds $\cL=\LADA[1]\oplus\moduleB\to\LADA[1]$.
Its space of sections
$\sections{\cL}$ is isomorphic to $\OmegaAB$.
The space of vector fields on $\LADA[1]$ is $\Der(\OmegaAsingle{\bullet})$.

\begin{lemma}
The $L_\infty$ algebroid structure on $\cL\to \LADA[1]$, where
$\cL=\LADA[1]\oplus\moduleB$,
as in Proposition \ref{Prop: main1} is determined by the following relations:
\begin{enumerate}
\item The $0$th anchor $\rho_0=\dA: $ $\OmegaAsingle{\bullet}\rightarrow \OmegaAsingle{\bullet+1}$ is the Chevalley--Eilenberg differential of the
Lie algebroid $A$.
\item The unary anchor $\rho_1: \cL\to T_{\LADA[1]}$ is determined by \[ \rho_1(u)=\Delta_u,\quad \forall u\in \sections{B} .\]
\item The binary anchor $\rho_2: \wedge^2\cL\to T_{\LADA[1]}$ is determined by \[ \rho_2(u_1,u_2)=i_{\beta(u_1,u_2)},\qquad \forall u_1,u_2\in \sections{B} .\]
\item The unary bracket $l_1$ coincides with
the Chevalley--Eilenberg differential $\dAB: \OmegaAsingle{\bullet}(B)\to\OmegaAsingle{\bullet+1}(B)$.
\item The binary bracket $l_2: \wedge^2\sections{\cL}\to \sections{\cL}$ is determined by \[ l_2(u_1,u_2)=[u_1,u_2]_B,\qquad \forall u_1,u_2\in \sections{B} .\]
\item The ternary bracket $l_3: \wedge^3\sections{\cL}\to \sections{\cL}$ is determined by \[ l_3(u_1,u_2,u_3)=0,\qquad \forall u_1,u_2,u_3\in \sections{B} .\]
\item All higher brackets $l_i$ ($i\geq 4$), and anchors $\rho_{j}$ ($j\geq 3$) vanish.
\end{enumerate}
\end{lemma}
\begin{proof}
We follow the construction as in the proof of Proposition~\ref{Prop: LinfinityalgebroidandQ}.
The relevant Voronov data are as follows: 
\begin{enumerate}
\item The graded Lie algebra of first order differential operators on
$\cL[1]$ is
\[ \mathfrak{A}=\mathscr{D}^{\leq 1}(\cL[1])\cong\Der(\OmegaAwedgeBs{\bullet}{\bullet})\oplus\OmegaAwedgeBs{\bullet}{\bullet} .\] 
Here elements in $\OmegaAwedgeBs{\bullet}{\bullet}=\cinf{\cL[1]}$ 
are considered as zeroth order differentials, i.e.
the multiplication by functions.
\item
The abelian Lie subalgebra $\mathfrak{a}=\sections{\cL[1]}\oplus\cinf{A[1]}\cong\OmegaAB\oplus\OmegaAsingle{\bullet}$.

The inclusion $ \mathfrak{a}\hookrightarrow\mathfrak{A}$ is given as
follows: $\forall$ $a\in\OmegaAsingle{r}(B)$, $\contraction{a}$ is the contraction operator
\[ \contraction{a} : \OmegaAwedgeBs{p}{q}\to{\OmegaAwedgeBs{p+r}{q-1}} .\]
Here $\OmegaAsingle{\bullet}$ is considered as a
subalgebra in $\OmegaAwedgeBs{\bullet}{\bullet}$ naturally.

\item The projection map $P : \mathfrak{A}\to\mathfrak{a}$ is given by
\[ P(v+\mu)=P_1(v)+P_2(\mu),\qquad \forall v\in\Der(\OmegaAwedgeBs{\bullet}{\bullet}), \mu\in \OmegaAwedgeBs{\bullet}{\bullet} ,\]
where 	 $P_2: \OmegaAwedgeBs{\bullet}{\bullet}\to\OmegaAsingle{\bullet}$
is the natural projection, and
$P_1: \Der(\OmegaAwedgeBs{\bullet}{\bullet})\to\OmegaAB$
is defined as follows: $\forall v\in\Der(\OmegaAwedgeBs{\bullet}{\bullet})$,
$P_1(v)\in\OmegaAB$ is determined by the composition
\[ \sections{B^\vee}\xto{v}\OmegaAwedgeBs{\bullet}{\bullet}\xto{\pr}\OmegaAsingle{\bullet} .\]
\item The homological vector field $Q\in\mathfrak{A}$ on $\cL[1]$
is $\dL=\dAB+\dBA+\dbeta$ as in Equation \eqref{Eqt: dL3parts}.
\end{enumerate}
We now apply Equations \eqref{Eqt: Qtobrackets} and \eqref{Eqt: Qtoanchors}
to obtain structure maps of the $L_\infty$ algebroid structure on $\cL\to
A[1]$.
For this purpose, observe that,
in the decomposition $\dL=\dAB+\dBA+\dbeta$, the operator
$\dAB$ is of weight $0$, $\dBA$ is of weight $1$, and $\dbeta$ is of weight $2$.

Applying Equation \eqref{Eqt: Qtoanchors}, the $0$th anchor is given by
\[ \rho_0(\omega)=P_2[\dL,\omega]=[\dAB,\omega]=\dA(\omega),\quad\forall\omega\in\OmegaAsingle{\bullet} .\]
This proves (1). The unary anchor can be obtained by the same method:
\begin{eqnarray*}
\rho_1(u)(\omega) &=&P_2[[\dL,\contraction{u}],\omega]= [[\dBA,\contraction{u}],\omega]\\
&=& (\dBA\circ \contraction{u}+\contraction{u}\circ \dBA)\omega\\
&=& \contraction{u}\dBA\omega=\Delta_u \omega, \ \ \ \forall\omega\in\OmegaAsingle{\bullet} .
\end{eqnarray*}
This proves (2). One can prove (3) similarly.

Next we describe the multi-brackets. For the unary bracket, we
apply Equation \eqref{Eqt: Qtobrackets}.
For any $u\in\sections{B}$, note that $l_1(u)\in\Omega_A^1(B)$.
For any $\eta\in \OmegaAwedgeBs{0}{1}=\sections{B^\vee}$, we have
\begin{eqnarray*}
\contraction{l_1(u)}(\eta)&=&(P_1[\dL,\contraction{u}])(\eta)
=[\dAB,\contraction{u}](\eta ) \\
&=& (\dAB\circ \contraction{u}+\contraction{u}\circ \dAB)(\eta )\\
&=& \contraction{\dAB u}(\eta).
\end{eqnarray*}
This proves (4).

For the binary bracket, let $u_1,u_2\in\sections{B}$,
$\eta\in\sections{B^\vee}$. Then,
\begin{eqnarray*}
\contraction{l_2(u_1,u_2)}(\eta) &=& (P_1[[\dL,\contraction{u_1}],\contraction{u_2}])(\eta)
= [[\dBA,\contraction{u_1}],\contraction{u_2}](\eta )\\
&=& (\dBA\circ \contraction{u_1}\circ \contraction{u_2}+
\contraction{u_1}\circ\dBA\circ \contraction{u_2}-\contraction{u_2}\circ\dBA\circ\contraction{u_1}-\contraction{u_2}\circ\contraction{u_1}
\circ\dBA)(\eta )\\
&=&\contraction{u_1}\dBA\pairing{u_2}{\eta}-\contraction{u_2}\dBA\pairing{u_1}{\eta}-
\contraction{u_1}\contraction{u_2}\dBA {(\eta)}\\
&=&\pairing{\eta}{[u_1,u_2]_B}=\contraction{{[u_1,u_2]}_B}(\eta).
\end{eqnarray*}
This proves (5). The rest of claims (6) and (7) can be proved in a similar
fashion.
\end{proof}

\begin{proof}[Proof of Proposition \ref{Prop: main1continue}]
According to Theorem \ref{Thm: LinfinitytotshiftedPoisson},
the $L_\infty$ algebroid $\cL\to A[1]$
induces a degree $(+1)$ derived Poisson algebra structure
on the space of sections
$\sections{\hat{\Sbullet}(\cL [1])} \cong \tot \OmegaAwedgeB$, whose
multi-brackets are obtained from
the structure maps on $\cL$ by applying the Leibniz rule.
The rest follows from a straightforward verification.
\end{proof}

\subsection{Second construction: Fedosov dg Lie algebroid arising from a Lie pair}
\label{pineapple}

Let us first recall Fedosov dg Lie algebroids as constructed
in \cite{BSX:17}.
We will use the same settings as in Section \ref{Sec: resultofoutaliepair}: a Lie pair $(L,A)$ and $B=L/A$.
Consider the graded vector bundle $\Fedosov\to L[1]$, where
$\Fedosov=L[1]\oplus B$.
It is clear that $\cinf{\Fedosov}\cong 
\sections{ \Lambda^{\bullet} L^\vee\otimes \hat{S} B^\vee}$.

Given a splitting of the short exact sequence \eqref{Eqt: exactseq}, the following
maps are established in \cite{BSX:17,arXiv:1605.09656}
\begin{itemize}
\item $\delta : \Gamma (\Lambda^{\bullet} L^\vee\otimes \hat{S} {B^\vee})
\to\Gamma (\Lambda^{\bullet+1}L^\vee\otimes \hat{S} {B^\vee})$, a degree $(+1)$ derivation;
\item $\sigma : \Gamma ( \Lambda^{\bullet} L^\vee\otimes \hat{S} {B^\vee})
\to \Gamma (\Lambda^{\bullet} A^\vee)$, the projection;
\item $\tau : \Gamma ( \Lambda^{\bullet} A^\vee) \to \Gamma (\Lambda^{\bullet} L^\vee\otimes \hat{S} {B^\vee})$, the inclusion;
\item $h : \Gamma (\Lambda^{\bullet} L^\vee\otimes \hat{S} {B^\vee}) \to
\Gamma (\Lambda^{\bullet-1}L^\vee\otimes \hat{S} {B^\vee})$, the homotopy map;
\item $ \sigmanatural=\sigma\otimes 1 : \Gamma ( \Lambda^{\bullet} L^\vee\otimes \hat{S} {B^\vee}\otimes \Lambda^\bullet B) \to \Gamma ( \Lambda^{\bullet} A^\vee \otimes\Lambda^\bullet B)$,
the projection;
\item $\taunatural=\tau\otimes 1 :
\Gamma (\Lambda^{\bullet} A^\vee \otimes\Lambda^\bullet B)
\to \Gamma ( \Lambda^{\bullet} L^\vee\otimes \hat{S} {B^\vee} \otimes\Lambda^\bullet B )$,
the inclusion;
\item $\etendu{h}=h\otimes 1: \Gamma (\Lambda^{\bullet} L^\vee\otimes \hat{S} {B^\vee} \otimes\Lambda^\bullet B)\to \Gamma (
{\Lambda^{\bullet-1}L^\vee\otimes \hat{S}{B^\vee}\otimes\Lambda^\bullet B})$,
the homotopy map.
\end{itemize}

We recall a result in \cite{arXiv:1605.09656}:

\begin{proposition}[Theorem~2.5 in \cite{arXiv:1605.09656}]
\label{SXfedosovpapermainresult1}
Let $(L,A)$ be a Lie pair.
Given a splitting of the short exact sequence \eqref{Eqt: exactseq}
and a torsion-free $L$-connection $\nabla$ on $B$ extending the Bott $A$-connection,
there exists a \emph{unique} 1-form valued in formal vertical vector fields of $B$:
\[ X^\nabla\in\sections{L\dual\otimes\hat{S}^{\geqslant 2} B\dual \otimes B} \]
satisfying $\etendu{h}(X^\nabla)=0$ and such that the derivation
$Q : \sections{\Lambda^\bullet L^\vee\otimes\hat{S} B^\vee} \to \sections{\Lambda^{\bullet+1}L^\vee\otimes\hat{S}B^\vee}$
defined by
\begin{equation}\label{eq: Q}
Q=-\delta+d_L^\nabla+X^\nabla
\end{equation}
satisfies $Q^2=0$.
Here
\begin{enumerate}
\item $d_L^\nabla : \sections{\Lambda^{\bullet}L^\vee\otimes\hat{S}{B^\vee}}\to\sections{\Lambda^{\bullet+1}L^\vee\otimes\hat{S}{B^\vee}}$
is the covariant derivative associated with the $L$-connection $\nabla$ on $B$;
\item $X^\nabla$ acts on the algebra $\sections{\Lambda^\bullet L^\vee\otimes\hat{S}B^\vee}$
as a derivation in a natural fashion.
\end{enumerate}
\end{proposition}

As a consequence, $(\Fedosov=L[1]\oplus B, Q)$ is a dg manifold, called the Fedosov dg manifold \cite{arXiv:1605.09656}. Consider the surjective submersion $\cM\to M$.
Let $\cF\to\cM$ denote the pullback of the vector bundle $B\to M$ through $\cM\to M$.
It is a graded vector bundle whose total space $\cF$ is the graded manifold with support $M$
associated with the graded vector bundle $L[1]\oplus B\oplus B\to M$.
Its space of sections $\sections{\cF\to\cM}$ is canonically identified with
$C^\infty(\cM)\otimes_{C^\infty(M)}\sections{B}=\sections{\Lambda^\bullet L\dual\otimes\hat{S}(B\dual)\otimes B}$.
It is naturally a vector subbundle of $T_{\cM}\to\cM$; the inclusion $\sections{\cF\to\cM}\into\XXa(\cM)$
takes the section $(\lambda\otimes\chi^J)\otimes\partial_k\in C^\infty(\cM)\otimes_{C^\infty(M)}\sections{B}$ of the vector bundle $\cF\to\cM$
to the derivation \[ \mu\otimes\chi^M\mapsto \sum_i \lambda\wedge\mu\otimes M_i\chi^{J+M-e_i} \] of $C^{\infty}(\cM)$. Here $M=(\cdots M_i \cdots)$ denotes a multi-index, and $e_i$ denotes the multi-index all of whose components are $0$ except for the $i$th which is equal to $1$.
It is simple to see that $\cF\subset T_\cM$ is a dg foliation of the dg manifold $(\cM, Q)$, which is called
the Fedosov dg Lie algebroid \cite{BSX:17}. 
Hence $(\Gamma (\cM; \wedge^\bullet \cF), \LieDer_Q)$ is 
a $(+1)$-shifted derived Poisson algebra (or a differential
Gerstenhaber algebra), where $\Gamma (\cM; \wedge^\bullet \cF)
=\oplus_k \Gamma (\cM; \wedge^k \cF)[k]$, and $ \LieDer_Q=[Q,\argument]$ denotes
the Lie derivative. Since $\Gamma (\cM; \Lambda^\bullet \cF)\cong
\sections{\Lambda^{\bullet}L^\vee\otimes\hat{S}B^\vee\otimes\Lambda^{\bullet}B}
=\oplus_k \sections{\Lambda^{\bullet}L^\vee\otimes\hat{S}B^\vee\otimes\Lambda^{k}B}[k]$,
it thus follows that
\[ (\sections{\Lambda^{\bullet}L^\vee\otimes\hat{S}B^\vee\otimes\Lambda^{\bullet}B},\LieDer_Q) \]
is a $(+1)$-shifted derived Poisson algebra (or a differential Gerstenhaber algebra).

We need another result:
\begin{proposition}[\cite{BSX:17}]
\label{SXfedosovpapermainresult2}
Under the same hypothesis as in Proposition \ref{SXfedosovpapermainresult1}, there is an induced contraction datum
\begin{equation}
\label{Eqt: sigmatausmailnaturalquasi}
\begin{tikzcd}
(\tot \Gamma(\Lambda^\bullet L^\vee\otimes\hat{S} B^\vee \otimes\Lambda^\bullet B),\LieDer_Q)
\arrow[loop left, "\smilehnatural"] \arrow[r, shift left, "\sigmanatural"]
& (\tot \OmegaAwedgeB,\dAB). \arrow[l, shift left, "\smiletaunatural"]
\end{tikzcd}
\end{equation}
The maps $\smilehnatural$ and $\smiletaunatural$ are defined by (see \cite{BSX:17}):
\begin{gather*}
\smilehnatural=\hnatural+\sum_{i=1}^\infty(\etendu{h}\circ \LieDer_\varrho)^i \hnatural ,\\
\smiletaunatural =\taunatural+\sum_{i=1}^\infty(\etendu{h}\circ \LieDer_\varrho)^i \taunatural
.\end{gather*}
\end{proposition}
\begin{remark} We remark that the above is indeed a semifull algebra 
contraction, in the sense of Definition \ref{def: semifullcontraction}. In fact, as explained in detail in \cite[loc. cit.]{BSX:17},
the contraction $(\sigmanatural,\smiletaunatural,\smilehnatural)$ is obtained by applying the perturbation Lemma \ref{lem: PL} to a second contraction involving the maps $(\sigmanatural,\taunatural,\hnatural)$. The fact that the latter contraction is semifull can be shown by a simple direct computation, using Remark \ref{rem: semifullcontraction} and the explicit definition of the maps $\sigmanatural,\taunatural,\hnatural$ given in
\cite{BSX:17}. The fact that the contraction $(\sigmanatural,\smiletaunatural,\smilehnatural)$ is semifull follows immediately from Proposition \ref{Prop:semifullpeturbalsosemiful}.
\end{remark}

The differential Gerstenhaber algebra $(\sections{\Lambda^{\bullet} L^\vee\otimes\hat{S}B^\vee\otimes\Lambda^{\bullet} B},\LieDer_Q)$ is, by definition,
a degree $(+1)$ derived Poisson algebra, whose unary
bracket is $\LieDer_Q$, binary bracket is the Schouten bracket, and all higher brackets vanishes.
By homotopy transfer Theorem \ref{Bandiera: homotopytransfer}, there is a degree $(+1)$ derived Poisson algebra structure induced on $\tot \OmegaAwedgeB $,
whose first bracket is $\dAB$.

Below is our main result in this section.
\begin{proposition}\label{Prop: FedosovSameGalgebra}
Let $(L, A)$ be a Lie pair.
Choose a splitting of the exact sequence \eqref{Eqt: exactseq}, and a torsion free $L$-connection $\nabla$ on $B$ that extends the $A$-module structure of $B$.
Then the degree $(+1)$ derived Poisson algebra structure on $\OmegaAwedgeB$ obtained by homotopy transfer from the one on
$\sections{\Lambda^{\bullet} L^\vee\otimes\hat{S} {B^\vee}\otimes\Lambda^{\bullet} B}$,
which is induced from the Fedosov dg Lie algebroid $\cF\to \cM$, coincides with the one as in Proposition \ref{Prop: main1continue}.
\end{proposition}

Note that the construction of a Fedosov manifold $\cM$ and
its Fedosov dg Lie algebroid $\cF\to \cM$ depends on
the choice of a torsion free $L$-connection $\nabla$ on $B$ extending
the $A$-module structure of $B$. The above proposition
indicates that, however, the degree $(+1)$ derived Poisson algebra structure on $\OmegaAwedgeB$ obtained by homotopy transfer from
$(\Gamma (\cM; \wedge^\bullet \cF), \LieDer_Q)$
is independent of the choice of $\nabla$.

The rest of this section is devoted to prove this proposition. We need to recall some techniques and facts that are already shown in \cite{BSX:17}.
Let us resume the settings and notions in Section \ref{Sec: resultofoutaliepair}.
We chose a splitting of Sequence \eqref{Eqt: exactseq} so that one treats $L=A\oplus B$ directly.

Given a torsion free $L$-connection $\nabla$ on $B$ that extends the $A$-module structure of $B$, one can write
\[ \nabla_{a+b}b'=\nabla^A_ab'+\Delta^B_bb',\quad\forall a\in\sections{A},b,b'\in\sections{B} .\]
Here the $\Delta^B$ can be thought of as a $B$-``connection'' on $B$.
The condition that $\nabla$ being torsion free means
\begin{equation}\label{Eqt: torisionfreenabla}
\Delta^B_{b_1}b_2-\Delta^B_{b_2}b_1=\baB{b_1}{b_2},\quad\forall b_1,b_2\in\sections{B}
.\end{equation}
Here $\baB{\argument}{\argument}$ is the $B$-``bracket'' introduced in Equation \eqref{Eqt: allstructuremaps}.

In what follows, let us denote \[ \Cddd{p}{q}{r}=\sections{ \Lambda^{p}A^\vee\otimes \Lambda^{q} B^\vee\otimes {S}^{r} B^\vee} .\]
Recall $\cM=L[1]\oplus B$. Hence \[ \cinf{\cM}=\sections{ \Lambda^{\bullet} L^\vee\otimes \hat{S} B^\vee}
=\sections{ \Lambda^{\bullet} A^\vee\otimes \Lambda^{\bullet} B^\vee\otimes \hat{S} B^\vee}
=\prod_{p,q,r\geq 0}\Cddd{p}{q}{r} .\]

Recall Equation \eqref{Eqt: dL3parts}, where we split $\dL$ into three components. Abusing notations, let us again write
\[ \dL^\nabla=\dAB+\dBA+\dbeta: \,\sections{\Lambda^{\bullet} L^\vee\otimes \hat{S} {B^\vee}}\to\sections{\Lambda^{\bullet+1}L^\vee\otimes \hat{S} {B^\vee}} ,\]
where
\begin{gather*}
\dAB: \,\Cddd{p}{q}{r} \to \Cddd{p+1}{q}{r} ,\\
\dBA: \,\Cddd{p}{q}{r} \to \Cddd{p}{q+1}{r} ,\\
\dbeta: \,\Cddd{p}{q}{r} \to \Cddd{p-1}{q+2}{r}
.\end{gather*}

All the following are due to the relevant definitions and facts in~\cite{BSX:17}.
\begin{itemize}
\item The kernel of $\sigmanatural$ is
$(\Cddd{\bullet}{\geq 1}{\bullet}\oplus \Cddd{\bullet}{\bullet}{\geq 1})\otimes \sections{\Lambda^\bullet B}$.
\item The map $\etendu{h}$ sends
$\Cddd{p}{q}{r}\otimes \sections{\Lambda^s B}$ to $\Cddd{p}{q-1}{r+1}\otimes \sections{\Lambda^s B}$.
\item The map \[ \varrho=d_L^\nabla+X^\nabla \]
is a degree $(+1)$ derivation of $\cinf{\Fedosov}$, and in fact a perturbation of the cochain complex $(\cinf{\Fedosov},-\delta)$.
Moreover, the map $\LieDer_\varrho$ satisfies
\[ \LieDer_\varrho(\Cddd{p}{q}{r}\otimes \sections{\Lambda^s B})\subset (\Cddd{p+1}{q}{\geq r}\oplus \Cddd{p}{q+1}{\geq r}\oplus \Cddd{p-1}{q+2}{\geq r})\otimes\sections{\Lambda^s B} .\]
\item The Schouten--Nijenhuis bracket in $\sections{\Lambda^{\bullet} L^\vee\otimes\hat{S} B^\vee\otimes\Lambda^{\bullet} B}=\Cddd{\bullet}{\bullet}{\bullet}\otimes \sections{\Lambda^\bullet B}$ satisfies
\[ [\Cddd{p}{q}{r}\otimes \sections{\Lambda^s B},\Cddd{a}{b}{c}\otimes \sections{\Lambda^d B}]
\subset \Cddd{p+a}{q+b}{r+c-1}\otimes \sections{\Lambda^{s+d-1} B} .\]
\end{itemize}

Using these facts, the following formulas can be straightforward verified.

\begin{lemma}\label{lemma: temp1}
\begin{itemize}
\item[1)]
For any $b\in\sections{B}$, one has
\begin{equation}\label{Eqt: smiletaunaturalb}
\smiletaunatural{(b)}=b+\sum_{i=1}^\infty(\etendu{h}\circ\LieDer_\varrho)^i(b)\equiv b+\etendu{h}(\dBA b) \mod \Cddd{0}{0}{\geq 2}\otimes\sections{B}.
\end{equation}
Note that $\etendu{h}(\dBA b)\in\Cddd{0}{0}{1}\otimes\sections{B}$.
\item[2)]
For any $\theta\in\sections{A^\vee}$, one has
\begin{equation}\label{Eqt: smiletaunaturalxi}
\smiletaunatural{(\theta)}=\theta+\sum_{i=1}^\infty({h}\circ\varrho)^i(\theta)\equiv \theta+ {h}(\dBA \theta)+{h}(\dbeta \theta)
\mod (\Cddd{1}{0}{\geq 2}\oplus\Cddd{0}{1}{\geq 2}).
\end{equation}
Note that $ {h}(\dBA \theta)\in\Cddd{1}{0}{1}$ and ${h}(\dbeta\theta)\in\Cddd{0}{1}{1}$.
\item[3)]For any $\varpi\in\Cddd{p}{q}{r}\otimes\sections{\Lambda^s B}$, one has
\begin{equation}\label{Eqt: smilehnaural}
\smilehnatural(\varpi)\equiv\hnatural(\varpi)\mod\bigoplus_{i+j=p+q-1}\Cddd{i}{j}{\geq r+2}\otimes\sections{\Lambda^s B}.
\end{equation}
Note that $\hnatural(\varpi)\in\Cddd{p}{q-1}{r+1}\otimes\sections{\Lambda^s B}$.
\end{itemize}
\end{lemma}

An immediate consequence is the following fact.
\begin{lemma}
For any $\varpi_1,\varpi_2\in\Cddd{\bullet}{\bullet}{\bullet}\otimes\sections{\Lambda^s B}$, one has
\begin{equation}\label{Eqt: temp3}
\sigmanatural[\smilehnatural(\varpi_1),\smilehnatural(\varpi_2)]=0.
\end{equation}
\end{lemma}

The following identities are due to the definition of $h$ (see \cite{BSX:17}). 
We omit the details of verification.

\begin{lemma}
For all $b,b_1,b_2\in\sections{B}$, $\theta\in\sections{A^\vee}$, we have
\begin{eqnarray*}
\contraction{b_1}\etendu{h}(\dBA b_2)&=&\Delta^B_{b_1}b_2,\\
\contraction{b}{h}(\dBA \theta)&=&\Delta_{b}\theta,\\
\contraction{b_1} h \contraction{b_2} {h}(\dbeta \theta) &=& \frac{1}{2}\pairing{\baBtoA{b_1}{b_2}}{\theta}.
\end{eqnarray*}
Here $\Delta$ is the $B$-``connection'' on $A$ introduced in
Equation \eqref{eq: ba}.
\end{lemma}

We are now able to show the following
\begin{lemma}
For $b_1,b_2\in\sections{B}$, $\theta_1,\theta_2\in\sections{A^\vee}$ and $f\in\cinf{M}$, one has
\begin{align}\label{Eqt: contractedbrackets}
\begin{cases}
\sigmanatural[\smiletaunatural{(b_1)},\smiletaunatural{(b_2)}]= \baB{b_1}{b_2} ,\\
\sigmanatural[\smiletaunatural{(b_1)},\smiletaunatural{(\theta_1)}]=\baBAtoA{b_1}{\theta_1} ,\\
\sigmanatural[\smiletaunatural{(b_1)},f]=\anchorB{(b_1)}{f} ,\\\sigmanatural[\smiletaunatural{(\theta_1) },\smiletaunatural{(\theta_2)}]=0
.\end{cases}
\end{align}
\end{lemma}

\begin{proof}
By Equation \eqref{Eqt: smiletaunaturalb}, we have
\begin{eqnarray}\nonumber
[\smiletaunatural{(b_1)},\smiletaunatural{(b_2)}]
&\equiv & \contraction{b_1}\etendu{h}(\dBA b_2)-\contraction{b_2}\etendu{h}(\dBA b_1) \mod \Cddd{0}{0}{\geq 1}\otimes\sections{B} \\\nonumber
&\equiv & \Delta^B_{b_1}b_2-\Delta^B_{b_2}b_1 \mod \Cddd{0}{0}{\geq 1}\otimes\sections{B} \\\label{Eqt: temp1}
&\equiv & [b_1, b_2]_B \mod \Cddd{0}{0}{\geq 1}\otimes\sections{B}.
\end{eqnarray}
The last step is due to $\nabla$ being torsion free (Equation \eqref{Eqt: torisionfreenabla}).
Applying $\sigmanatural$, the first identity in Equation \eqref{Eqt: contractedbrackets} is immediate.

Similarly, by Equations \eqref{Eqt: smiletaunaturalb} and \eqref{Eqt: smiletaunaturalxi},
\begin{eqnarray}\nonumber
[\smiletaunatural{(b)},\smiletaunatural{(\theta)}]
&\equiv & \contraction{b} {h}(\dBA \theta)+\contraction{b} {h}(\dbeta \theta) \mod (\Cddd{1}{0}{\geq 1}\oplus\Cddd{0}{1}{\geq 1}) \\\label{Eqt: temp2}
&\equiv & \Delta_b \theta +\contraction{b} {h}(\dbeta \theta) \mod (\Cddd{1}{0}{\geq 1}\oplus\Cddd{0}{1}{\geq 1}).
\end{eqnarray}
Notice that $\contraction{b}{h}(\dbeta\theta)\in\Cddd{0}{1}{0}$.
Then applying $\sigmanatural$, the second identity in Equation \eqref{Eqt: contractedbrackets} is immediate.

The third identity easily follows from Equation \eqref{Eqt: smiletaunaturalb}:
\[ [\smiletaunatural{(b_1)},f]\equiv \anchorB(b_1) f \mod \Cddd{0}{0}{\geq 1} .\]
The last identity is obvious.
\end{proof}

\begin{lemma}
For $b,b_1,b_2,b_3\in\sections{B}$ and $\theta,\theta_1,\theta_2,\theta_3\in\sections{A^\vee}$, one has
\begin{align}\label{Eqt: contracted3brackets}
\begin{cases}
\sigmanatural\bigl[\smiletaunatural(b_1),\smilehnatural[\smiletaunatural{(b_2)},\smiletaunatural{(\theta)}]\bigr]= \frac{1}{2}\pairing{\baBtoA{b_1}{b_2}}{\theta} ,\\
\sigmanatural\bigl[\smiletaunatural(\theta),\smilehnatural[\smiletaunatural{(b_1)},\smiletaunatural{(b_2)}] \bigr]= 0,\\
\sigmanatural\bigl[\smiletaunatural(b_1),\smilehnatural[\smiletaunatural{(b_2)},\smiletaunatural{(b_3)}] \bigr]= 0,\\
\sigmanatural\bigl[\smiletaunatural(b),\smilehnatural[\smiletaunatural{(\theta_1)},\smiletaunatural{(\theta_2)}] \bigr]= 0,\\
\sigmanatural\bigl[\smiletaunatural(\theta_1),\smilehnatural[\smiletaunatural{(\theta_2)},\smiletaunatural{(b)}] \bigr]= 0,\\
\sigmanatural\bigl[\smiletaunatural(\theta_1),\smilehnatural[\smiletaunatural{(\theta_2)},\smiletaunatural{(\theta_3)}] \bigr]= 0
.\end{cases}
\end{align}
\end{lemma}

\begin{proof}
We show the first identity in \eqref{Eqt: contracted3brackets}.
By Equations \eqref{Eqt: temp2} and \eqref{Eqt: smilehnaural}, we have
\[ \smilehnatural [\smiletaunatural{(b_2)},\smiletaunatural{(\theta)}] \equiv h\contraction{b_2}{h}(\dbeta \theta) \mod\Cddd{0}{0}{\geq 2} ,\]
where $h\contraction{b_2}{h}(\dbeta \theta)\in\Cddd{0}{0}{1}$.

Therefore, using Equation \eqref{Eqt: smiletaunaturalb}, we get
\begin{eqnarray*}
\bigr[\smiletaunatural(b_1),\smilehnatural
[\smiletaunatural{(b_2)},\smiletaunatural{(\theta)}]\bigr]
&\equiv& \contraction{b_1} h \contraction{b_2} {h}(\dbeta \theta) \mod \Cddd{0}{0}{\geq 1}\\
&\equiv& \frac{1}{2}\pairing{\baBtoA{b_1}{b_2}}{\theta} \mod \Cddd{0}{0}{\geq 1}.
\end{eqnarray*}
Applying $\sigmanatural$, one gets the first identity.
The remaining identities can be worked out similarly.
\end{proof}

The following lemma is proved along the same lines.
\begin{lemma}
The following equation holds for any $x,y,z\in\sections{B}$ or $\sections{A^\vee}$:
\begin{equation}\label{Eqt: contracted4brackets2}
\smilehnatural\bigl[\smiletaunatural(x),\smilehnatural[\smiletaunatural{(y)},\smiletaunatural{(z)}]\bigr]= 0.
\end{equation}
\end{lemma}

With these preparatory work, we finally give the proof of our main result --- Proposition \ref{Prop: FedosovSameGalgebra}.
\begin{proof}[Proof of Proposition \ref{Prop: FedosovSameGalgebra}]

Via the contraction data $(\smilehnatural,\sigmanatural,\smiletaunatural)$ in
Equation \eqref{Eqt: sigmatausmailnaturalquasi},
one constructs the $L_\infty$ brackets on $\tot \OmegaAwedgeB$ (Theorem \ref{Bandiera: homotopytransfer}).
The first one is already shown to be $\dAB$.

According to the proof of Theorem \ref{Bandiera: homotopytransfer}, the binary bracket reads
\[ l_2(X,Y)=\sigmanatural[\smiletaunatural{(X)},\smiletaunatural{(Y)}],\quad\forall
X,Y\in \tot \OmegaAwedgeB .\]
Compare the identities in Equation \eqref{Eqt: contractedbrackets} with those in (2) of Proposition \ref{Prop: main1continue},
we see the generating relations of the binary bracket are exact the same.

The ternary bracket reads
\[ l_3(X,Y,Z)=\sigmanatural\bigl[\smiletaunatural(X),\smilehnatural[\smiletaunatural{(Y)},\smiletaunatural{(Z)}]\bigr] +c.p. \]
Let us examine the ternary bracket on generating elements. For $b_1,b_2\in\sections{B}$ and $\theta\in\sections{A^\vee}$, we have
\begin{eqnarray*}
&&l_3(b_1,b_2,\theta)\\
&=&\sigmanatural\Bigl(\bigl[\smiletaunatural(b_1),\smilehnatural[\smiletaunatural{(b_2)},\smiletaunatural{(\theta)}]
-\bigl[\smiletaunatural(b_2),\smilehnatural[\smiletaunatural{(b_1)},\smiletaunatural{(\theta)}]
+\bigl[\smiletaunatural(\theta),\smilehnatural[\smiletaunatural{(b_1)},\smiletaunatural{(b_2)}]
\bigr] \Bigr)\\
&=&\frac{1}{2}\pairing{\baBtoA{b_1}{b_2}}{\theta}-
\frac{1}{2}\pairing{\baBtoA{b_2}{b_1}}{\theta}(\mbox{ by Equation \eqref{Eqt: contracted3brackets}}) \\
&=& \pairing{\baBtoA{b_1}{b_2}}{\theta}.
\end{eqnarray*}
For the same reasons, one sees that $l_3 $ vanishes if restricted to $\sections{\moduleB}\times\sections{\moduleB}\times\sections{\moduleB}$,
$\sections{\moduleB}\times\sections{\LADAs}\times\sections{\LADAs}$ and $\sections{\LADAs}\times\sections{\LADAs}\times\sections{\LADAs}$.
So, the generating relations of the ternary bracket are exact the same as those in (3) of Proposition \ref{Prop: main1continue}.

The $4$th-bracket reads
\begin{eqnarray*}
&&l_4(X,Y,Z,W)\\
&=&\lambda\sigmanatural\bigl[\smilehnatural[\smiletaunatural{(X)},\smiletaunatural{(Y)}],\smilehnatural[\smiletaunatural{(Z)},\smiletaunatural{(W)}]\bigr] +c.p.\\
&&\qquad+\mu\sigmanatural\bigl[\smiletaunatural(X),\smilehnatural\bigl[\smiletaunatural(Y),\smilehnatural[\smiletaunatural{(Z)},\smiletaunatural{(W)}]\bigr] \bigr] +c.p.
\end{eqnarray*}
where $\lambda,\mu$ are two constants.
By Equations \eqref{Eqt: temp3} and \eqref{Eqt: contracted4brackets2},
we see that $l_4$ vanishes if restricted to generating elements in $\sections{B}$ or $\sections{A^\vee}$.
Hence $l_4$ is trivial.
It can be similarly verified that all higher brackets $l_j$ ($j\geq 5$) are trivial.

This completes the proof.
\end{proof}

\subsection{Third construction: Dirac deformation and proof of main theorems}
\label{melon}

Deformation of Dirac structures has been studied at least 15 years ago by
Severa \cite{arXiv:1707.00265}
and Roytenberg \cite{MR1936572}.
It is well known that the deformation is controlled by an $L_\infty$ algebra, which is canonical up to $L_\infty$ isomorphisms.
In fact, it is a degree $(-1)$ derived Poisson algebra. Let us recall the construction below.

Let $E$ be a Courant algebroid of signature $(n,n)$ over a smooth manifold $M$, and $D\subset E$ a Dirac structure.
Choose a transversal almost Dirac (i.e.\ maximal isotropic) subbundle $C\subset E$ such that $E\cong D\oplus C$.
Identify $C$ with $D^\vee$. Then we have $E\cong D\oplus D^\vee$.

The Courant bracket $\lie{\argument}{\argument}_E$ on $\Gamma (E)$ and the anchor map $\rho_E: E\to T_M$ induce, by restrictions, a skew-symmetric bracket
$[\argument,\argument]_{D^\vee}$ on $\Gamma(D^\vee)$ and an anchor map $\rho_\vee : D^\vee\to T_M$,
\begin{eqnarray*}
&& [\xi,\eta]_{D^\vee}=\pr_{D^\vee}\lie{\xi}{\eta}_E,\quad\forall\xi,\eta\in\Gamma(D^\vee) ,\\
&& \rho_\vee=\rho_E|_{D^\vee}
.\end{eqnarray*}

Let $\phi\in\Gamma(\Lambda^3 D)$ be the section defined by
\[ \phi(\xi,\eta,\zeta)=2\pairing{[\xi,\eta]_E}{\zeta}_E=2\pairing{\pr_{D}[\xi,\eta]_E}{\zeta}_E,\quad\forall\xi,\eta,\zeta\in\Gamma(D^\vee) .\]

Here $\pairing{\argument}{\argument}_E$ is the symmetric metric on $E$.
It can be easily verified that $\phi$ is indeed skew-symmetric.

Unless $C\cong D^\vee$ is again a Dirac structure, in general $\phi$ is non-zero and $\big(D^\vee,[\argument,\argument]_{D^\vee},\rho_\vee\big)$ is not a Lie algebroid.
Instead, $(D,D^\vee)$ forms a quasi-Lie bialgebroid \cite{MR1936572}.

Consider the graded algebra $\Gamma(\Lambda^\bullet D^\vee)
=\oplus_{k=0}\Gamma(\wedge^k D^\vee)[-k]$,
whose degree $n$-part is $\Gamma(\Lambda^n D^\vee)$, for $n\geq 0$. Let \[ \lambda_1=d_D : \Gamma(\Lambda^i D^\vee)\to\Gamma(\Lambda^{i +1}D^\vee) \]
be the Chevalley--Eilenberg differential of the Lie algebroid $D$.

Define a binary bracket
\[ \lambda_2 : \Gamma(\Lambda^i D^\vee)\otimes\Gamma(\Lambda^l D^\vee)\to\Gamma(\Lambda^{i +l-1}D^\vee) \]
by extending, using Leibniz rule (see Equation \eqref{leibniz}, for $n=2$, $k=-1$), 
the relation
\[ \lambda_2 (\xi,\eta)=[\xi,\eta]_{D^\vee}, \quad \lambda_2(\xi, f)=\rho_\vee(\xi)(f), \quad\forall\xi,\eta\in\Gamma(D^\vee), \ f\in C^\infty(M) .\]

Similarly, let
\[ \lambda_3 : \Gamma(\Lambda^i D^\vee)\otimes\Gamma (\Lambda^l D^\vee)\otimes\Gamma(\Lambda^r D^\vee)\to\Gamma(\Lambda^{i +l+r-3} D^\vee) \]
be the ternary bracket extending $\phi$, by Leibniz rule, in each argument ($\lambda_3$ vanishes if one of the argument is a function on $M$).

The following result is due to Severa \cite{arXiv:1707.00265} and
Roytenberg \cite{MR1936572}. Relevant results appeared
in \cite{MR2275207} and more recently in \cite{arXiv:1702.08837}.

\begin{proposition}\label{pro: dirac}
Let $E$ be a Courant algebroid of signature $(n, n)$ over a smooth manifold $M$, and $D\subset E$ a Dirac structure.
Choose a transversal almost Dirac structure $C$. Then
\begin{itemize}
\item $\Gamma(\Lambda^\bullet D^\vee)$, together with $\lambda_1,\lambda_2,\lambda_3$ defined above and $\lambda_l=0, l>3$,
and the wedge product, is a degree $(-1)$ derived Poisson algebra.
\item The underlying $L_\infty$ algebra structure on $\Gamma (\Lambda^\bullet D^\vee)[1]$ controls deformations of the Dirac structures $D\subset E$
in the following sense: the graph $\{X+\omega^b(X)|X\in D\}\subset E$ of an element $\omega\in\Gamma(\Lambda^2 D^\vee)[1]$ is a Dirac structure
if and only if $\omega$ satisfies the Maurer--Cartan equation:
\[ \lambda_1(\omega)+\frac{1}{2}\lambda_2(\omega,\omega)+\frac{1}{6}\lambda_3(\omega,\omega,\omega)=0 .\]
\end{itemize}
\end{proposition}

In fact, such a degree $(-1)$ derived Poisson algebra structure on
$\Gamma (\Lambda^\bullet D^\vee)$ is canonical, up to isomorphisms.
Note that there is a one-one correspondence between almost Dirac structures transversal to $D$ and elements in $\Gamma(\Lambda^2 D)$.
Their relation is established as follows:
\[ \pi\in\sections{\Lambda^2 D} \quad\leftrightarrow\quad C_\pi=\set{\pi^\sharp(\xi)+\xi|\xi\in C\cong D^\vee} .\]

\begin{proposition}[\cite{Severa_private_communication}]
\label{Prop: cannonicalLinfyDiracDual}
Under the same hypothesis as in Proposition \ref{pro: dirac}, assume that $C_\pi$ is another almost Dirac structure transversal to $D$,
which corresponds to an element $\pi\in\Gamma(\Lambda^2 D)$.
Then the $(-1)$ derived Poisson algebra structures on $\Gamma(\Lambda^\bullet D^\vee)$ induced from $C$ and $C_\pi$ are isomorphic.
The isomorphism is given by $\exp{\delta_\pi}$, where $\delta_\pi$ is a coderivation on $\overline{S}(\Gamma(\Lambda^\bullet D^\vee)[2]$ generated by
\begin{eqnarray}\nonumber
&& S^2\big(\Gamma(\Lambda^\bullet D^\vee)\big)\to\Gamma(\Lambda^\bullet D^\vee) \\
&& \xi\otimes\eta\mapsto\contraction{\pi} (\xi)\wedge\eta+\xi\wedge \contraction{\pi}(\eta)-\contraction{\pi}(\xi\wedge\eta) \label{eq: deltapi}
.\end{eqnarray}
\end{proposition}

Note that the bilinear map in Equation \eqref{eq: deltapi} is a biderivation of the graded commutative algebra $\Gamma(\Lambda^\bullet D^\vee)$
with respect to the wedge product. Therefore $\exp{\delta_\pi}$ is indeed compatible with respect to the associative algebra structure,
according to Proposition \ref{prop: morphisms}.

Now consider a Lie pair $(L, A)$. Let $E=L\oplus L^\vee$ be the standard
Courant algebroid \cite{MR1472888},
where the Courant bracket and the anchor map are defined, respectively, by 
\begin{eqnarray*}
[X+\alpha , Y+\beta ]_E &=& [X, Y]_L+(\LieDerivative_X\beta -\contraction{Y} d_L \alpha) ,\\
\rho_E (X+\alpha) &=& \rho_L (X),
\end{eqnarray*}
$\forall X+\alpha,Y+\beta\in\sections{L\oplus L^\vee}$.
The symmetric pairing is:
\[ \pairing{X+\alpha}{Y+\beta}_E=\frac{1}{2}(\pairing{X}{\beta}+\pairing{Y}{\alpha}) .\]

It is a standard result that $D=A\oplus A^\perp$ is a Dirac structure
of $E$ \cite{MR1472888}.
Choose a splitting of the exact sequence \eqref{Eqt: exactseq} so that $L\cong A\oplus B$.
Then $A^{\perp}\cong B^\vee$ and $B\oplus A^\vee$ is a transversal almost Dirac structure of $D$ in $E$.
That is \[ E\cong (A\oplus B^\vee)\oplus (B\oplus A^\vee)\cong D\oplus D^\vee \]
with $D\cong A\oplus B^\vee$ and $D^\vee\cong B\oplus A^\vee$.
Thus $(D, D^\vee)$ is a quasi-Lie bialgebroid, and the induced structure maps on $D^\vee$ are, respectively, given by the following relations:
\begin{itemize}
\item[1)]
The bracket $[\argument,\argument]_{D^\vee}$ reads
\begin{eqnarray*}
[u+\theta,v+\omega]_{D^\vee} &=& \pr_{D^\vee}\bigl([u,v]_L+(\LieDerivative_u \omega-\contraction{v} d_L\theta)\bigr) \\
&=& \pr_{B}[u,v]_L+\pr_{A^\vee}(\LieDerivative_u \omega-\LieDerivative_v \theta),
\end{eqnarray*}
where $u+\theta,v+\omega\in\sections{D^\vee}=\sections{B\oplus A^\vee}$.
\item[2)]
The anchor $\rho_\vee$ is simply
\begin{eqnarray*}
&& \rho_\vee(u+\theta)=\rho_L (u)
.\end{eqnarray*}
\item[3)]
The $3$-form $\phi$ on $D^\vee$ is:
\begin{eqnarray*}
\phi (u_1+\theta_1, u_2+\theta_2, u_3+\theta_3) &=& 2\pairing{\pr_{D}[u_1+\theta_1,u_2+\theta_2]_E}{u_3+\theta_3}_E \\
&=& 2\pairing{\pr_{A}[u_1,u_2]_L+\pr_{B^\vee}(\LieDerivative_{u_1}\theta_2-\LieDerivative_{u_2}\theta_1)}{u_3+\theta_3}_E \\
&=& \pairing{\pr_{A}[{u_1 },{u_2}]_L}{\theta_3}+\pairing{\LieDerivative_{u_1}\theta_2}{u_3}-\pairing{\LieDerivative_{u_2}\theta_1}{u_3} \\
&=& \pairing{\pr_{A}[{u_1 },{u_2}]_L}{\theta_3}+\pairing{\pr_{A}[{u_3 },{u_1}]_L}{\theta_2}+\pairing{\pr_{A}[{u_2 },{u_3}]_L}{\theta_1}
,\end{eqnarray*}
$\forall u_i\in B,\theta_i\in A^\vee$, $i=1,2,3$.
\end{itemize}

The above generating relations determine a
degree $(-1)$ derived Poisson algebra structure on
\[ \Gamma(\Lambda^\bullet D^\vee)\cong\sections{\Lambda^\bullet(B\oplus A^\vee)}\cong
\oplus_{k=0, l=0}\Omega_A^k (\Lambda^l B)[-k-l] .\]
In order to be consistent with the degree convention in
Section \ref{Sec: resultofoutaliepair} and Section \ref{pineapple},
we now redesignate the degrees so that the
subspace $\Omega_A^k(\Lambda^l B)$ is of degree $(k-l)$.
In this way, we obtain the same $(+1)$ derived Poisson algebra
on $\tot \OmegaAwedgeB \Omega_A^k (\Lambda^l B)[-k+l]$
as in Proposition \ref{Prop: main1continue}, since
they have exactly the same generating relations.

In summary, we have proved the following

\begin{proposition}\label{pro: unique}
Let $(L, A)$ be a Lie pair, and $E=L\oplus L^\vee$ the standard Courant
algebroid associated to the Lie algebroid $L$ \cite{MR1472888}.
Consider the Dirac structure $D=A\oplus A^\perp\cong A\oplus B^\vee$. Then
\begin{enumerate}
\item A splitting of the exact sequence \eqref{Eqt: exactseq} determines an almost Dirac structure $B\oplus A^\vee$ transversal to $D$.
\item The degree $(+1)$ derived Poisson algebra $\tot \OmegaAwedgeB$ corresponding
to the transversal almost Dirac structure $B\oplus A^\vee$ as in
Proposition \ref{pro: dirac} coincides with the one as
in Proposition \ref{Prop: main1continue}.
\end{enumerate}
\end{proposition}

\begin{proof}[Proof of Theorem \ref{thm: main1} and Theorem \ref{thm: main2}]

According to Proposition \ref{Prop: cannonicalLinfyDiracDual} and Proposition \ref{pro: unique},
the degree $(+1)$ derived Poisson algebra $\tot \OmegaAwedgeB$ induced from Proposition \ref{pro: dirac}, the one as in
Proposition \ref{Prop: main1continue}, and the one as in Proposition \ref{Prop: FedosovSameGalgebra} all coincide,
and is canonical up to derived Poisson algebra isomorphisms with the linear map being the identity.
As a consequence, the induced degree $(+1)$ Poisson algebra, or a Gerstenhaber algebra,
on the level of cohomology $\hypercohomology (\OmegaAwedgeB,\dAB)$, is indeed canonical.
\end{proof}

\subsection{Examples}

\subsubsection{Matched pairs of Lie algebroids}
Let $L$ be a Lie algebroid, $A$ and $B$ two Lie subalgebroids of $L$ such that $L\cong A\oplus B$ as vector bundles.
Then $L/A\cong B$ is naturally an $A$-module, while $L/B\cong A$ is naturally a $B$-module.
Then $(A,B)$ is said to form a matched pair. Alternatively, one can define a matched pair as follows.

\begin{definition}[\cite{MR1430434, MR1716681, MR1460632}]
Lie algebroids $A$ and $B$ over the same base manifold $M$
are said to form a matched pair if there exists an action $\nabla$ of $A$ on $B$ and an action $\anadelta$ of $B$ on $A$, such that the identities
\begin{gather*}
\lie{\anchor_A(X)}{\anchor_B(Y)} = -\anchor_A\big(\anadelta_Y X\big)+\anchor_B\big(\nabla_X Y\big) , \\
\nabla_X\lie{Y_1}{Y_2} = \lie{\nabla_X Y_1}{Y_2} +\lie{Y_1}{\nabla_X Y_2} +\nabla_{\anadelta_{Y_2} X}Y_1 -\nabla_{\anadelta_{Y_1} X} Y_2 , \\
\anadelta_Y\lie{X_1}{X_2} = \lie{\anadelta_Y X_1}{X_2} +\lie{X_1}{\anadelta_Y X_2} +\anadelta_{\nabla_{X_2} Y} X_1 -\anadelta_{\nabla_{X_1}Y}X_2
\end{gather*}
hold for all $X_1,X_2,X\in\sections{A}$ and $Y_1,Y_2,Y\in\sections{B}$.
Here $\anchor_A$ and $\anchor_B$ denote the anchor maps of $A$ and $B$ respectively.
\end{definition}

Given a matched pair $(A,B)$ of Lie algebroids, there is a Lie algebroid structure
on the direct sum vector bundle $L=A\oplus B$, with anchor \[ X\oplus Y\mapsto\rho_A(X)+\rho_B(Y) \] and the Lie bracket
\begin{equation*}
\lie{X_1\oplus Y_1}{X_2\oplus Y_2} = \big(\lie{X_1}{X_2} + \anadelta_{Y_1}X_2 - \anadelta_{Y_2}X_1 \big)
\oplus \big( \lie{Y_1}{Y_2} + \nabla_{X_1}Y_2 - \nabla_{X_2}Y_1 \big) .
\end{equation*}

Clearly, the pair $(L,A)$ is a Lie pair, and the Bott $A$-connection on $L/A\cong B$ coincides with $\nabla$.
Applying Proposition \ref{Prop: main1continue}, we see that
the ternary bracket vanishes,
and $\tot \Omega^{\bullet}_A (\Lambda^\bullet B)$ is in fact a dgla.
\begin{theorem}
For any given matched pair $(A,B)$ of Lie algebroids, $\tot \Omega^{\bullet}_A(\Lambda^\bullet B)$ admits
a canonical differential Gerstenhaber algebra structure\footnote{This means that the differential $d_A^\nabla$
is compatible with the Gerstenhaber bracket:
\[ d_A^\nabla \binarybracket{X}{Y} = \binarybracket{d_A^\nabla X}{Y} +\minuspower{\abs{X}+1}\binarybracket{X}{d_A^\nabla Y},
\qquad\forall X,Y\in\OmegaAwedgeB .\]},
where the multiplication is the wedge product, and the differential is the Chevalley--Eilenberg differential
\[ d_A^\nabla : \Omega^{\bullet}_A (\Lambda^\bullet B) \to\Omega^{\bullet +1}_A(\Lambda^\bullet B) .\]
Here the $A$-module structure on $\Lambda^\bullet B$ is the natural extension of the $A$-action on $B$.
\end{theorem}

\begin{theorem}\label{Thm: matchedpairHG}
For any given matched pair $(A,B)$ of Lie algebroids, the Chevalley--Eilenberg
hypercohomology
$\mathbb{H}(\Omega^{\bullet}_A(\Lambda^\bullet B),d_A^\nabla)$
admits a canonical Gerstenhaber algebra structure.
\end{theorem}

As an application, consider a complex manifold $X$.
Set $A=T^{0,1}_X$ and $B=T^{1,0}_X$.
Then $(A,B)$ is a matched pair of Lie algebroids over $\CC$,
and its direct sum $A\bowtie B$ is isomorphic, as a Lie algebroid, to $T_X\otimes\CC$.
It is simple to see that $\tot \Omega^{\bullet}_A(\Lambda^\bullet B)=
\tot \Omega^{0,\bullet}_X(T^{\bullet,0}_X)$,
the differential $d_A^\nabla$ is the standard $\bar{\partial}$-operator,
and the Lie bracket $\lceil\argument,\argument\rceil$ is the Schouten bracket. Therefore,
$\tot (\Omega^{0,\bullet}_X(T^{\bullet,0}_X),\wedge,\lceil\argument,\argument\rceil,\bar{\partial})$ is a differential Gerstenhaber algebra.
The corresponding hypercohomology is isomorphic to the sheaf cohomology $\mathbb{H} (X,\Lambda^\bullet \Theta )$ of holomorphic polyvector fields.
The Gerstenhaber algebra structure in Theorem \ref{Thm: matchedpairHG} becomes the standard Gerstenhaber algebra structure on $\mathbb{H} (X,\Lambda^\bullet \Theta )$.

Another example of matched pairs arises from {$\frkg$-manifolds}.
Let $\frkg$ be a Lie algebra, $M$ a $\frkg$-manifold with infinitesimal
action $\phi: \frkg\to\XXa (M)$.
Let $A=M\rtimes\frkg$ be the associated action Lie algebroid and $B=T_M$.
Then $A$ and $B$ form a Lie algebroid matched pair, with mutual actions
\[ \nabla_{a}Y=[\phi(a),Y],\quad\Delta_Y a=0 ,\]
where $a\in\frkg$ is considered as a constant section in $A$, and $Y\in\XXa (M)$.
It is clear that $\tot \Omega^{\bullet}_A (\Lambda^\bullet B)\cong
\tot \Lambda^\bullet\frkg^\vee \otimes\XX^\bullet_{\poly}(M)$,
and $d_A^\nabla$ is the standard Chevalley--Eilenberg differential, and the Lie bracket $\lceil\argument,\argument\rceil$ is the
extension of the Schouten bracket.
Hence $\tot \Lambda^\bullet\frkg^\vee \otimes\XX^\bullet_{\poly}(M)$ is a
differential Gerstenhaber algebra and its hypercohomology
$\mathbb{H}_{\mathrm{CE}}(\frkg,\XX^\bullet_{\poly}{(M)})$ is a Gerstenhaber algebra.
See \cite[Lemma 3.1]{MR3650387}.

\subsubsection{Semisimple Lie algebras}
Let $\frkg$ be a complex semisimple Lie algebra, and $\frkh$ a Cartan subalgebra in $\frkg$.
Let $\frkg=\frkh \oplus \bigoplus_{\alpha\in \Delta} \mfg_\alpha$
be its root decomposition,
where $\Delta\subset \frkh^\vee $ is the root system of $\frkg$.
It is standard that there exist
$h_\alpha\in\frkh$ and $x_\alpha\in\mfg_{\alpha}$,
such that \[ [h_\alpha,x_\alpha]=2 x_\alpha,\qquad [h_\alpha,x_{-\alpha}]=-2 x_{-\alpha},\quad [x_{\alpha},x_{-\alpha}]=h_\alpha ,\]
for all $\alpha\in \Delta^+$.

Consider the Lie pair $(\frkg, \frkh)$.
Applying Theorems \ref{thm: main1} and \ref{thm: main2}, we have the following
\begin{theorem}
Let $\frkg$ be a complex semisimple Lie algebra
with the root decomposition $\frkg=\frkh\oplus\bigoplus_{\alpha\in \Delta}\mfg_\alpha$.
Then $\tot \Lambda^\bullet\frkh^\vee \otimes\Lambda^{\bullet} (\oplus_{\alpha \in \Delta}\frkg_\alpha)$ admits a canonical degree $(+1)$
derived Poisson algebra structure,
where the multiplication is the wedge product, and the
shifted $L_\infty$-brackets are given as follows:
\begin{itemize}
\item[(1)] The unary bracket is the Chevalley--Eilenberg differential
\[ d_\frkh: \Lambda^\bullet \frkh^\vee \otimes\Lambda^{\bullet} (\oplus_{\alpha \in \Delta}\frkg_\alpha)
\to\Lambda^{\bullet+1}\frkh^\vee \otimes\Lambda^{\bullet} (\oplus_{\alpha \in \Delta}\frkg_\alpha) .\]
Here the $\frkh$-module structure on $\Lambda^\bullet (\oplus_{\alpha \in \Delta}\frkg_\alpha)$ is the natural extension of the $\frkh$-action
on $\oplus_{\alpha \in \Delta}\frkg_\alpha$.
\item[(2)] The binary bracket $\binarybracket{\argument}{\argument}_2$
is generated by
\[ \binarybracket{x_\alpha}{x_\beta}_2=c_{\alpha,\beta}x_{\alpha+\beta}, \quad\forall\alpha,\beta\in\Delta,\alpha+\beta\in\Delta .\]
Here $c_{\alpha,\beta}$ is the standard structure constant (see \cite[Chapter VII]{MR499562}).
\item[(3)] The ternary bracket $\trinarybracket{\argument}{\argument}{\argument}_3$ is generated by
\[ \trinarybracket{x_\alpha}{x_{-\alpha}}{\xi}_3=\pairing{h_\alpha}{\xi}, \qquad\forall\xi\in\frkh^\vee ,\]
with all other situations being trivial.
\item[(4)] All higher brackets vanish.
\end{itemize}
\end{theorem}

\begin{theorem}Let $\frkg =\frkh \oplus \bigoplus_{\alpha\in \Delta} \mfg_\alpha$ be a complex semisimple Lie algebra,
the Chevalley--Eilenberg hypercohomology
$\mathbb{H}( \Lambda^\bullet \frkh^\vee \otimes\Lambda^{\bullet}(\oplus_{\alpha \in \Delta}\frkg_\alpha), d_\frkh)$
admits a canonical Gerstenhaber algebra structure.
\end{theorem}

\subsubsection{Transitive Lie algebroids}
Let $(\LADL,\baL{\argument}{\argument},\anchorL)$ be a transitive Lie algebroid over $M$, i.e.\ $\anchorL : \LADL\to T_M$ is surjective.
Let $\mfk=\ker\anchorL$ be the adjoint bundle of $\LADL$, which is a Lie algebra bundle over $M$.
Then $(L,\mfk)$ is a Lie pair. The quotient $\LADL/\mfk \cong T_M$ has the trivial $\mfk $-module structure.

Following Mackenzie \cite{MR2157566},
a connection is a bundle map $\gamma: T_M\to\LADL$ such that $\anchorL\circ\gamma=\id_{T_M}$.
A connection always exists.
The curvature of $\gamma$ is the bundle map $R : \wedge^2 T_M \to\mfk$:
\[ R(X,Y)=\baL{\gamma(X)}{\gamma(Y)}-\gamma[X,Y], \qquad\forall X,Y\in\XXa (M) .\]

There also induces a $T_M$-connection on $\mfk $ defined by
\[ \nabla_X u=\baL{\gamma(X)}{u}, \qquad\forall X\in\XXa (M),
u\in\sections{\mfk } .\]

By choosing such a connection, $L$ can be identified with $\mfk \oplus T_M$ as
a vector bundle over $M$.
The Lie algebroid structure on $\LADL$ can be described as follows:
the anchor is the projection to the second component,
and the Lie bracket is
\[ \baL{(u,X)}{(v,Y)}=( R(X,Y)+[u,v]_{\mfk }+\nabla_Xv-\nabla_Yu,[X,Y]), \quad\forall u,v\in\sections{\mfk }, X,Y\in\XX(M) .\]

Now for the Lie pair $(\LADL,\mfk )$, by Proposition \ref{Prop: main1continue}, Theorems \ref{thm: main1} and \ref{thm: main2}, we have
\begin{theorem}\label{Thm: transitivemain}
Let $L$ be a transitive Lie algebroid over $M$ with anchor $\anchorL$, and $\mfk =\ker\anchorL$.
Then, up to degree $(+1)$ derived Poisson algebra isomorphisms whose first
Taylor coefficient is the identity map,
$\sections{\Lambda^{\bullet}\mfk ^\vee\otimes\XX^\bullet_{\poly}(M)}$
admits a unique degree $(+1)$ derived Poisson algebra structure,
where the multiplication is the wedge product, and shifted $L_\infty$ brackets
are given as follows:
\begin{itemize}
\item[(1)] The unary bracket
\[ d : \sections{\Lambda^{\bullet}\mfk ^\vee\otimes\XX^\bullet_{\poly}(M)}
\to\sections{\Lambda^{\bullet+1}\mfk ^\vee\otimes\XX^\bullet_{\poly}(M)} \]
is equal to $d_{\mfk }\otimes 1$, where $ d_{\mfk }: \sections{\Lambda^{\bullet}\mfk ^\vee}\to \sections{\Lambda^{\bullet+1}\mfk ^\vee}$ is
the Chevalley--Eilenberg differential of $\mfk $.
\item[(2)] The binary bracket is generated by the following relations:
\begin{eqnarray*}
&& \binarybracket{X}{ Y}_2=[X,Y] ,\\
&& \binarybracket{X}{\omega}_2=\nabla_X\omega ,\\
&& \binarybracket{\omega}{\eta}_2=0
,\end{eqnarray*}
$\forall X,Y\in\XXa (M),\omega,\eta\in\sections{\mfk ^\vee}$.
\item[(3)] The ternary bracket is generated by the following relations:
\[ \trinarybracket{X}{Y}{\omega}_3=\pairing{R^\gamma(X,Y)}{\omega}, \quad\forall X,Y\in\XX(M),\omega\in\sections{\mfk ^\vee} ,\]
and otherwise vanishes.
\item[(4)] All the higher brackets vanish.
\end{itemize}
\end{theorem}

\begin{theorem}
Under the same hypothesis as in Theorem \ref{Thm: transitivemain},
the space
$\hypercohomology_{\mathrm{CE}}\big(\mfk , \XX^\bullet_{\poly}(M) \big)$ 
admits a canonical Gerstenhaber algebra structure.
\end{theorem}

\subsubsection{Foliations}
Consider a regular foliation $F$ on a smooth manifold $M$.
Then $(T_M,T_F)$ is a Lie pair, where $T_F$ is the
integrable distribution corresponding to $F$.
Let $N_F=T_M/T_F$ be the associated normal bundle.
Applying Theorems \ref{thm: main1} and \ref{thm: main2}, we have

\begin{theorem}\label{Thm: foliationmain}
Let $F$ be a regular foliation on $M$.
Then, up to degree $(+1)$ derived Poisson algebra isomorphisms whose first Taylor coefficient is the identity map,
$\tot \sections{\Lambda^{\bullet}T_F^\vee\otimes\Lambda^\bullet N_F}$ admits a
canonical degree $(+1)$ derived Poisson algebra structure,
where the multiplication is the wedge product,
and the shifted $L_\infty$ brackets are
given as follows:
\begin{itemize}
\item[(1)] The unary bracket is the leafwise de Rham differential
\[ d_{dR}^{F} : \sections{\Lambda^{\bullet} T_F^\vee\otimes\Lambda^\bullet N_F} \to\sections{\Lambda^{\bullet+1}T_F^\vee\otimes\Lambda^\bullet N_F} .\]
\item[(2)]Choose a
complementary subbundle to $T_F$ in $T_M$ so that $T_M\cong T_F\oplus N_F$.
The binary bracket
\[ \binarybracket{\argument}{\argument}_2 : \sections{\Lambda^{i} T_F^\vee\otimes\Lambda^j N_F} \times \sections{\Lambda^{p} T_F^\vee\otimes\Lambda^l N_F} \to \sections{\Lambda^{i+p} T_F^\vee\otimes\Lambda^{j+l-1} N_F} \]
is generated by the following relations:
\begin{itemize}
\item[a)] $ \binarybracket{u}{v}_2=\pr_{N_F} [{u},{v}]$,\ \ \ $\forall u,v\in\sections{N_F}$;
\item[b)] $\binarybracket{u}{\omega}_2 =\pr_{T_F^\vee}(\LieDerivative_{u}{\omega})$, \ \ \ $\forall u\in \sections{N_F}$, $\omega\in \sections{T_F^\vee}$;
\item[c)] $\binarybracket{u}{f}_2=u(f)$, \ \ \ $\forall u\in \sections{N_F}$, $ f\in \cinf{M}$;
\item[d)] $\binarybracket{\omega_1}{\omega_2}_2=0$,\ \ \ $\forall \omega_1,\omega_2\in\sections{\Lambda^\bullet T_F^\vee}$.
\end{itemize}
\item[(3)] The ternary bracket
\[ \trinarybracket{\argument}{\argument}{\argument}_3 :
\sections{\Lambda^{i} T_F^\vee\otimes\Lambda^j N_F} \times \sections{\Lambda^{p} T_F^\vee\otimes\Lambda^l N_F} \times \sections{\Lambda^{r} T_F^\vee\otimes\Lambda^s N_F} \to\sections{\Lambda^{i+p+r-1} T_F^\vee\otimes\Lambda^{j+l+s-2} N_F} \]
is $\cinf{M}$-linear in each entry and generated by the following relations:
\begin{itemize}
\item[a)] $\trinarybracket{\argument}{\argument}{\argument}_3$
vanishes when being restricted to $\sections{N_F} \times \sections{N_F} \times \sections{N_F}$,
$\sections{N_F} \times \sections{T_F^\vee} \times \sections{T_F^\vee}$,
and $\sections{T_F^\vee} \times \sections{T_F^\vee} \times \sections{T_F^\vee}$;
\item[b)]
$\trinarybracket{u}{v}{\omega}_3=\pairing{\pr_{T_F}[{u},{v}] }{\omega}$, for all $u,v\in \sections{N_F}$ and $\omega\in \sections{T_F^\vee}$.
\end{itemize}
\item[(4)] All the rest of higher brackets vanish.
\end{itemize}
\end{theorem}
\begin{theorem}\label{Thm: distributionGalgebra}
Let $F$ be a regular foliation on $M$.
The leafwise de Rham hypercohomology $\mathbb{H}_{\mathrm{dR}}(F, { \Lambda^\bullet N_F})$
admits a canonical Gerstenhaber algebra structure.
\end{theorem}

\begin{remark}
According to Proposition~\ref{Prop: main1},
a choice of a complementary subbundle to $T_F$ in $T_M$
induces an $L_\infty$ algebroid structure on
$T_F[1]\oplus N_F \to T_F[1]$.
In terms of purely algebraic term, this amounts to saying that the pair
$(\sections{\Lambda^\bullet T_F^\vee},\sections{\Lambda^\bullet T_F^\vee \otimes N_F})$
is an $LR_\infty$ algebra. This result is due to Vitagliano \cite{MR3277952}.
\end{remark}

\appendix

\section{$L_\infty$ algebroids and dg manifolds}

We recall some standard notions and results which we used throughout
the paper. We mainly follow the conventions of Bruce \cite{MR2840338}, Voronov \cite{MR2163405}, and Lada--Markl \cite{MR1327129}.

The following proposition reveals a close relation between
$L_\infty$ algebroids and dg manifolds. See \cite{MR2840338, arXiv:1808.10049}.	

\begin{proposition}\label{Prop: LinfinityalgebroidandQ}
Let $\cL\to \cM$ be a vector bundle of $\ZZ$-graded manifolds.
Then $\cL$ is an $L_\infty$ algebroid if and only if $\cL[1]$ is a dg manifold whose homological vector field $Q$ is tangent to the zero section $\cM \xhookrightarrow{0} \cL[1]$.
\end{proposition}

The rest of the section is devoted to the proof of this proposition, where
the set-up was used in the paper. We first start with the following

\begin{lemma}
\label{Lem: Linftyaldequivdefn}
Let $\cL\to\cM$ be a vector bundle of $\ZZ$-graded manifolds,
and let $k$ be a fixed integer.
An $L_\infty$ algebroid structure on $\cL\to\cM$ is equivalent to an
$L_\infty$ algebra structure on the $\kk$-vector space
$\sections{\cL}\oplus\cinf{\cM}[k]$
with structure maps \[ \barlambda_l: \Lambda^l\big(\sections{\cL}\oplus\cinf{\cM}[k]\big) \to \big(\sections{\cL}\oplus\cinf{\cM}[k]\big)[2-l] \]
satisfying the Leibniz rule
\begin{multline}\label{pbracketwithf2}
\barlambda_l({w_1,w_2,\cdots,w_{l-1},fw_l}) = \barlambda_l({w_1,w_2,\cdots,w_{l-1}, f})w_l \\
+\minuspower{(l+\degree{w_1}+\cdots+\degree{w_{l-1}}) \degree{f}} f \barlambda_l({w_1,w_2,\cdots,w_{l-1},w_l})
,\end{multline}
for all $w_1,\cdots,w_l\in\sections{\cL}\oplus\cinf{\cM}[k]$,
and $f\in\cinf{\cM}$,
and the following conditions:
\begin{enumerate}
\item $\sections{\cL}$ is an $L_\infty$ subalgebra of
$\sections{\cL}\oplus\cinf{\cM}[k]$;
\item $\cinf{\cM}[k]$ is an $L_\infty$ ideal of $\sections{\cL}\oplus\cinf{\cM}[k]$,
i.e.
\ $\barlambda_l(w_1,w_2,\cdots,w_l)\in\cinf{\cM}[k]$
if at least one of the arguments $w_1,w_2,\dots,w_l$ is in $\cinf{\cM}[k]$; and
\item $\cinf{\cM}[k]$ is abelian,
i.e.\ $\barlambda_l(w_1,w_2,\cdots,w_l)=0$ if at least two of the arguments
$w_1,w_2,\dots,w_l$ are in $\cinf{\cM}[k]$.
\end{enumerate}
\end{lemma}
\begin{proof}
Assume that $\cL$ is an $L_\infty$ algebroid with
multi-brackets $(\lambda_l)_{l\geq 1}$ and multi-anchor maps
$(\rho_l)_{l\geq 0}$
as in Definition \ref{Defn: Linfinityalgebroid}.
Define a sequence $(\barlambda_l)_{l\geq 1}$ of $\kk$-multilinear maps
\[ \barlambda_l: \Lambda^l\big(\sections{\cL}\oplus\cinf{\cM}[k]\big) \to \big(\sections{\cL}\oplus\cinf{\cM}[k]\big)[2-l] \]
by the following relations:
\begin{itemize}
\item $\barlambda_l(x_1,\cdots,x_l)=0$ if at least two of the arguments
$x_1,\dots,x_l$ are in $\cinf{\cM}[k]$;
\item $\barlambda_l(a_1,\cdots,a_l)=\lambda_l(a_1,\cdots,a_l)$, for all $a_1,\cdots,a_l\in\sections{\cL}$; and
\item $\barlambda_{l+1}(a_1,\cdots,a_l,f)=\rho_{l}(a_1,\cdots,a_l)f$,
for all $a_1,\cdots,a_l\in\sections{\cL}$ and $f\in\cinf{\cM}[k]$.
\end{itemize}
It is straightforward to verify that $(\barlambda_l)_{l\geq 1}$
satisfy all the required properties.

The converse can be proved by going backwards.
\end{proof}

Given a vector bundle $\cE\xto{\pi}\cM$ of $\ZZ$-graded manifolds,
consider the graded Lie algebra $\diffopoo{\cE}$ of first-order differential operators on $\cE$.
It can be identified to $\XXa(\cE)\oplus C^\infty(\cE)$ in a canonical way.
Since $C^\infty(\cE) \cong \sections{\hat{S}(\cE\dual)}$, the contraction
operator $\contraction{s}$ by
a section $s\in\sections{\cE}$ defines a derivation of $C^\infty(\cE)$,
i.e.\ a vector field on $\cE$.
The inclusion $ \sections{\cE}\oplus C^\infty(\cM)\into\XXa(\cE)\oplus
C^\infty(\cE)$ sending $ s+f$ to $ \contraction{s}+\pi^*(f)$
embeds $\sections{\cE}\oplus C^\infty(\cM)$ into an abelian Lie subalgebra of $\diffopoo{\cE}$.
We proceed to define a projection $P: \diffopoo{\cE}\onto\sections{\cE}\oplus
C^\infty(\cM)$.
Given a vector field $X\in\XXa(\cE)$, consider the composition
\[ \sections{\cE\dual}\into\sections{\hat{S}(\cE\dual)}\cong C^\infty(\cE)\xto{X}C^\infty(\cE)\xto{0^*}C^\infty(\cM) ,\]
where $0^*$ denotes the pullback of functions through the zero section of the
vector bundle $\cM \xhookrightarrow{0} \cL$.
There exists a unique $X^\uparrow\in\sections{\cE}$ such that $\pairing{\xi}{X^\uparrow} = 0^*\big(X(\xi)\big)$, for all $\xi\in\sections{\cE\dual}$.
Define $P: \diffopoo{\cE}\onto\sections{\cE}\oplus C^\infty(\cM)$ by
$P(X+f)=X^\uparrow+0^*(f)$, for all $X\in\XXa(\cE)$ and $f\in C^\infty(\cE)$.
Note that the projection operator $P$ satisfies
\begin{equation}\label{Eqt: Pcondition}
P\big(\lie{x}{y}\big) =
P\big(\lie{ P(x)}{y}+\lie{x}{ P(y)}\big),
\quad \forall x,y\in \diffopoo{\cE}.
\end{equation}

The following lemma is easily verified, and is left to the reader.

\begin{lemma}\label{snake}
For any $Q\in\XX (\cE)$, the following assertions are equivalent.
\begin{enumerate}
\item The vector field $Q$ is tangent to the zero section of $\cE\xto{\pi}\cM$.
\item There exists a unique vector field $\Xi$ on $\cM$ such that the diagram
\[ \begin{tikzcd}
C^\infty(\cE) \arrow[r, "Q"] \arrow[d, "0^*"'] & C^\infty(\cE) \arrow[d, "0^*"] \\
C^\infty(\cM) \arrow[r, "\Xi"'] & C^\infty(\cM)
\end{tikzcd} \]
commutes.
\item The ideal $\ker(0^*)$ of $C^\infty(\cE)$ is $Q$-stable.
\item $Q\in\ker( P)$
\end{enumerate}
\end{lemma}

\begin{proof}[Proof of Proposition~\ref{Prop: LinfinityalgebroidandQ}]
Consider the vector bundle $\cE\xto{\pi}\cM$ of $\ZZ$-graded manifolds, where
$\cE=\cL[1]$.
Assume that $Q$ is a homological vector field on $\cL[1]$ tangent to the zero section of $\cL[1]\xto{\pi}\cM$.
According to Lemma \ref{snake}, we have $Q\in\ker( P)$.

The graded Lie algebra $\mathfrak{A}=\diffopoo{\cL[1]}$, its abelian
Lie subalgebra $\mathfrak{a}= \sections{\cL[1]}\oplus C^\infty(\cM) $,
the projection $ P: \mathfrak{A}\to \mathfrak{a}$ of $\diffopoo{\cL[1]}$ onto $ \sections{\cL[1]}\oplus C^\infty(\cM) $,
and together with the vector field $Q\in\ker( P)$
constitute an $L_\infty[1]$ algebra Voronov data \cite[Theorem 1 and Corollary 1]{MR2163405}.
The multibrackets $(\mu_l)_{l\geq 1}$ are given
by a sequence of derived brackets:
\begin{equation}\label{Eqt: VoronovtoLinftyone}
\mu_l(z_1,z_2,\cdots,z_l) = P \big(\lie{\lie{\lie{\lie{Q}{z_1}}{z_2}}{\cdots}}{z_l}\big)
,\end{equation}
for all $z_1,z_2,\dots,z_l\in\sections{\cL[1]}\oplus C^\infty(\cM)$.

Applying the d\'ecalage isomorphism, we obtain an
$L_\infty$ algebra on $\sections{\cL}\oplus C^\infty(\cM)[-1]$ with multibrackets $(\barlambda_l)_{l\geq 1}$.
The multibrackets $(\mu_l)_{l\geq 1}$ and $(\barlambda_l)_{l\geq 1}$ are related as follows:
\[ \barlambda_l(w_1,w_2,\cdots,w_l) =\minuspower{\star} \mu_l(w_1,w_2,\cdots,w_l) ,\]
where $\star=(l-1)\abs{w_{ 1}}+(l-2)\abs{w_{2}}+\cdots+\abs{w_{ {l-1}}}$
for all homogeneous $w_1,w_2,\dots,w_l\in\sections{\cL}\oplus
C^\infty(\cM)[-1]$.

It is straightforward to verify that the $L_\infty$ algebra structure $(\barlambda_l)_{l\geq 1}$ on $\sections{\cL}\oplus C^\infty(\cM)[-1]$
satisfies the four conditions listed in Lemma \ref{Lem: Linftyaldequivdefn}.
Therefore, $\cL\to\cM$ is an $L_\infty$ algebroid.
Its multi-anchor maps $\rho_l: \Lambda^l {\cL}\to T_\cM$ (with $l\geq 0$)
and multi-brackets $\lambda_l: \Lambda^l\sections{\cL}\to
\sections{\cL}$ (with $l\geq 1$) are defined by the relations:
\begin{gather}
\rho_l(a_1,a_2,\cdots,a_l) f= \minuspower{\flat } 0^*\big( \lie{\lie{\lie{\lie{Q}{\contraction{a_1}}}{\contraction{a_2}}}{\cdots}}{\contraction{a_l}}
(\pi^* f) \big),
\label{Eqt: Qtoanchors}
\end{gather}
where $\flat=l\abs{a_{ 1}}+(l-1)\abs{a_{2}}+\cdots+\abs{a_{ {l}}}$
and
\begin{gather}
\pairing{\lambda_l(a_1,a_2,\cdots,a_l)}{\xi} = \minuspower{\sharp } 0^*\big( \lie{\lie{\lie{\lie{Q}{\contraction{a_1}}}{\contraction{a_2}}}{\cdots}}{\contraction{a_l}} (\xi) \big)
\label{Eqt: Qtobrackets}
\end{gather}
where $\sharp=(l-1)\abs{a_{ 1}}+(l-2)\abs{a_{2}}+\cdots+\abs{a_{ {l-1}}}$
for all $\xi\in\sections{\cL^\vee},a_1,a_2,\dots,a_l\in\sections{\cL}$ and $f\in C^\infty(\cM)$.

Conversely, given an $L_\infty$ algebroid $\cL\to\cM$ with multi-anchors
$(\rho_l)_{l\geq 0}$ and multibrackets $(\lambda_l)_{l\geq 1}$,
one can recover the corresponding homological vector field $Q$ on $\cL[1]$
satisfying the desired properties.

The algebra $C^\infty(\cL[1])$ admits the direct product decomposition
\[ C^\infty(\cL[1]) = \prod_{k=0}^\infty \sections{S^k(\cL\dual[-1])} .\]
We will refer to $\sections{S^k(\cL\dual[-1])}$ as the
weight $k$ component of $C^\infty(\cL[1])$.
Note that $\pi^* C^\infty(\cM)$ (the component of weight $0$) and $\sections{\cL\dual[-1]}$ (the component of weight $1$)
generate the associative algebra $C^\infty(\cL[1])$.
A vector field $Q$ on $\cL[1]$ is necessarily of the form
$Q=\sum_{l=-1}^\infty D_l$,
where $D_l$ is
a derivation on $C^\infty(\cL[1])$ of weight $l$:
\[ D_l: \sections{\hat{S}^{\bullet}(\cL\dual[-1])}\to\sections{\hat{S}^{\bullet+l}(\cL\dual[-1])} .\]

Since $Q$ is tangent to $\cM$, i.e.\ we want $Q\in\ker( P)$,
its weight $(-1)$ component $D_{-1}$ must vanish.
Choose a local coordinate chart $(x^j)_{j\in J}$ on $\cM$; a local frame $(s_k)_{k\in K}$ for $\cL[1]$; and the dual local frame $(\xi^k)_{k\in K}$ for $\cL\dual[-1]$,
the derivation $D_l$ can be written as
\[ D_l=\frac{1}{l!}\sum_{j\in J} \xi^{i_l}\cdots\xi^{i_{2}}\xi^{i_1} \pi^*(\tilde{Q}^j_{i_1,i_2,\cdots,i_l}) \frac{\partial}{\partial x^j}
+ \frac{1}{(l+1)!}\sum_{k\in K} \xi^{i_l}\cdots\xi^{i_1}\xi^{i_0} \pi^*(Q^k_{i_0,i_1,\cdots,i_l}) \frac{\partial}{\partial \xi^k} ,\]
where
\begin{gather*}
\tilde{Q}^j_{i_1,i_2,\cdots,i_l} = \minuspower{\diamond} \rho_l(s_{i_1},s_{i_2},\cdots,s_{i_l}) x_j
\end{gather*}
with $\diamond=(l-1)\abs{s_{i_1}}+(l-2)\abs{s_{i_2}}+\cdots+\abs{s_{i_{l-1}}}$,
and
\[ Q^k_{i_0,i_1,\cdots,i_l} = \minuspower{\dag} \pairing{\lambda_{l+1}(s_{i_0},s_{i_1},\cdots,s_{i_l})}{\xi^k} \]
with $\dag=(l+1)\abs{s_{i_0}}+l\abs{s_{i_1}}+\cdots+\abs{s_{i_{l}}}$.
One checks that $D_l$, and therefore $Q$, is well defined. In summary,
the multi-anchors and multi-brackets of the $L_\infty$ algebroid $\cL\to\cM$
determine through Equations \eqref{Eqt: Qtoanchors} and \eqref{Eqt: Qtobrackets}
a vector field $Q$ of degree $+1$ on $\cL[1]$, which is
tangent to the zero section $\cM$.

The $L_\infty[1]$ algebra structure $(\mu_l)_{l\geq 1}$ on $\sections{\cL[1]}\oplus C^\infty(\cM)$ determined by the $L_\infty$ algebroid $\cL\to\cM$
as per Lemma \ref{Lem: Linftyaldequivdefn} is related to the vector field $Q$ through Equation \eqref{Eqt: VoronovtoLinftyone}.
It follows from the generalized Jacobi identity, Equation \eqref{Eqt: VoronovtoLinftyone}, and repetitive use of Equation \eqref{Eqt: Pcondition} that
\[ 0 = \sum_{p+r=l+1}(\mu_p\circ\mu_r)(z_1,\cdots,z_l) = P \big(\lie{\lie{\lie{\lie{\lie{Q}{Q}}{z_1}}{z_2}}{\cdots}}{z_l}\big) ,\]
for all $z_1,z_2,\dots,z_l\in\sections{\cL[1]}\oplus C^\infty(\cM)$.
Therefore $\lie{Q}{Q}=0$, i.e. $Q$ is homological.
\end{proof}

\section{Shifted Poisson algebras}
\begin{definition}
A degree $k$ Poisson algebra is a $\ZZ$-graded commutative and associative algebra $\mathfrak{R}$ with a degree $k$ Poisson bracket,
denoted by $[\argument,\argument]$ (i.e.\ $[\mathfrak{R}^i,\mathfrak{R}^j]\subset \mathfrak{R}^{i+j+k}$), satisfying 

\begin{enumerate}
\item $[a,b]=-\minuspower{ \abs{a}^{[k]} \abs{b}^{[k]} }[b,a]$,
\item $[a,[b,c]]=[[a,b],c]+\minuspower{ \abs{a}^{[k]} \abs{b}^{[k]}}[b,[a,c]]$,
\item $[a,bc]=[a,b]c+\minuspower{ \abs{a}^{[k]}\abs{b}}b[a,c]$,
\end{enumerate}
for all homogeneous elements $a,b,c\in\mathfrak{R}$.
\end{definition}

Note that Conditions (1) and (2) are equivalent to that $\mathfrak{R}[-k]$ is a
$\ZZ$-graded Lie algebra, while (3) means that the Lie bracket is a biderivation.

Also note that degree $0$ Poisson algebras are usual Poisson algebras, and degree $(+1)$ Poisson algebras are Gerstenhaber algebras.
In the meantime, one can obtain a Poisson algebra of degree $k$ out of a graded Lie algebra as indicated in the following
\begin{proposition}\label{Prop: LtoSLk}
Let $\frkg$ be a graded Lie algebra.
Then the symmetric product $\Sbullet(\frkg[k])$ (similarly $\hat{S} (\frkg[k])$) admits a unique degree $k$ Poisson algebra structure
which extends the original Lie bracket on $\frkg$.
\end{proposition}

\section{Shifted polyvector fields}
\label{Appendix: shiftedpoly}

Let $\cM$ be a $\mathbb{Z}$-graded manifold.
A degree $l$ vector field $X\in \XXa(\cM)$ is a derivation
$\cinf{\cM}\xto{X}\cinf{\cM}$ of degree $l$.
The degree of $X$ is denoted by $\abs{X}=l$.

The commutator in $\XXa (\cM)$ is standard:
\[ [X,Y]=X\circ Y-\minuspower{\abs{X}\abs{Y}}Y\circ X ,\]
for all homogeneous $X,Y\in\XXa (\cM)$.
It is obvious that $\XXa (\cM)$ is a left $\cinf{\cM}$-module, and the pair
$(\XXa (\cM),\cinf{\cM})$ forms a $\ZZ$-graded Lie--Rinehart algebra.

Let $n\in\mathbb{Z}$ be a fixed integer. Following Pridham \cite{MR3653066,arXiv:1804.07622}, let $\XX^0_{\poly}(\cM,n)=\cinf{\cM}$, and for each $m\geq 1$,
\[ \XX^{m}_{\poly}(\cM,n) ={S}^m_{\cinf{\cM}}\bigl(\Gamma (\cM; T_\cM \mbox{$[n+1]$}\bigr) .\]
Elements in $\XX ^{m}_{\poly}(\cM,n) $ are called $n$-shifted $m$-polyvector fields on $\cM$.
Then the space \[ \XX ^{\bullet}_{\poly}(\cM,n) =\oplus_{m\geq 0}\XX ^{m}_{\poly}(\cM,n) \]
is called the \emph{$n$-shifted Schouten--Nijenhuis algebra} of $\cM$.
Its completion is denoted by $\hat{\XX}^\bullet_{\poly}(\cM,n)$.
It is simple to see that
\[ \hat{\XX}^\bullet_{\poly}(\cM,n)\cong \cinf{T^\vee_\cM[-n-1]} ,\]
the space of functions on the $(-n-1)$-shifted cotangent bundle $T^\vee_\cM$.

Since $T_\cM$ is a Lie algebroid, 
$T^\vee_\cM$ is a canonical Poisson manifold, which is
in fact symplectic.
According to Proposition \ref{Thm: LinfinitytotshiftedPoisson},
we have the following

\begin{lemma}\label{Lem: AppendixPoissonNatural}
The space $\XX ^{\bullet}_{\poly}(\cM,n) $ admits a degree $(n+1)$
Poisson algebra structure, similarly $\hat{\XX}^\bullet_{\poly}(\cM,n)$.
\end{lemma}

This degree $(n+1)$ Poisson bracket is also known as the $n$-shifted 
Schouten--Nijenhuis bracket.
When $n=0$, the space of $0$-shifted polyvector fields $\XX ^{\bullet}_{\poly}(\cM,0) $,
coincides with the usual Schouten--Nijenhuis algebra on $\cM$, 
which is simply denoted by $\XX^\bullet_{\poly}(\cM)$.
When $n=-1$, the space of $(-1)$-shifted polyvector fields
$\XX^{\bullet}_{\poly}(\cM,-1)$ is the Poisson algebra $\mbox{Pol}({T^\vee_\cM})$. 
Its completion $\hat{\XX}^\bullet_{\poly}(\cM,-1)\cong \cinf{T^\vee_\cM}$.

Any element in $\XX ^{m}_{\poly}(\cM,n)$ is a finite sum of
homogeneous elements of the form:
\[ \Pi=\bar{X}_1\odot\bar{X}_2\odot\cdots\odot\bar{X}_m ,\]
where $X_i\in\XXa (\cM)$, and $\bar{X}_i\in\XXa (\cM)[n+1]$ denotes
the corresponding element with shifted degree.
The number $\abs{\Pi}=\abs{X_1}+\cdots+\abs{X_m}$ is called the \emph{pure degree} of $X$, whereas $m$ is called the \emph{weight}.
By \[ \totalabs{\Pi}_{n}=\abs{\Pi}-m(n+1) ,\]
we denote the \emph{total degree} of $\Pi$.
The following lemma provides an alternative description of shifted polyvector fields.

\begin{lemma}
\label{Lem: polyasmultiderivative}
A homogeneous $n$-shifted $m$-polyvector field $\Pi$ on $\cM$
is equivalently to a $m$-ary operation of degree $\abs{\Pi}$
(the pure degree of $\Pi$):
\[ \Pi: {(\cinf{\cM})}^{\otimes m}\to\cinf{\cM} \]
satisfying the following properties:
\begin{itemize}
\item[1)] $\Pi$ is symmetric multilinear on $\cinf{\cM}[-n-1]$:
\[ \Pi(f_1,\cdots,f_{i-1},f_i,f_{i+1},f_{i+2},\cdots,f_m) = \minuspower{\abs{{f}_i}^{[n+1]}\abs{{f}_{i+1}}^{[n+1]}}\Pi(f_1,\cdots,f_{i-1}, f_{i+1},f_{i},f_{i+1},\cdots,f_m) ;\]
\item[2)] $\Pi$ is a derivation of degree $\abs{\Pi}$:
\begin{align*}
& \Pi(f_1,\cdots,f_{m-1},f_mf'_m) \\
=& \Pi(f_1,\cdots,f_{m-1},f_m)f'_m +\minuspower{(\abs{\Pi}+\abs{f_1}+\cdots+\abs{f_{m-1}})\abs{f_m}} f_m\Pi(f_1,\cdots,f_{m-1},f'_m)
.\end{align*}
\end{itemize}
\end{lemma}
The proof is omitted as it is completely analogous to the usual unshifted polyvector fields on ordinary smooth manifolds.

Finally, we need a technical lemma
for an explicit formula describing
the $(n+1)$-shifted Poisson bracket in $\XX ^{\bullet}_{\poly}(\cM,n)$.
For any $\Pi\in \XX^p_{\poly}(\cM,n) $ and $\Lambda\in\XX^q_{\poly}(\cM,n)$,
let $\Pi\circ \Lambda$ be the $(p+q-1)$-ary operation
$({\cinf{\mathcal{M}}})^{\otimes{p+q-1}}\to\cinf{\mathcal{M}}$ given by
\begin{eqnarray*}
&& (\Pi\circ \Lambda)(f_1,\cdots,f_{p+q-1}) \\
&=& \sum_{\sigma\in \mathrm{Sh}(p,q-1)}{\epsilon^{[n+1]}}(\sigma)\Pi(\Lambda(f_{\sigma(1)},\cdots,f_{\sigma(q)}),f_{\sigma(q+1)},\cdots,f_{\sigma(p+q-1)})
.\end{eqnarray*}
Here ${\epsilon}^{[n+1]}(\sigma)$ denotes the Koszul sign with respect to
the shifted degrees $\abs{{f}_1}^{[n+1]}$, $\cdots$, $\abs{{f}_{p+q}}^{[n+1]}$.

\begin{lemma}
\label{eq: Recollects}
For any $\Pi\in\XX ^{p}_{\poly}(\cM,n) $ and
$\Lambda\in\XX ^{q}_{\poly}(\cM,n) $, the degree $(n+1)$
Poisson bracket $[\Pi,\Lambda]$ in
$\XX ^{\bullet}_{\poly}(\cM,n)$ as in Lemma \ref{Lem: AppendixPoissonNatural}
coincides with the graded commutator:
\[ [\Pi,\Lambda]=\Pi\circ \Lambda-\minuspower{(\totalabs{\Pi}_{n}+n+1)(\totalabs{\Lambda}_{n}+n+1)} \Lambda\circ \Pi .\]
\end{lemma}

\section*{Acknowledgments}
We would like to thank
Martin Bordemann, Oliver Elchinger,
Camille Laurent-Gengoux,
Pavol \v{S}evera,
Jim Stasheff,
and Luca Vi\-ta\-glia\-no
for fruitful discussions and useful comments.

\bibliography{references}

\def\cprime{$'$}
\providecommand{\bysame}{\leavevmode\hbox to3em{\hrulefill}\thinspace}
\providecommand{\MR}{\relax\ifhmode\unskip\space\fi MR }
\providecommand{\MRhref}[2]{%
  \href{http://www.ams.org/mathscinet-getitem?mr=#1}{#2}
}
\providecommand{\href}[2]{#2}
\begin{thebibliography}{10}

\bibitem{arXiv:1705.02880}
Ruggero {Bandiera}, \emph{{Descent of Deligne-Getzler $\infty$-groupoids}},
  ArXiv e-prints (2017).

\bibitem{BSX:17}
Ruggero {Bandiera}, Mathieu {Stiénon}, and Ping {Xu}, \emph{{Polyvector fields
  and polydifferential operators associated with Lie pairs}}, work in progress.

\bibitem{MR3654355}
Denis Bashkirov and Alexander~A. Voronov, \emph{The {BV} formalism for
  {$L_\infty$}-algebras}, J. Homotopy Relat. Struct. \textbf{12} (2017), no.~2,
  305--327. \MR{3654355}

\bibitem{MR3090103}
G.~Bonavolont\`a and N.~Poncin, \emph{On the category of {L}ie
  {$n$}-algebroids}, J. Geom. Phys. \textbf{73} (2013), 70--90. \MR{3090103}

\bibitem{arXiv:1807.03086}
M.~Bordemann, O.~Elchinger, S.~Gutt, and A.~Makhlouf, \emph{{L-infinity
  Formality check for the Hochschild Complex of Certain Universal Enveloping
  Algebras}}, ArXiv e-prints (2018).

\bibitem{MR0362335}
Raoul Bott, \emph{Lectures on characteristic classes and foliations}, Lectures
  on algebraic and differential topology ({S}econd {L}atin {A}merican {S}chool
  in {M}ath., {M}exico {C}ity, 1971), Springer, Berlin, 1972, Notes by Lawrence
  Conlon, with two appendices by J. Stasheff, pp.~1--94. Lecture Notes in
  Math., Vol. 279. \MR{0362335}

\bibitem{MR0220273}
Ronald Brown, \emph{The twisted {E}ilenberg-{Z}ilber theorem}, Simposio di
  {T}opologia ({M}essina, 1964), Edizioni Oderisi, Gubbio, 1965, pp.~33--37.
  \MR{0220273 (36 \#3339)}

\bibitem{MR2840338}
Andrew~James Bruce, \emph{From {$L_\infty$}-algebroids to higher
  {S}chouten/{P}oisson structures}, Rep. Math. Phys. \textbf{67} (2011), no.~2,
  157--177. \MR{2840338}

\bibitem{MR2304327}
Alberto~S. Cattaneo and Giovanni Felder, \emph{Relative formality theorem and
  quantisation of coisotropic submanifolds}, Adv. Math. \textbf{208} (2007),
  no.~2, 521--548. \MR{2304327}

\bibitem{MR3439229}
Zhuo Chen, Mathieu Sti{\'e}non, and Ping Xu, \emph{From {A}tiyah classes to
  homotopy {L}eibniz algebras}, Comm. Math. Phys. \textbf{341} (2016), no.~1,
  309--349. \MR{3439229}

\bibitem{arXiv:1406.1751}
V.~A. {Dolgushev}, A.~E. {Hoffnung}, and C.~L. {Rogers}, \emph{{What do
  homotopy algebras form?}}, ArXiv e-prints (2014).

\bibitem{arXiv:1407.6735}
V.~A. {Dolgushev} and C.~L. {Rogers}, \emph{{A Version of the Goldman-Millson
  Theorem for Filtered L-infinity Algebras}}, ArXiv e-prints (2014).

\bibitem{MR0056295}
Samuel Eilenberg and Saunders Mac~Lane, \emph{On the groups of {$H(\Pi,n)$}.
  {I}}, Ann. of Math. (2) \textbf{58} (1953), 55--106. \MR{0056295 (15,54b)}

\bibitem{MR2361936}
Domenico Fiorenza and Marco Manetti, \emph{{$L_\infty$} structures on mapping
  cones}, Algebra Number Theory \textbf{1} (2007), no.~3, 301--330.
  \MR{2361936}

\bibitem{arXiv:1702.08837}
M.~{Gualtieri}, M.~{Matviichuk}, and G.~{Scott}, \emph{{Deformation of Dirac
  structures via $L\_\infty$ algebras}}, ArXiv e-prints (2017).

\bibitem{MR0301736}
V.~K. A.~M. Gugenheim, \emph{On the chain-complex of a fibration}, Illinois J.
  Math. \textbf{16} (1972), 398--414. \MR{0301736}

\bibitem{MR1103672}
V.~K. A.~M. Gugenheim, L.~A. Lambe, and J.~D. Stasheff, \emph{Perturbation
  theory in differential homological algebra. {II}}, Illinois J. Math.
  \textbf{35} (1991), no.~3, 357--373. \MR{1103672}

\bibitem{MR2103009}
Johannes Huebschmann, \emph{Higher homotopies and {M}aurer-{C}artan algebras:
  quasi-{L}ie-{R}inehart, {G}erstenhaber, and {B}atalin-{V}ilkovisky algebras},
  The breadth of symplectic and {P}oisson geometry, Progr. Math., vol. 232,
  Birkh\"auser Boston, Boston, MA, 2005, pp.~237--302. \MR{2103009}

\bibitem{MR3584886}
\bysame, \emph{Multi derivation {M}aurer-{C}artan algebras and sh
  {L}ie-{R}inehart algebras}, J. Algebra \textbf{472} (2017), 437--479.
  \MR{3584886}

\bibitem{MR1109665}
Johannes Huebschmann and Tornike Kadeishvili, \emph{Small models for chain
  algebras}, Math. Z. \textbf{207} (1991), no.~2, 245--280. \MR{1109665}

\bibitem{MR1932522}
Johannes Huebschmann and James~D. Stasheff, \emph{Formal solution of the master
  equation via {HPT} and deformation theory}, Forum Math. \textbf{14} (2002),
  no.~6, 847--868. \MR{1932522}

\bibitem{MR499562}
J.E. Humphreys, \emph{Introduction to {L}ie algebras and representation
  theory}, Graduate Texts in Mathematics, vol.~9, Springer-Verlag, New
  York-Berlin, 1978, Second printing, revised. \MR{499562}

\bibitem{MR2275207}
Frank Keller and Stefan Waldmann, \emph{Formal deformations of {D}irac
  structures}, J. Geom. Phys. \textbf{57} (2007), no.~3, 1015--1036.
  \MR{2275207}

\bibitem{MR2757715}
H.~Khudaverdian and Theodore~T. Voronov, \emph{Higher {P}oisson brackets and
  differential forms}, Geometric methods in physics, AIP Conf. Proc., vol.
  1079, Amer. Inst. Phys., Melville, NY, 2008, pp.~203--215. \MR{2757715}

\bibitem{arXiv:1808.10049}
\bysame, \emph{Thick morphisms, higher koszul brackets, and
  {$L_\infty$}-algebroids}, ArXiv Mathematics e-print 1808.10049 (2014).

\bibitem{MR2695305}
L.J. Kjeseth, \emph{B{RST} cohomology and homotopy {L}ie-{R}inehart pairs},
  ProQuest LLC, Ann Arbor, MI, 1996, Thesis (Ph.D.)--The University of North
  Carolina at Chapel Hill. \MR{2695305}

\bibitem{MR1854642}
\bysame, \emph{Homotopy {R}inehart cohomology of homotopy {L}ie-{R}inehart
  pairs}, Homology Homotopy Appl. \textbf{3} (2001), no.~1, 139--163.
  \MR{1854642}

\bibitem{MR1327129}
Tom Lada and Martin Markl, \emph{Strongly homotopy {L}ie algebras}, Comm.
  Algebra \textbf{23} (1995), no.~6, 2147--2161. \MR{1327129}

\bibitem{MR3633027}
H.L. Lang, Y.H. Sheng, and X.M. Xu, \emph{Strong homotopy {L}ie algebras,
  homotopy {P}oisson manifolds and {C}ourant algebroids}, Lett. Math. Phys.
  \textbf{107} (2017), no.~5, 861--885. \MR{3633027}

\bibitem{camillesingularfoliation}
Camille {Laurent-Gengoux}, Sylvain Lavau, and Thomas Strobl, \emph{The
  universal {L}ie {$\infty$}-algebroid over a singular foliation}, 2017, work
  in progress.

\bibitem{MR3650387}
Hsuan-Yi Liao, Mathieu Sti\'enon, and Ping Xu, \emph{Formality theorem for
  {$\mathfrak{g}$}-manifolds}, C. R. Math. Acad. Sci. Paris \textbf{355}
  (2017), no.~5, 582--589. \MR{3650387}

\bibitem{MR1472888}
Zhang-Ju Liu, Alan Weinstein, and Ping Xu, \emph{Manin triples for {L}ie
  bialgebroids}, J. Differential Geom. \textbf{45} (1997), no.~3, 547--574.
  \MR{1472888}

\bibitem{MR1430434}
Jiang-Hua Lu, \emph{Poisson homogeneous spaces and {L}ie algebroids associated
  to {P}oisson actions}, Duke Math. J. \textbf{86} (1997), no.~2, 261--304.
  \MR{1430434}

\bibitem{MR2157566}
Kirill C.~H. Mackenzie, \emph{General theory of {L}ie groupoids and {L}ie
  algebroids}, London Mathematical Society Lecture Note Series, vol. 213,
  Cambridge University Press, Cambridge, 2005. \MR{2157566}

\bibitem{MR1716681}
Kirill C.~H. Mackenzie and Tahar Mokri, \emph{Locally vacant double {L}ie
  groupoids and the integration of matched pairs of {L}ie algebroids}, Geom.
  Dedicata \textbf{77} (1999), no.~3, 317--330. \MR{1716681}

\bibitem{MR2130146}
Marco Manetti, \emph{Lectures on deformations of complex manifolds
  (deformations from differential graded viewpoint)}, Rend. Mat. Appl. (7)
  \textbf{24} (2004), no.~1, 1--183. \MR{2130146}

\bibitem{MR1460632}
Tahar Mokri, \emph{Matched pairs of {L}ie algebroids}, Glasgow Math. J.
  \textbf{39} (1997), no.~2, 167--181. \MR{1460632}

\bibitem{MR2180451}
Yong-Geun Oh and Jae-Suk Park, \emph{Deformations of coisotropic submanifolds
  and strong homotopy {L}ie algebroids}, Invent. Math. \textbf{161} (2005),
  no.~2, 287--360. \MR{2180451}

\bibitem{MR3653066}
J.~P. Pridham, \emph{Shifted {P}oisson and symplectic structures on derived
  {$N$}-stacks}, J. Topol. \textbf{10} (2017), no.~1, 178--210. \MR{3653066}

\bibitem{arXiv:1804.07622}
\bysame, \emph{{An outline of shifted Poisson structures and deformation
  quantisation in derived differential geometry}}, ArXiv e-prints (2018).

\bibitem{MR1782593}
Pedro Real, \emph{Homological perturbation theory and associativity}, Homology
  Homotopy Appl. \textbf{2} (2000), 51--88. \MR{1782593}

\bibitem{MR0154906}
George~S. Rinehart, \emph{Differential forms on general commutative algebras},
  Trans. Amer. Math. Soc. \textbf{108} (1963), 195--222. \MR{0154906 (27
  \#4850)}

\bibitem{MR2699145}
Dmitry Roytenberg, \emph{Courant algebroids, derived brackets and even
  symplectic supermanifolds}, ProQuest LLC, Ann Arbor, MI, 1999, Thesis
  (Ph.D.)--University of California, Berkeley. \MR{2699145}

\bibitem{MR1936572}
\bysame, \emph{Quasi-{L}ie bialgebroids and twisted {P}oisson manifolds}, Lett.
  Math. Phys. \textbf{61} (2002), no.~2, 123--137. \MR{1936572}

\bibitem{MR2966944}
H.~Sati, U.~Schreiber, and J.D. Stasheff, \emph{Twisted differential string and
  fivebrane structures}, Comm. Math. Phys. \textbf{315} (2012), no.~1,
  169--213. \MR{2966944}

\bibitem{MR3631929}
Y.H. Sheng and C.C. Zhu, \emph{Higher extensions of {L}ie algebroids}, Commun.
  Contemp. Math. \textbf{19} (2017), no.~3, 1650034, 41. \MR{3631929}

\bibitem{MR0144348}
Weishu Shih, \emph{Homologie des espaces fibr\'{e}s}, Inst. Hautes \'{E}tudes
  Sci. Publ. Math. (1962), no.~13, 88. \MR{0144348}

\bibitem{arXiv:1605.09656}
Mathieu {Stiénon} and Ping {Xu}, \emph{{Fedosov dg manifolds associated with
  Lie pairs}}, ArXiv e-prints (2016).

\bibitem{Severa_private_communication}
Pavol {{\v S}evera}, private communication.

\bibitem{arXiv:1707.00265}
\bysame, \emph{{Letters to Alan Weinstein about Courant algebroids}}, ArXiv
  e-prints (2017).

\bibitem{MR1480150}
Arkady~Yu. Va{\u\i}ntrob, \emph{Lie algebroids and homological vector fields},
  Uspekhi Mat. Nauk \textbf{52} (1997), no.~2(314), 161--162. \MR{1480150}

\bibitem{MR3277952}
Luca Vitagliano, \emph{On the strong homotopy {L}ie-{R}inehart algebra of a
  foliation}, Commun. Contemp. Math. \textbf{16} (2014), no.~6, 1450007, 49.
  \MR{3277952}

\bibitem{MR3313214}
\bysame, \emph{On the strong homotopy associative algebra of a foliation},
  Commun. Contemp. Math. \textbf{17} (2015), no.~2, 1450026, 34. \MR{3313214}

\bibitem{MR3300319}
\bysame, \emph{Representations of homotopy {L}ie-{R}inehart algebras}, Math.
  Proc. Cambridge Philos. Soc. \textbf{158} (2015), no.~1, 155--191.
  \MR{3300319}

\bibitem{MR2163405}
Theodore~T. Voronov, \emph{Higher derived brackets and homotopy algebras}, J.
  Pure Appl. Algebra \textbf{202} (2005), no.~1-3, 133--153. \MR{2163405}

\bibitem{MR2223157}
\bysame, \emph{Higher derived brackets for arbitrary derivations}, Travaux
  math\'ematiques. {F}asc. {XVI}, Trav. Math., vol.~16, Univ. Luxemb.,
  Luxembourg, 2005, pp.~163--186. \MR{2223157}

\bibitem{arXiv:1411.6720}
\bysame, \emph{Microformal geometry and homotopy algebras}, ArXiv Mathematics
  e-print 1411.6720 (2014).

\bibitem{MR3575558}
\bysame, \emph{``{N}onlinear pullbacks'' of functions and
  {$L_\infty$}-morphisms for homotopy {P}oisson structures}, J. Geom. Phys.
  \textbf{111} (2017), 94--110. \MR{3575558}

\end{thebibliography}
\end{document}